\documentclass[11pt, a4paper]{article}

\usepackage{BA_Titelseite}

\usepackage[T1]{fontenc}
\usepackage[utf8]{inputenc}
\usepackage{amsmath}
\usepackage{amssymb}
\usepackage{amsthm}
\usepackage{colonequals}
\usepackage{mathtools}
\usepackage{hyperref}
\usepackage[inner=2.7cm, outer=2.7cm, top=2.7cm, bottom=2cm]{geometry}
\usepackage[backend=biber, style=alphabetic]{biblatex}
\usepackage{titlesec}
\usepackage{tikz}
\usepackage{pgfplots}
\usepackage{wrapfig}
\usepackage{booktabs}
\usepackage{enumitem}
\usepackage{caption}
\usepackage[x11names]{xcolor}
\usepackage{mathrsfs}

\pgfplotsset{compat=1.18}

\title{Transport by rough vector fields}
\author{Finn Rudolph}
\date{3. Juli 2026}

\geburtsdatum{13. Juli 2005}
\geburtsort{Neustadt a.d. Aisch}
\betreuer{Prof. Dr. Herbert Koch}
\zweitgutachter{Prof. Dr. Juan J. L. Velázquez}
\institut{Mathematisches Institut}
\ausarbeitungstyp{Bachelorarbeit Mathematik}

\newcommand{\N}{\mathbb{N}}
\newcommand{\Z}{\mathbb{Z}}

\newcommand{\R}{\mathbb{R}}

\newcommand{\eps}{\varepsilon}

\DeclarePairedDelimiter{\paren}{\lparen}{\rparen}
\DeclarePairedDelimiter{\braces}{\{}{\}}

\DeclarePairedDelimiter{\abs}{\lvert}{\rvert}
\DeclarePairedDelimiter{\norm}{\lVert}{\rVert}

\newcommand{\del}{\partial}
\newcommand{\loc}{\textnormal{loc}}
\newcommand{\id}{\textnormal{id}}

\DeclareMathOperator{\supp}{supp}
\let\div\undefined
\DeclareMathOperator{\div}{div}

\DeclareMathOperator{\sgn}{sgn}
\DeclareMathOperator{\Lip}{Lip}

\newtheorem{theorem}{Theorem}
\newtheorem{lemma}{Lemma}
\newtheorem{proposition}{Proposition}
\newtheorem{corollary}{Corollary}
\theoremstyle{definition}
\newtheorem{definition}{Definition}
\theoremstyle{remark}
\newtheorem{remark}{Remark}

\numberwithin{equation}{section}

\newcommand{\meas}{m}
\newcommand{\contractint}{\hspace{-0.5em}}
\newcommand{\contractintsmall}{\hspace{-0.3em}}
\newcommand{\iTR}{\int_0^T \contractint \int_{\R^d}}

\setlist[enumerate]{label=\arabic*)}

\titleformat{\paragraph}[runin]{\normalfont\itshape}{\theparagraph}{1em}{}

\titleformat{\subsubsection}[runin]
  {\normalfont\itshape}
  {\normalfont\thesubsubsection}
  {1em}
  {}
  [.]

\titlespacing*{\subsubsection}
  {0pt}
  {1ex plus .2ex}
  {1em}

\makeatletter
\def\cleardoublepage{
  \clearpage
  \if@twoside
    \ifodd\c@page
    \else
      \hbox{}
      \thispagestyle{empty}
      \newpage
    \fi
  \fi
}
\makeatother

\begin{document}
\maketitle

% \newgeometry{twoside, inner=3.1cm, outer=2.3cm, top=2.4cm, bottom=2.3cm}

\section*{Abstract}

The Cauchy-Lipschitz theorem guarantees that the ODE $\dot\gamma = b(t,\gamma)$ is well-posed if 
$b$ is Lipschitz in space, an assumption too restrictive for applications such as fluid mechanics. 
In 1989, DiPerna and Lions extended the theory to Sobolev vector fields with bounded divergence. 
Well-posedness then holds not for individual trajectories but for regular Lagrangian flows: 
flows that solve the ODE for almost every initial point and compress volume by at most a bounded 
factor. Their method rests on renormalized solutions of the associated transport equation.

This thesis is primarily expository: it presents the DiPerna-Lions theory with detailed proofs. 
We treat existence, uniqueness and stability of weak and renormalized solutions of the transport 
equation, deduce existence (for autonomous fields) and uniqueness of regular Lagrangian flows, 
and discuss applications to the two-dimensional Euler equations. Three counterexamples with 
divergence-free fields show that uniqueness can fail without a growth condition at infinity, 
if $b$ has only fractional Sobolev regularity (due to DiPerna and Lions), or for $L^1$ solutions 
when $\nabla b \in L^r$, $r<d$ (the simplest case of a recent construction by Bruè, Colombo and Kumar).

Contributions are (1) smooth approximations of $b$ with uniformly bounded divergence for the existence proof; (2) a direct proof of the known uniqueness of regular Lagrangian flows 
from uniqueness for the continuity equation; (3) a construction of the auxiliary cube-translating field in the example of Bruè, Colombo and Kumar that works for all $d \ge 2$, replacing an extension from two dimensions that is not a vector field on $\R^d$ for $d \ge 3$.

\clearpage

\section*{Zusammenfassung}

Ein klassisches Ergebnis zur Existenz und Eindeutigkeit von Lösungen gewöhnlicher
Differenzialgleichungen ist der Satz von Cauchy-Lipschitz. Wir betrachten
die Differenzialgleichung $\dot \gamma = b(t, \gamma(t))$, wobei $b : [0, T] \times \R^d \to \R^d$
ein Vektorfeld ist und eine Anfangsbedingung $\gamma(0) \in \R^d$ gegeben ist.
Für die Eindeutigkeit von Lösungen ist eine zentrale Voraussetzung des Satzes, dass $b$
Lipschitz-stetig in der räumlichen Variable ist.
In einigen Anwendungen ist diese Regularitätsannahme jedoch nicht erfüllt,
beispielsweise in der Strömungsmechanik.
DiPerna und Lions \cite{DL89} haben den Satz von Cauchy-Lipschitz auf weniger reguläre Vektorfelder
verallgemeinert. Der klassische punktweise Lösungsbegriff wird hierbei durch \emph{Regular Lagrangian flows}
ersetzt.
Ziel dieser Arbeit ist es, die Theorie von DiPerna und Lions mit detaillierten Beweisen vorzustellen,
sowie Anwendungen und Gegenbeispiele zu zeigen.

Die Methode von DiPerna und Lions basiert auf der Analyse einer zugehörigen Transportgleichung.
In Abschnitt §\ref{sec:transport} entwickeln wir die Theorie renormalisierter Lösungen für die
Transportgleichung. Daraus werden in Abschnitt §\ref{sec:rlf} Existenz- und Eindeutigkeitsergebnisse
für die gewöhnliche Differenzialgleichung $\dot \gamma = b(t, \gamma(t))$ hergeleitet.
Abschnitt §\ref{sec:euler} behandelt Anwendungen auf die zweidimensionale Eulergleichung.
Zuletzt diskutieren wir drei Gegenbeispiele in Abschnitt §\ref{sec:counterexamples}.

\clearpage

\tableofcontents

\clearpage

\section{Introduction}

Consider the ordinary differential equation
\begin{align}
	\begin{cases}
		\dot \gamma(t) = b(t, \gamma(t)) & \quad \text{for every $t \in (0, T)$} \\
		\gamma(0) = x^0                  & \quad \text{for some $x^0 \in \R^d$}
	\end{cases} \,.
	\label{eq:ode}
\end{align}
Here, $T > 0$ is a real number and $b : [0, T] \times \R^d\to \R^d$ is a (possibly)
time-dependent vector field. The classical existence and uniqueness result for (\ref{eq:ode})
is the Cauchy-Lipschitz theorem. We denote by $\Lip(\cdot)$ the Lipschitz constant
of a Lipschitz-continuous function. The following is a global version of the theorem:

\begin{theorem}[Cauchy-Lipschitz]
	Let $b : [0, T] \times \R^d \to \R^d$ be measurable. Assume that $b(t)$ is
	Lipschitz-continuous on $\R^d$ for every $t \in [0, T]$ with $\Lip (b(t)) \in L^1([0, T])$.
	Then, for every $x^0 \in \R^d$, there exists a unique absolutely continuous solution
	$\gamma : [0, T] \to \R^d$ to (\ref{eq:ode}).
	\label{theorem:cauchy-lipschitz}
\end{theorem}

A natural question is whether the assumption of Lipschitz-continuity can be weakened.
In \cite{DL89}, DiPerna and Lions extended the above result to vector
fields that are not Lipschitz-continuous, but only belong to a Sobolev space.
However, their results are of a different kind than Theorem \ref{theorem:cauchy-lipschitz}:
While Theorem \ref{theorem:cauchy-lipschitz} provides uniqueness for every initial condition $x^0$
individually, in the more general setting we only get uniqueness by adding some
additional information over all initial conditions.

\subsection{Flows with controlled compressibility}

We will see that there exists a unique \emph{regular Lagrangian flow} $X : [0, T] \times \R^d \to \R^d$
under suitable assumptions on $b$. In the following, $\meas^d$ denotes the $d$-dimensional Lebesgue measure.
If $f : \R^d \to \R^d$ is measurable, we denote by $f_\# \meas^d$ the pushforward measure of $\meas^d$
along $f$, i.e. $f_\#\meas^d(A) = \meas^d(f^{-1}(A))$ for every $\meas^d$-measurable $A \subset \R^d$.

\begin{definition}[Regular Lagrangian Flow]
	Let $b : [0, T] \times \R^d \to \R^d$.
	A measurable function $X : [0, T] \times \R^d \to \R^d$ is
	a regular Lagrangian flow for $b$ if
	\begin{enumerate}[label={F\arabic*)}]
		\item for $\meas^{d}$-almost every $x \in \R^d$, the map $t \mapsto X(t, x)$ is
		      an absolutely continuous solution of (\ref{eq:ode}) with
		      $x^0 = x$,
		      \label{enum:F1}
		\item there exists a $K \ge 0$ such that
		      \begin{align*}
			      X(t, \cdot)_\#\meas^{d} \le K \meas^{d}
			      \hspace{3em} \text{for every $t \in [0, T]$.}
		      \end{align*}
		      The constant $K$ is called the \emph{compressibility constant} of $X$.
		      \label{enum:F2}
	\end{enumerate}
	\label{def:lrf}
\end{definition}

This is a common definition, see e.g. \cite{CL08}, p.~15 or \cite{AT17}, p.~731.
Property \ref{enum:F2} is the additional information over all initial conditions mentioned above.
The following example shows that uniqueness can fail pointwise, but still
there is a unique regular Lagrangian flow.

\begin{remark}[Uniqueness of regular Lagrangian flows]
	We say that regular Lagrangian flows for a vector field $b$ are unique,
	if any two regular Lagrangian flows $X$ and $Y$ satisfy
	$X(\cdot, x) = Y(\cdot, x)$ for almost every $x \in \R^d$.
	\label{remark:rlf-unique}
\end{remark}

\paragraph*{Example.}

Take $d = 1$, $b(t, x) = \sqrt{|x|}$, and $x^0 \le 0$. Then, for every
$a \in [0, \infty]$, the following is a solution to (\ref{eq:ode}).
\begin{align}
	\gamma_{a}(t) =
	\begin{cases}
		-(t/2 - \sqrt{-x^0})^2      & \quad \text{if }t \le 2\sqrt{-x^0}                  \\ 0
		                            & \quad \text{if }2\sqrt{-x^0}< t \le 2\sqrt{-x^0}+ a \\
		((t - a)/2 - \sqrt{-x^0})^2 & \quad \text{if }t > 2\sqrt{-x^0}+ a
	\end{cases}
	\label{eq:gamma-a}
\end{align}
Informally speaking, since $b$ is not Lipschitz-continuous, the solution is
able to reach the point 0 and may stay there indefinitely.
However, this is the only obstruction to uniqueness in this example.
If we require that a solution $\gamma$ satisfies
\begin{align*}
	\meas^1(\{t \in [0, T] : \gamma(t) = 0\}) = 0,
\end{align*}
then the only solution to (\ref{eq:ode}) is $\gamma_a$ with $a = 0$.

Now suppose that $X : [0, T] \times \R \to \R$ is a regular Lagrangian flow for $b$.
By \ref{enum:F2}, $\meas^1 \paren*{X(t, \cdot)^{-1}(0)} = 0$ for every $t \in [0, T]$.
Therefore
\begin{align*}
	\meas^2\big(\{(t, x) \in [0, T] \times \R : X(t, x) = 0\}\big) = 0 \,.
\end{align*}
By Fubini's theorem, this implies that $\meas^1(\{t \in [0, T] : X(t, x) = 0\}) = 0$
for almost every $x \in \R$. Hence, for almost every $x \in \R$, $X(\cdot, x)$
must be equal to $\gamma_a$ from (\ref{eq:gamma-a}) with $a = 0$ and $x^0 = x$.
This shows that $X$ is unique in the sense of Remark \ref{remark:rlf-unique}.

\subsection{Flows in the Lipschitz case}
\label{sec:lipschitz-flows}

For later reference and to motivate \ref{enum:F2}, we collect some properties of the flow map
in the classical setting. We assume that $b(t)$ is Lipschitz-continuous on $\R^d$
and $\Lip(b(t)) \in L^1([0, T])$. Then, by Theorem \ref{theorem:cauchy-lipschitz}, there exists
a map $X : [0, T] \times \R^d \to \R^d$ such that for every $x \in \R^d$, $X(\cdot, x)$ is
the unique absolutely continuous solution to (\ref{eq:ode}) with $x^0 = x$.

The first observation is that for all $x, y \in \R^d$
\begin{align}
	\frac{d}{dt} \abs*{X(t, x) - X(t, y)}
	\le |b(t, X(t, x)) - b(t, X(t, y))|
	\le \Lip(b(t)) |X(t, x) - X(t, y)| \,.
	\label{eq:X-lipschitz}
\end{align}
Hence, by Grönwall's inequality (see \cite[p.~708]{Eva10}),
\begin{align}
	|X(t, x) - X(t, y)| \le e^{\ell(t)} |x - y|, \hspace{2em} \text{where} \hspace{.5em}
	\ell(t) \colonequals \int_0^t \Lip b(s) \,ds \,.
	\label{eq:one-sided-lipschitz}
\end{align}
By reversing time, we can also show a similar reverse inequality. Let $t_0 \in [0, T]$ and
$Y^{t_0} : [0, t_0] \times \R^d \to \R^d$ be the flow map of $-b(t_0 - t, x)$. Since
\begin{align*}
	\frac d{dt} X(t_0 - t, x) = -b(t_0 - t, X(t_0 - t, x)) \quad \text{and} \quad
	X(t_0 - t, x)|_{t = 0} = X(t_0, x),
\end{align*}
both $X(t_0 - t, x)$ and $Y^{t_0}(t, X(t_0, x))$ are solutions to (\ref{eq:ode})
with the vector field $-b(t_0 - t, x)$ and initial condition $X(t_0, x)$. By uniqueness of
solutions to (\ref{eq:ode}) in the Lipschitz case, $Y^{t_0}(t_0, X(t_0, x)) = x$
for every $x \in \R^d$. Hence $X(t_0) : \R^d \to \R^d$ is invertible. Moreover, by
applying the computation in (\ref{eq:X-lipschitz}) to $Y^{t_0}$ for every $t_0 \in [0, T]$,
\begin{align}
	e^{-\ell(t)} |x - y| \le |X(t, x) - X(t, y)| \le e^{\ell(t)} |x - y|
	\hspace{3em} \text{for all $x, y \in \R^d$.}
	\label{eq:two-sided-lipschitz}
\end{align}
The lower bound in (\ref{eq:two-sided-lipschitz}) shows that \ref{enum:F2} holds with compressibility constant
\begin{align*}
	K = \exp\big(d\norm*{\Lip b(t)}_{L^1([0, T])}\big) \,.
\end{align*}

The compressibility constant of $X$ can be improved to depend only on the divergence of $b$.
For this, we assume in addition that $b \in L^1_\loc([0, T] \times \R^d)$.
We claim that the measure $\mu(t) \colonequals X(t)_\#\meas^d$ satisfies the following
equation in the distributional sense:
\begin{align}
	\frac{d\mu}{dt} + \div(b\mu) = 0 \qquad \text{on $(0, T) \times \R^d$,}
	\qquad \mu|_{t = 0} = \meas^d \qquad \text{on $\R^d$}
	\label{eq:mu-continuity}
\end{align}
To prove this, let $\phi \in C^\infty_c([0, T) \times \R^d)$. Then, by the chain rule
\begin{align*}
	\frac{d}{dt}\big(\phi(t, X(t, x))\big)
	= \del_t \phi(t, X(t, x)) +  b(t, X(t, x)) \cdot \nabla \phi(t, X(t, x))
	\qquad \text{for a.e. $t \in (0, T)$.}
\end{align*}
Integrating over $[0, T] \times \R^d$ yields
\begin{align*}
	-\int_{\R^d} \phi(0, x) \,dx
	= \iTR \del_t \phi(t, X(t, x)) \,dx \,dt
	+ \iTR b(t, X(t, x)) \cdot \nabla \phi(t, X(t, x)) \,dx \,dt
\end{align*}
By (\ref{eq:two-sided-lipschitz}), these integrals are well defined.
By definition of the distributions in (\ref{eq:mu-continuity}),
\begin{align*}
	\frac{d\mu}{dt}[\phi] = -\iTR \del_t \phi(t, X(t, x)) \,dx \,dt
\end{align*}
and
\begin{align*}
	\div\paren*{b\mu}[\phi] = -\mu[b \cdot \nabla \phi]
	= -\iTR b(t, X(t, x)) \cdot \nabla \phi(t, X(t, x)) \,dx \,dt \,.
\end{align*}
Therefore, (\ref{eq:mu-continuity}) holds.

Next, we look at the density of $\mu(t)$ with respect to $\meas^d$. By (\ref{eq:two-sided-lipschitz}),
$\mu(t)$ is absolutely continuous and $\mu(t) = h(t, \cdot) \, d\meas^d$ for some
$h\in L^\infty([0, T] \times \R^d)$. Then also
\begin{align}
	\frac{dh}{dt} + \div(bh) = 0 \quad \text{on $(0, T) \times \R^d$,}
	\qquad h|_{t = 0} = 1 \quad \text{on $\R^d$}
	\label{eq:h-continuity}
\end{align}
in the weak sense, since $h$ is the same distribution as $\mu$. Now, we use the fact
that weak solutions of the continuity equation (\ref{eq:h-continuity}) in $L^\infty([0, T] \times \R^d)$
are unique. This is not clear from the discussion so far, but it will be a consequence of the results
in section §\ref{sec:uniqueness}.
Therefore, $h$ satisfies the following ODE along the characteristics of $b$:
\begin{align*}
	\frac{d}{dt} \big(h(t, X(t, x))\big)
	= -h(t, X(t, x)) \div b(t, X(t, x))
\end{align*}
The assumption $\Lip (b(t)) \in L^1([0, T])$ implies $\div b \in L^1(0, T; L^\infty(\R^d))$.
Thus, by Grönwall's inequality,
\begin{align*}
	e^{-D(t)} \le h(t, x) \le e^{D(t)}, \hspace{2em} \text{where} \hspace{.5em}
	D(t) \colonequals \int_0^t \norm*{\div b(s)}_{L^\infty(\R^d)} \,ds \,.
\end{align*}
By definition of $h$, this implies
\begin{align}
	e^{-D(t)}m^d \le X(t)_\#\meas^d \le e^{D(t)}m^d
	\hspace{3em} \text{for all $t \in [0, T]$.}
	\label{eq:two-sided-compressibility}
\end{align}
Thus, the divergence of $b$ controls the rate of contraction and expansion of the flow $X$.

\begin{remark}[Local flows]
	Corresponding results hold if $b$ belongs only to $L^1(0, T; W^{1, \infty}_\loc(\R^d))$.
	In this case, by the local version of the
	Cauchy-Lipschitz theorem, for every compact set $K \subset \R^d$, there exists
	a $0 < \tilde T \le T$ and a map $X : [0, \tilde T] \times K \to \R^d$ such that
	for all $x \in K$, $X(\cdot, x)$ is the unique absolutely continuous solution to
	(\ref{eq:ode}) with $x^0 = x$.

	The analogue of (\ref{eq:two-sided-lipschitz}) is that for all $x, y \in K$ and $t \in [0, \tilde T]$
	\begin{align}
		e^{-\ell(t)} |x - y| \le |X(t, x) - X(t, y)| \le e^{\ell(t)} |x - y|,
		\hspace{3em} \ell(t) \colonequals \int_0^t \Lip \big(b(s)|_{K'}\big) \,ds \,.
		\label{eq:two-sided-lipschitz-local}
	\end{align}
	Here, $K' \subset \R^d$ is a compact set, for example $K' = X([0, \tilde T] \times K)$.
	\label{remark:local-flows}
\end{remark}

\section{The transport equation}
\label{sec:transport}

The ODE (\ref{eq:ode}) is closely related to the following linear transport equation.
\begin{align}
	\begin{cases}
		\del_tu - b \cdot \nabla u + cu = 0
		 & \qquad \text{in $(0,T) \times \R^d$} \\ u|_{t = 0}= u^0
		 & \qquad \text{in $\R^d$}
	\end{cases}
	\label{eq:transport}
\end{align}
Equations (\ref{eq:ode}) and (\ref{eq:transport}) are connected via the theory of characteristics,
which will be explained in the proof of Proposition \ref{prop:existence}.
We consider solutions $u \in L^\infty(0, T; L^p_\loc(\R^d))$ for some $1 \le p \le \infty$
and understand (\ref{eq:transport}) in the distributional sense, that is,
\begin{align}
	-\iTR u \, \del_t\phi \,dx \,dt
	- \int_{\R^d} u^0\phi(0, x) \,dx
	+ \iTR u (b \cdot \nabla \phi + (\div b + c)\phi) \, dx \,dt = 0
	\label{eq:transport-weak}
\end{align}
for every $\phi \in C^\infty_c([0, T) \times \R^d)$. For (\ref{eq:transport-weak})
to make sense, we assume $u^0 \in L^1_\loc(\R^d)$, $b \in L^1(0, T; L^q_\loc(\R^d; \R^d))$ and
$c, \, \div b \in L^1(0, T; L^q_\loc(\R^d))$, where $1/p + 1/q = 1$.

We will use the following notation frequently: We say that $u$ is a solution to (\ref{eq:transport})
with $(\tilde b, \tilde c, \tilde u^0)$, if $u$ satisfies (\ref{eq:transport-weak}), where
$b, c$ and $u^0$ are replaced by $\tilde b, \tilde c$ and $\tilde u^0$, respectively.

\subsection{Existence}
\label{sec:existence}

Before proving the existence of solutions to (\ref{eq:transport}),
let us introduce some notation.
We denote by $C^1_b(\R^d; \R^m)$ the space of $C^1$-functions
for which the norm
\begin{align*}
	\norm*{f}_{C^1_b(\R^d; \R^m)} \colonequals \norm*{f}_{L^\infty(\R^d; \R^m)}
	+ \norm*{Df}_{L^\infty(\R^d; \R^{m \times d})}
\end{align*}
is finite. The norm on $\R^{m \times d}$ is the operator norm of linear maps $\R^d \to \R^m$.
We write $C^1_c(\R^d; \R^m)$ for the subspace of compactly supported functions,
equipped with the same norm.

For $x \in \R^d$ and $R > 0$, we denote by $B_R(x) \colonequals \{y \in \R^d : |x - y| < R\}$
the ball of radius $R$ centered at $x$. We sometimes abbreviate $B_R(0)$ by $B_R$.
We also sometimes abbreviate $L^\gamma(0, T; L^p(\R^d))$ by $L^\gamma_t L^p_x$
and $C([0, T]; L^p(\R^d))$ by $C_t L^p_x$ $(1 \le \gamma,p \le \infty)$.

In the proof of existence of weak solutions to the transport equation (\ref{eq:transport}),
we will approximate $b$ by a sequence of vector fields $b_n \in L^1(0, T; C^1_b(\R^d; \R^d))$.
The next lemma shows that this is possible while controlling the divergence of $b_n$
uniformly.

\begin{lemma}
	Let $1 \le q < \infty$ and $b \in L^1(0,T; L^q_\loc(\R^d; \R^d))$. If $\div b \in L^1(0, T; L^\infty(\R^d))$,
	then there exists a sequence of vector fields $b_n \in L^1(0, T; C^1_b(\R^d; \R^d))$ such that
	\begin{align}
		b_n \to b, \ \, \div b_n \to \div b \qquad \text{in $L^1(0, T; L^q_\loc(\R^d))$ as $n \to \infty$,}
		\label{eq:bn-convergence}
	\end{align}
	and
	\begin{align}
		\norm*{\div b_n(t)}_{L^\infty_x(\R^d)} \le C \norm*{\div b(t)}_{L^\infty_x(\R^d)}
		\qquad \text{for a.e. $t \in [0, T]$,}
		\label{eq:div-bn-bounded}
	\end{align}
	for some constant $C$ independent of $b$.
	\label{lemma:bn-seq}
\end{lemma}

\begin{proof}
	Let $\rho_\eps$ be regularizing kernels defined by
	\begin{equation}
		\begin{split}
			\begin{gathered}
				\rho_\eps(x) \colonequals \eps^{-d} \rho(x/\eps) \quad
				\text{for $\eps > 0$}, \qquad \text{for some } \rho \in C^\infty_c(\R^d),
				\
				\rho \ge 0, \ \int_{\R^d} \rho \,dx = 1 \,.
			\end{gathered}
		\end{split}
		\label{eq:rho-eps}
	\end{equation}
	Also, let $\eta_R$ be smooth cutoff functions defined by
	\begin{equation}
		\begin{split}
			\begin{gathered}
				\eta_R(x) \colonequals \eta(|x|/R) \quad \text{for $R > 0$}, \\
				\addlinespace
				\text{for some } \eta \in C^\infty_c([0, \infty)), \
				\eta' \le 0,
				\ \eta|_{[0, 5/4]} = 1, \ \supp \eta \subset [0, 2) \,.
			\end{gathered}
		\end{split}
		\label{eq:eta-R}
	\end{equation}
	We set $\tilde b_n \colonequals (b * \rho_{1/n})\eta_n$.
	The symbol ``$*$'' means convolution only in the spatial variable.
	This will always be the convention.
	Since $q < \infty$, by standard properties of convolutions
	\begin{align}
		\tilde b_n \to b, \ \, \div \tilde b_n \to \div b \qquad
		\text{in $L^1(0, T; L^q_\loc(\R^d))$ as $n \to \infty$.}
		\label{eq:tilde-bn-convergence}
	\end{align}
	However,
	\begin{align}
		\div \tilde b_n = (\div b * \rho_{1/n})\eta_n + (b * \rho_{1/n}) \cdot \nabla \eta_n \,.
		\label{eq:div-bn}
	\end{align}
	Due to the second term, (\ref{eq:div-bn-bounded}) need not hold for $\tilde b_n$.
	In fact $\div \tilde b_n(t)$ need not even be bounded in $L^\infty(\R^d)$ for any $t$.
	Thus, we have to modify $\tilde b_n$.

	Let
	\begin{align*}
		I_n(t) \colonequals
		\int_{\R^d} (b * \rho_{1/n}) \cdot \nabla \eta_n \,dx \,.
	\end{align*}
	Also, let $\zeta_n \in C^\infty_c(B_{2n} \setminus \overline{B_n})$ such that $\int \zeta_n \,dx = 1$
	and $\norm*{\zeta_n}_{L^\infty} \le 2/\meas^d(B_{2n} \setminus B_n)$.
	We set
	\begin{align*}
		f_n(t, x) \colonequals (b * \rho_{1/n}) \cdot \nabla \eta_n - I_n(t)\zeta_n(x) \,.
	\end{align*}
	Then $f_n \in L^1(0, T; L^\infty(\R^d))$ and $f_n(t, \cdot) \in C^\infty_c(B_{2n} \setminus \overline{B_n})$
	for a.e. $t \in [0, T]$. Furthermore
	\begin{align*}
		\int_{B_{2n} \setminus B_n} f_n(t, x) \,dx = 0 \hspace{4em}
		\text{for a.e. $t \in [0, T]$.}
	\end{align*}

	Let $L^2_0(B_{2n} \setminus B_n)$ be the subspace of functions $f \in L^2(B_{2n} \setminus B_n)$
	such that $\int f \,dx = 0$. If $d \ge 2$, by Theorem III.3.1 of \cite{Gal11}, there exists a
	bounded linear operator
	\begin{align*}
		\mathcal B : L_0^2(B_{2n} \setminus B_n) \to W^{1, 2}_0(B_{2n} \setminus B_n; \R^d),
	\end{align*}
	such that $\div \mathcal B(f) = f$ for every $f \in L_0^2(B_{2n} \setminus B_n)$.
	Moreover, by Theorem III.3.3 of \cite{Gal11},
	$\mathcal B(f) \in C^\infty_c(B_{2n} \setminus \overline{B_n}; \R^d)$
	if $f \in C^\infty_c(B_{2n} \setminus \overline {B_n})$.

	We define $v_n \in L^1(0, T; W^{1, 2}_0(B_{2n} \setminus B_n; \R^d))$ by $v_n(t) = \mathcal B(f_n(t))$.
	Since $\mathcal B$ is bounded and linear, $v_n$ is measurable on $[0, T] \times \R^d$.
	Also
	\begin{align*}
		v_n(t, \cdot) \in C^\infty_c(B_{2n} \setminus \overline{B_n}; \R^d), \quad
		\div v_n(t) = f_n(t) \qquad \text{for a.e. $t \in [0, T]$.}
	\end{align*}
	Let $b_n \colonequals \tilde b_n - v_n$. Now
	\begin{align*}
		\div b_n = (\div b * \rho_{1/n}) \eta_n + I_n(t) \zeta_n(x) \,.
	\end{align*}
	By Hölder's inequality,
	\begin{align*}
		|I_n(t)| = \abs*{\int_{\R^d} (\div b(t) * \rho_{1/n}) \eta_n \,dx}
		\le \meas^d(B_{2n}) \norm*{\div b(t)}_{L^\infty(\R^d)} \,.
	\end{align*}
	Thus
	\begin{align*}
		\norm*{\div b_n(t)}_{L^\infty_x} \le \norm*{\div b}_{L^\infty_x}
		+ \frac {2\meas^d(B_{2n})}{\meas^d(B_{2n} \setminus B_n)} \norm*{\div b}_{L^\infty_x}
		\le 5 \norm*{\div b}_{L^\infty_x} \,,
	\end{align*}
	so (\ref{eq:div-bn-bounded}) holds. Since
	$\supp v_n \subset [0, T] \times B_{2n} \setminus \overline{B_n}$,
	(\ref{eq:bn-convergence}) follows from (\ref{eq:tilde-bn-convergence}).

	If $d = 1$, then $\del_x b = \div b \in L^\infty(\R)$ for almost every $t \in [0, T]$.
	Hence, $b(t, \cdot)$ is Lipschitz continuous (after modifying $b$ on a
	$\meas^2$-negligible set).
	Because $b \in L^1(0, T; L^1([-1, 1]))$, the Lipschitz continuity in space implies
	$b(t, 0) \in L^1((0, T))$. Thus, we may assume without loss of
	generality that $b(t, 0) = 0$ for a.e. $t \in [0, T]$.\footnote{
		This reduction is the only reason why the lemma states $b_n \in L^1(0, T; C^1_b)$
		instead of $b_n \in L^1(0, T; C^1_c)$.
	}
	In this case, $|b(t, x)| \le |x| \norm*{\div b(t)}_{L^\infty_x}$. Since
	$\supp \nabla \eta_n \subset B_{2n} \setminus B_n$ and
	$\norm*{\nabla \eta_n}_{L^\infty} \le \norm*{\nabla \eta}_{L^\infty}/n$,
	\begin{align*}
		|(b(t) * \rho_{1/n})(x) \cdot \nabla \eta_n(x)|
		 & \le |x| \norm*{\div b(t)}_{L^\infty_x(\R)} |\nabla \eta_n(x)|                   \\
		 & \le 2 \norm*{\nabla \eta}_{L^\infty(\R)} \norm*{\div b(t)}_{L^\infty_x(\R)} \,.
	\end{align*}
	Thus, by (\ref{eq:div-bn})
	\begin{align*}
		\big\|{\div \tilde b_n(t)}\big\|_{L^\infty_x}
		\le \paren*{1 + 2 \norm*{\nabla \eta}_{L^\infty(\R)}} \big\|{\div b(t)}\big\|_{L^\infty_x}
	\end{align*}
	for a.e. $t \in [0, T]$. We set $b_n = \tilde b_n$.
\end{proof}

\begin{proposition}
	Let $1 < p \le \infty, \ 1/p + 1/q = 1$. Let $u^0\in L^p(\R^d)$ and
	\begin{gather*}
		b \in L^1(0, T; L^q_\loc(\R^d; \R^d)), \qquad
		c, \, \div b \in L^1(0, T; L^\infty(\R^d)) \,.
	\end{gather*}
	Then there exists a weak solution
	$u \in L^\infty(0, T; L^p(\R^d))$ to (\ref{eq:transport})
	for the initial data $u^0$.
	\label{prop:existence}
\end{proposition}

\begin{proof}
	This proof follows that of Proposition II.1 in \cite{DL89}.

	\paragraph{Step 1: Existence in the smooth case.}

	For now, assume in addition to the above assumptions that $b \in L^1(0, T; C^1_b(\R^d; \R^d))$,
	$c \in L^1(0, T; C^1_b(\R^d))$ and $u^0\in C^1_{c}(\R^d)$.
	By Theorem \ref{theorem:cauchy-lipschitz}, there exists a unique
	$X : [0, T] \times \R^d \to \R^d$ satisfying $\del_t X(t, x) = -b(t, X(t, x))$ and $X(0, \cdot) = \id_{\R^d}$.
	Moreover $X(t, \cdot) \in C^1(\R^d)$ for every $t \in [0, T]$ by Proposition \ref{prop:smooth-dependence}.
	We define
	\begin{align}
		v(t, x) \colonequals u^0(x) \exp\paren*{
			-\!\int_0^t c(s, X(s, x)) \,ds
		} \,.
		\label{eq:u-along-char-def}
	\end{align}
	For every $x \in \R^d$, $v(\cdot, x)$ is an absolutely continuous solution of
	\begin{align}
		\del_t v(t, x) = -c(t, X(t, x))\, v(t, x),
		\qquad v(0, x) = u^0(x) \,.
		\label{eq:characteristic-ode}
	\end{align}
	Because $c$ need not be continuous in $t$, in general $v(\cdot, x) \notin C^1([0, T])$.

	Next, we note that by Proposition \ref{prop:smooth-dependence},
	$DX(t, x)$ is absolutely continuous in $t$ and satisfies
	\begin{align}
		\del_t DX(t, x) = -Db(t, X(t, x)) DX(t, x)
		\qquad \text{for a.e. $t \in (0, T)$.}
		\label{eq:Db-evo}
	\end{align}
	Since $Db \in L^1(0, T; C_b(\R^d))$, (\ref{eq:u-along-char-def}) and (\ref{eq:Db-evo}) imply
	$v \in L^\infty(0, T; C^1_c(\R^d))$.
	Also, note that $\supp v(t, \cdot) \subset \supp u^0$ for every $t \in [0, T]$.

	By (\ref{eq:Db-evo}) and the formula for the derivative of the determinant,
	\begin{align*}
		\del_t \big({\det DX(t, x)}\big) = -\div b(t, X(t, x)) \det DX(t, x)
		\qquad \text{for a.e. $t \in (0, T)$.}
	\end{align*}
	Since $\div b \in L^1(0,T ; C_b(\R^d))$, it follows that
	$X(t, \cdot) : \R^d \to \R^d$ is a diffeomorphism for every $t \in [0, T]$.
	We define $u(t, x) \colonequals v(t, X(t, \cdot)^{-1}(x))$.
	Since $X(t, \cdot)$ and $X(t, \cdot)^{-1}$ are Lipschitz continuous uniformly in time,
	$u \in L^\infty(0, T; C^1_c(\R^d))$ and $\supp u$ is compact.
	Let $\phi \in C^\infty_c([0, T) \times \R^d)$ and $t \in [0, T]$. Then
	\begin{align}
		\int_{\R^d} u(t, x) \phi(t, x) \,dx = \int_{\R^d} v(t, x) \phi(t, X(t, x)) \det DX(t, x) \,dx
		\label{eq:uv-cov}
	\end{align}
	by changing variables $x = X(t, x')$.
	Since each of the three factors on the right-hand side in (\ref{eq:uv-cov}) is absolutely continuous
	with respect to $t$, for almost every $t \in [0, T]$
	\begin{align*}
		\frac d {dt}
		 & \int_{\R^d} v(t, x) \phi(t, X(t, x)) \det DX(t, x) \,dx                                   \\
		 & = \int_{\R^d} \Big({-}c(t, X(t, x))v(t, x) \phi(t, X(t, x)) \det DX(t, x)                 \\
		 & \hspace{3.4em} +v(t, x)\big(\del_t \phi(t, X(t, x))
		- \nabla \phi(t, X(t, x)) \cdot b(t, X(t, x))\big) \det DX(t, x)                             \\
		\addlinespace
		 & \hspace{3.4em} - v(t, x) \phi(t, X(t, x)) \div b(t, X(t, x)) \det DX(t, x) \Big) \,dx \,.
	\end{align*}
	By changing variables again, we get
	\begin{align*}
		 & \frac d {dt} \int_{\R^d} u(t, x) \phi(t, x) \,dx \\
		 & \quad = \int_{\R^d} - \phi(t, x)u(t, x)(c(t, x)
		+ \div b(t, x)) - u(t, x) b(t, x) \cdot \nabla \phi(t, x)
		+ u(t, x) \del_t \phi(t, x) \,dx \,.
	\end{align*}
	Thus, $u$ is a weak solution to the transport equation.

	We now derive $L^p$-estimates of $u$ in terms of $u^0$.
	In the case $p = \infty$, by (\ref{eq:u-along-char-def})
	\begin{align}
		\|u(t, \cdot)\|_{L^\infty(\R^d)}
		\le \exp\paren*{\|c\|_{L^1(0, T; L^\infty(\R^d))}}\|u^0\|_{L^\infty(\R^d)}\,.
		\label{eq:u-Linfty-estimate}
	\end{align}
	In the case $1 < p < \infty$, we note that $u$ is weakly differentiable in time,
	since $u$ satisfies the transport equation and $u \in L^\infty(0, T; C^1_c(\R^d))$. By the
	chain rule,
	\begin{align}
		\del_t |u|^p - b \cdot \nabla |u|^p+ pc|u|^p= 0 \,.
		\label{eq:ddt-norm}
	\end{align}
	Integrating this equation over $\R^d$ and using integration by parts yields
	\begin{align*}
		\del_t \int_{\R^d} |u|^p \,dx + \int_{\R^d}|u|^p \div b + pc|u|^p\,dx
		= 0 \qquad \text{for a.e. $t \in (0, T)$.}
	\end{align*}
	Moving the time derivative in front of the integral is possible by the dominated convergence
	theorem, since $-b\cdot \nabla |u|^p + pc|u|^p \in L^1_tL^\infty_x$ and compactly supported. Thus
	\begin{align*}
		\del_t \int_{\R^d}|u|^p\,dx
		\le \norm*{\div b + pc}_{L^\infty(\R^d)}\int_{\R^d}|u|^p\,dx
		\qquad \text{for a.e. $t \in (0, T)$.}
	\end{align*}
	By Grönwall's inequality, we conclude
	\begin{align}
		\|u(t, \cdot)\|_{L^p(\R^d)}
		\le \exp\paren*{\|(\div b)/p + c\|_{L^1(0, T; L^\infty(\R^d))}}
		\|u^0\|_{L^p(\R^d)}\,.
		\label{eq:u-Lp-estimate}
	\end{align}

	\paragraph*{Step 2: Existence in the general case.}

	Let $b_n \in L^1(0, T; C^1_b(\R^d; \R^d))$ be a sequence from Lemma \ref{lemma:bn-seq}.
	For every $n \in \N$, we set
	\begin{align*}
		c_n\colonequals (c * \rho_{1/n}) \eta_n,
		\quad u^0_n\colonequals (u^0* \rho_{1/n}) \eta_n\,.
	\end{align*}
	By Step 1, there exists a weak solution $u_n$ of
	\begin{align*}
		\del_t u_n - b_n\cdot \nabla u_n + c_nu_n= 0, \qquad u_n|_{t = 0}= u^0_n\,.
	\end{align*}
	Since $\norm*{\div b_n}_{L^1_t L^\infty_x} \le C \norm*{\div b}_{L^1_t L^\infty_x}$,
	by Lemma \ref{lemma:bn-seq}, (\ref{eq:u-Linfty-estimate}) and (\ref{eq:u-Lp-estimate})
	show that $u_n$ is bounded in $L^\infty(0, T; L^p(\R^d))$.

	By the Banach-Alaoglu theorem, a
	subsequence of $u_n$ converges weak-$*$ in $L^\infty(0, T; L^p(\R^d))$ to some $u$.
	We just consider this subsequence from now on and
	also denote it by $u_n$. Let $\phi \in C^\infty_c([0, T) \times \R^d)$.
	Then for every $n \in \N$
	\begin{align}
		\begin{split}
			- \iTR u_n
			 & \del_t\phi \,dx \,dt
			-\int_{\R^d}u_n^0\phi(0, x) \,dx      \\
			 & +\iTR u_n(b_n \! \cdot \nabla \phi
			+ (\div b_n \!+ c_n)\phi) \, dx \,dt = 0 \,.
		\end{split}
		\label{eq:u_n-transport-weak}
	\end{align}
	Since $q < \infty$, by standard properties of convolutions, $c_n\to c$
	in $L^1(0, T; L^{q}_\loc(\R^d))$. By Lemma \ref{lemma:bn-seq}, also $b_n \to b$ and $\div b_n \to \div b$
	in $L^1(0, T; L^{q}_\loc(\R^d))$. Thus, by
	letting $n \to \infty$ in (\ref{eq:u_n-transport-weak}), we conclude that
	$u$ satisfies (\ref{eq:transport-weak}).
\end{proof}

\begin{remark}
	The existence result also holds for the transport equation with a
	right-hand side $f \in L^1(0,T; L^p(\R^d))$:
	\begin{align*}
		\del_t u - b \cdot \nabla u + cu = f, \qquad u|_{t = 0} = u^0 \,.
	\end{align*}
	The proof is very similar. For Step 1, we assume $f \in L^1(0, T; C^1_b(\R^d))$.
	Instead of (\ref{eq:characteristic-ode}), one
	solves a linear, inhomogeneous ODE along the characteristics, and (\ref{eq:u-Linfty-estimate}),
	(\ref{eq:u-Lp-estimate}) are replaced by similar estimates with a contribution from $f$.
	In Step 2, we set $f_n = (f * \rho_{1/n})\eta_n$.
	\label{remark:existence-rhs}
\end{remark}

\subsection{Regularization}
\label{sec:regularization}

The following lemma will be an important tool in subsequent proofs, in particular for the
uniqueness results in §\ref{sec:uniqueness}. It shows that,
if $b$ belongs to a suitable Sobolev space, the regularization
$u * \rho_\eps$ of a solution $u$ is still an approximate solution. The lemma
is from \cite{DL89}, Theorem II.1.

\begin{lemma}
	Let $1 \le p, \alpha \le \infty$ and $1 \le \lambda < \infty$ such that
	$1/p + 1/\alpha \le 1/\lambda$. Assume that
	\begin{align*}
		b \in L^1(0, T; W^{1, \alpha}_\loc(\R^d)), \qquad c \in L^1(0, T; L^{\alpha}_\loc(\R^d))
	\end{align*}
	and that $u \in L^\infty(0, T; L^p_\loc(\R^d))$ solves (\ref{eq:transport})
	for some initial data $u^0 \in L^p(\R^d)$.
	Then $u_\eps \colonequals u * \rho_\eps$ satisfies
	\begin{align}
		\del_tu_\eps- b \cdot \nabla u_\eps+ cu_\eps= r_\eps,
		\qquad u_\eps|_{t = 0}= u^0 * \rho_\eps,
		\label{eq:regularized-transport}
	\end{align}
	where $r_\eps\to 0$ in $L^1(0, T; L^{\lambda}_\loc(\R^d))$.
	\label{lemma:regularization}
\end{lemma}

The proof of Lemma \ref{lemma:regularization} relies on estimating the
commutator between the directional derivative along $b$ and convolution with $\rho_\eps$.
Such an estimate is established in Lemma \ref{lemma:commutator}.
In the proof of the lemma and in later proofs, we use $C_{x, y, \dots}$ to denote a constant
depending only on $x, y, \dots$ (where $x, y$ can be numbers, functions etc.)

\begin{lemma}[Commutator estimate]
	Let $p, \alpha, \lambda$ be as in Lemma \ref{lemma:regularization}. If
	$b \in W^{1, \alpha}_\loc(\R^d; \R^d)$ and
	$w \in L^p_\loc(\R^d)$, then
	\begin{align*}
		(b \cdot \nabla w) * \rho_\eps- b \cdot \nabla(w * \rho_\eps) \to 0
		\qquad \text{in $L^{\lambda}_\loc(\R^d)$}
	\end{align*}
	as $\eps \to 0$. \label{lemma:commutator}
\end{lemma}

\begin{proof}
	This proof is from \cite{DL89}, p.~516, Lemma II.1. We first write out the convolution explicitly.
	\begin{align}
		(b \cdot \nabla w) * \rho_\eps- b \cdot \nabla(w * \rho_\eps)
		 & = -\int_{\R^d}w(y) \div_{y}(b(y) \rho_\eps(x - y))
		+ b(x) \cdot \nabla \rho_\eps(x - y) w(y) \, dy \notag                \\
		 & = \int_{\R^d}w(y) (b(y) - b(x)) \cdot \nabla \rho_\eps(x - y) \,dy
		- (w \div b) * \rho_\eps
		\label{eq:convolution-explicitly}
	\end{align}
	Since $\lambda < \infty$, in view of the assumed integrability on $w$ and $\nabla b$,
	we have $(w \div b) * \rho_\eps \to w \div b$ in $L^{\lambda}_\loc(\R^d)$.
	Thus, our goal is to show that the first term also converges to $w \div b$ in
	$L^{\lambda}_\loc(\R^d)$.

	\paragraph{Step 1: The smooth case.}

	For now, assume that in addition $b \in C^1(\R^d; \R^d)$
	and $w \in C^1(\R^d)$. Then an integration by parts yields
	\begin{align*}
		\int_{\R^d}
		 & w(y) (b(y) - b(x)) \cdot \nabla \rho_\eps(x - y) \,dy                \\
		 & = \int_{\R^d}\nabla w(y) \cdot (b(y) - b(x)) \,
		\rho_\eps(x - y) \,dy + \int_{\R^d}w(y) \div b(y) \rho_\eps(x - y) \,dy \\
		 & = (\nabla w \cdot b) * \rho_\eps
		- (\nabla w * \rho_\eps) \cdot b + (w \div b) * \rho_\eps\,.
	\end{align*}
	The first and third term converge to $\nabla w \cdot b$ and $w \div b$
	in $L^{\lambda}_\loc(\R^d)$, respectively. To see that the
	second term converges to $\nabla w \cdot b$, note that $b$ and $w$ belong to $C^1$,
	hence $\nabla w \in L^{\lambda}_\loc(\R^d)$ and
	$b \in L^\infty_\loc(\R^d)$.

	\paragraph{Step 2: Estimating the first term in (\ref{eq:convolution-explicitly}).}

	To reduce to the smooth case, we will need an estimate of
	\begin{align*}
		A \colonequals \norm*{
			\int_{\R^d} w(y) (b(y) - b(x)) \cdot \nabla \rho_\eps(x - y) \,dy
		}_{L^\lambda_x(B_R)} \,.
	\end{align*}
	By changing variables $z = (y - x)/\eps$ in integral over $y$,
	\begin{align*}
		A = \eps^d \, \norm*{
			\int_{\R^d} w(x + \eps z) (b(x + \eps z) - b(x)) \cdot \nabla \rho_\eps(\eps z) \,dz
		}_{L^\lambda_x(B_R)} \,.
	\end{align*}
	By the definition of $\rho_\eps$ in (\ref{eq:rho-eps}), $\supp \nabla \rho_\eps\subset B_{\eps C_\rho}$
	for some $C_{\rho}> 0$. Thus the integrand is nonzero only for $z \in B_{C_\rho}$. Moreover
	$|\nabla \rho_\eps| \le \tilde C_\rho \eps^{-d - 1}$, hence
	\begin{align*}
		A \le \frac{\tilde C_\rho}{\eps} \, \norm*{
			\int_{B_{C_\rho}} |w(x + \eps z)| |b(x + \eps z) - b(x)| \,dz
		}_{L^\lambda_x(B_R)} \,.
	\end{align*}
	By Minkowski's inequality and Hölder's inequality, and since $1/\lambda \ge 1/\alpha + 1/p$,
	\begin{align*}
		A & \le \frac{\tilde C_\rho}{\eps} \int_{B_{C_\rho}}
		\big\|{
		|w(x + \eps z)| |b(x + \eps z) - b(x)|
		}\big\|_{L^\lambda_x(B_R)} \,dz                             \\
		  & \le \frac{C_{\rho,\lambda,R,d}}{\eps} \int_{B_{C_\rho}}
		\norm*{w(x + \eps z)}_{L^p_x(B_R)} \,
		\norm*{b(x + \eps z) - b(x)}_{L^\alpha_x(B_R)} \,dz
		\,.
	\end{align*}
	For small enough $\eps$, we have $\eps C_\rho < 1$, hence
	\begin{align}
		A \le \frac{C_{\rho,\lambda,R,d}}{\eps} \, \norm*{w}_{L^p(B_{R + 1})}
		\int_{B_{C_\rho}} \norm*{b(x + \eps z) - b(x)}_{L^\alpha_x(B_R)} \,dz \,.
		\label{eq:A-w-int-b-diff-estimate}
	\end{align}
	Next, we use the fact that for every $h \in \R^d$
	\begin{align}
		\norm{b(\cdot + h) - b}_{L^\alpha(B_R)}\le |h| \norm{\nabla b}_{L^\alpha(B_{R + 2|h|})}\,.
		\label{eq:finite-differences}
	\end{align}
	Combining (\ref{eq:A-w-int-b-diff-estimate}) with (\ref{eq:finite-differences}) yields
	\begin{align}
		A \le C \, \|w\|_{L^p(B_{R + 1})}\|\nabla b\|_{L^\alpha(B_{R + 1 + C_\rho})}\,.
		\label{eq:A-estimate}
	\end{align}
	for some $C$ not depending on $w$ and $b$.

	\paragraph*{Step 3: The general case.}

	Let $b_{\delta}\colonequals b * \rho_{\delta}, \, w_{\gamma}\colonequals w * \rho_{\gamma}$
	for some $\delta, \gamma > 0$ and write
	\begin{align}
		\int_{\R^d}w(y) (b(y) - b(x)) \cdot \nabla \rho_\eps
		  & (x - y) \,dy = D + E + F   ,\quad \text{where}                                      \\
		D & \colonequals \int_{\R^d}w_{\gamma}(y)
		(b_{\delta}(y) - b_{\delta}(x)) \cdot \nabla \rho_\eps(x - y) \,dy \notag               \\
		E & \colonequals \int_{\R^d}w(y)
		(b(y) - b_{\delta}(y) - b(x) + b_{\delta}(x)) \cdot \nabla \rho_\eps(x - y) \,dy \notag \\
		F & \colonequals \int_{\R^d}
		(w(y) - w_{\gamma}(y))(b_{\delta}(y) - b_{\delta}(x)) \cdot \nabla \rho_\eps(x - y)\,dy \,. \notag
	\end{align}
	Let $R > 0$. We first consider the case $p <\infty$ and $\alpha < \infty$. Then
	$E \to 0$ as $\delta \to 0$ and $F \to 0$ as $\gamma \to 0$ in
	$L^{\lambda}(B_R)$ by (\ref{eq:A-estimate}). Moreover, both converge
	uniformly in $\eps$ and $F$ converges uniformly in $\delta$.
	By writing
	\begin{align}
		 & \norm*{
			\int_{\R^d} w(y) (b(y) - b(x)) \cdot \nabla \rho_\eps (x - y) \,dy - w \div b\,
		}_{L^\lambda(B_R)}\label{eq:final-diff} \\
		 & \hspace{2.4em}\le \norm*{
			w_\gamma \div b_\delta - w \div b}_{L^\lambda(B_R)}+ \norm*{D - w_\gamma \div b_\delta
		}_{L^\lambda(B_R)}
		+ \|E\|_{L^\lambda(B_R)}+ \|F\|_{L^\lambda(B_R)} \,,
		\notag
	\end{align}
	we may finish the argument in the case $p < \infty, \, \alpha < \infty$
	as follows. Choose $\delta$ and $\gamma$ such that
	$\norm*{w_\gamma \div b_\delta - w \div b}_{L^\lambda(B_R)}$ is small
	and $\|E\|_{L^\lambda(B_R)}$, $\|F\|_{L^\lambda(B_R)}$ are small. By
	Step 1, we may then choose $\eps$ such that
	$\norm*{D - w_\gamma \div b_\delta}_{L^\lambda(B_R)}$ is small.

	Next, we consider the case $p = \infty$. By (\ref{eq:A-estimate}),
	applied with the $\lambda$-norm for $w - w_{\gamma}$ and the $\infty$-norm
	for $b_{\delta}$,
	\begin{align}
		\|F\|_{L^\lambda(B_R)}
		\le C\|w - w_{\gamma}\|_{L^\lambda(B_R)}\|\nabla b_{\delta}\|_{L^\infty(B_R)}\,.
		\label{eq:F-lambda-infty}
	\end{align}
	Thus still $F \to 0$ as $\gamma \to 0$ in $L^{\lambda}(B_R)$, but no
	longer uniformly in $\delta$. Hence, $\gamma$ now has to be chosen after
	$\delta$. This is not a problem for $E$, since $E$ is independent of
	$\gamma$. To see that it is also not a problem for the first term on the right
	hand side of (\ref{eq:final-diff}), note that
	\begin{align*}
		\|w_{\gamma}\div b_{\delta}- w \div b\|_{L^\lambda(B_R)}
		\le \|w_{\gamma}(\div b_{\delta}- \div b)\|_{L^\lambda(B_R)}
		+ \|(w_{\gamma}- w) \div b \, \|_{L^\lambda(B_R)}\,.
	\end{align*}
	Since $\div b_{\delta}\to \div b$ in $L^{\lambda}(B_R)$ and
	$w_{\gamma}$ is bounded in $L^\infty(B_R)$, the first term converges
	to 0 as $\delta \to 0$, uniformly in $\gamma$. The second term converges to
	0 since $w_{\gamma}$ is bounded in $L^\infty(B_R)$ and $w_{\gamma} \to w$
	almost everywhere.

	The last step of the proof in the case $\alpha = \infty$ works
	symmetrically to $p = \infty$. Now, $\delta$ has to be chosen after
	$\gamma$. Note that for (\ref{eq:F-lambda-infty}), the smoothness of
	$b_{\delta}$ was important. Thus the definitions of $E$ and $F$ must be adapted:
	$w$ is replaced by $w_{\gamma}$ in $E$, and $b_{\delta}$ is replaced
	by $b$ in $F$. Then, one obtains an estimate for $E$ similar to (\ref{eq:F-lambda-infty}).
\end{proof}

\begin{remark}
	Lemma \ref{lemma:commutator} yields the following analogous result for time-dependent
	vector fields. If
	$b \in L^1(0, T; W^{1,\alpha}_\loc(\R^d; \R^d))$
	and $w \in L^\infty(0, T; L^p_\loc(\R^d))$, then
	\begin{align}
		(b \cdot \nabla w) * \rho_\eps - b \cdot \nabla (w * \rho_\eps) \to 0 \qquad
		\text{in $L^1(0, T; L^{\lambda}_\loc(\R^d))$}
		\label{eq:commutator-time-dependent}
	\end{align}
	as $\eps \to 0$. As always, the convolution applies only to the spatial variables.

	To prove (\ref{eq:commutator-time-dependent}), let $R > 0$.
	By Lemma \ref{lemma:commutator},
	\begin{align*}
		(b(t) \cdot \nabla w(t)) * \rho_\eps- b(t) \cdot \nabla (w(t) * \rho_\eps) \to 0
		\qquad \text{in $L^{\lambda}(B_R)$}
	\end{align*}
	for almost every $t \in [0, T]$. Moreover, by (\ref{eq:convolution-explicitly})
	and (\ref{eq:A-estimate}),
	\begin{align*}
		\|(b(t) \cdot \nabla w(t))
		 & * \rho_\eps-b(t) \cdot \nabla (w(t) * \rho_\eps)\|_{L^\lambda(B_R)}          \\
		 & \le C\|w(t)\|_{L^p(B_{R + 1})}\|\nabla b(t)\|_{L^\alpha(B_{R + 1 + C_\rho})}
		+ \|w(t) \div b(t)\|_{L^\lambda(B_{R + 1})}
	\end{align*}
	for almost every $t \in [0, T]$. Thus we conclude (\ref{eq:commutator-time-dependent})
	by the dominated convergence theorem, since the right-hand side is independent of $\eps$ and
	in $L^1([0, T])$.
	\label{remark:commutator-time}
\end{remark}

\begin{proof}[Proof of Lemma \ref{lemma:regularization}]
	Let $\tilde \rho_\eps(x) = \rho_\eps(-x)$ and $\phi \in C^\infty_c([0, T) \times \R^d)$.
	Note that $\phi * \tilde \rho_\eps$ is also in $C^\infty_c([0, T) \times \R^d)$.
	Using that $u$ satisfies (\ref{eq:transport-weak}), we compute
	\begin{align}
		-\iTR u_\eps \del_t \phi \,dx \,dt
		 & = -\iTR u \, \del_t \phi * \tilde \rho_\eps \,dx \,dt
		\label{eq:u-eps-convolution}                                \\
		 & = \int_{\R^d} u^0 \phi(0, \cdot) * \tilde \rho_\eps \,dx
		-\!\iTR u \div (b (\phi * \tilde \rho_\eps)) + uc (\phi * \tilde \rho_\eps) \,dx \,dt \,. \notag
	\end{align}
	By definition, the integral of $-u \div (b(\phi * \tilde \rho_\eps))$ is
	the distribution $(b \cdot \nabla u) * \rho_\eps$ evaluated at $\phi$.
	Thus, by Lemma \ref{lemma:commutator} and Remark \ref{remark:commutator-time}
	\begin{align*}
		-\iTR u \div (b(\phi * \tilde \rho_\eps)) \,dx \,dt
		= \iTR (b \cdot \nabla u_\eps + s_\eps) \phi \,dx \,dt
	\end{align*}
	where $s_\eps \to 0$ in $L^1(0, T; L^\lambda_\loc(\R^d))$.
	Substituting this into (\ref{eq:u-eps-convolution}), we obtain
	\begin{align*}
		-\iTR u_\eps \del_t \phi & \,dx \,dt
		= \int_{\R^d} u^0 * \rho_\eps \, \phi(0, x) \,dx                                   \\
		                         & + \iTR (b \cdot \nabla u_\eps - cu_\eps) \phi \,dx \,dt
		+ \iTR (s_\eps + c u_\eps - (cu) * \rho_\eps) \phi \,dx \,dt \,.
	\end{align*}
	Because $(cu) * \rho_\eps- cu_\eps \to 0$ in $L^1(0, T; L^\lambda_\loc(\R^d))$,
	we may set $r_\eps \colonequals s_\eps + c u_\eps - (cu) * \rho_\eps$,
	which completes the proof.
\end{proof}

\begin{remark}
	A corresponding result also holds for the transport equation with
	a right-hand side.
	We keep the notation and assumptions of Lemma \ref{lemma:regularization}.
	If $f \in L^1(0, T; L^\lambda_\loc(\R^d))$ and $u$ satisfies (\ref{eq:transport})
	with $f$ on the right-hand side, then also
	\begin{align*}
		\del_t u_\eps - b \cdot \nabla u_\eps + c u_\eps = f + r_\eps,
		\qquad u_\eps|_{t = 0} = u^0 * \rho_\eps
	\end{align*}
	for some $r_\eps \to 0$ in $L^1(0, T; L^\lambda_\loc(\R^d))$.
	This is a direct consequence of the above proof, because $\lambda < \infty$, hence
	$f * \rho_\eps \to f$ in $L^1(0, T; L^\lambda_\loc(\R^d))$.
	\label{remark:regularization-rhs}
\end{remark}

The regularization result, Lemma \ref{lemma:regularization}, allows us
to show that the chain rule holds for solutions of the transport equation,
in the following sense.

\begin{lemma}
	Let $1 \le p \le \infty, \ 1/p + 1/q = 1$. Assume
	\begin{align}
		b \in L^1(0, T; W^{1, q}_\loc(\R^d; \R^d)),
		\qquad c \in L^1(0, T; L^q_\loc(\R^d))
		\label{eq:b-W1q-c-Lq}
	\end{align}
	and let $u \in L^\infty(0, T; L^p_\loc(\R^d))$ be a
	solution of (\ref{eq:transport}) for the initial data
	$u^0\in L^1_\loc(\R^d)$. If $\beta \in C^1(\R)$ and $\beta'$ is bounded, then $\beta(u)$ satisfies
	\begin{align}
		\del_t \beta(u) - b \cdot \nabla \beta(u) + cu\beta'(u) = 0,
		\qquad \beta(u)|_{t = 0}= \beta(u^0)
		\label{eq:chain-rule}
	\end{align}
	in the weak sense.
	\label{lemma:solutions-are-renormalized}
\end{lemma}

That (\ref{eq:chain-rule}) holds in the weak sense means
that for every $\phi \in C^\infty_c([0, T) \times \R^d)$
\begin{align}
	\begin{split}
		-\iTR & \beta(u) \del_t \phi \,dx \,dt
		-\int_{\R^d}\beta(u^0) \phi(0, x) \,dx                            \\
		      & \qquad + \iTR \beta(u)(\phi \div b + b \cdot \nabla \phi)
		+ c u \beta'(u) \phi \,dx \,dt = 0 \,.
	\end{split}
	\label{eq:beta-weak}
\end{align}

\begin{proof}
	By Lemma \ref{lemma:regularization}, $u_\eps= u * \rho_\eps$
	satisfies
	\begin{align*}
		\del_tu_\eps- b \cdot \nabla u_\eps+ cu_\eps= r_\eps, \qquad u_\eps|_{t = 0}= u_\eps^0\,,
	\end{align*}
	where $u^0_\eps= u^0* \rho_\eps$ and $r_\eps\to 0$ in
	$L^1_\loc([0, T] \times \R^d)$. Since $u_\eps\in W^{1, 1}_\loc((0, T) \times \R^d)$
	and $\beta'$ is bounded, by the chain rule for Sobolev functions
	\begin{align*}
		\del_t \beta(u_\eps) - b \cdot \nabla \beta(u_\eps) + cu_\eps\beta'(u_\eps)
		= r_\eps\beta'(u_\eps),
		\qquad \beta(u_\eps)|_{t = 0} = \beta(u_\eps^0) \,.
	\end{align*}
	Testing against $\phi \in C^\infty_c([0, T) \times \R^d)$ yields
	\begin{align}
		\begin{split}
			-\iTR & \beta(u_\eps) \del_t \phi \,dx \,dt
			-\int_{\R^d}\beta(u_\eps^0) \phi(0, x) \,dx                     \\
			      & + \iTR \beta(u_\eps)(\phi \div b + b \cdot \nabla \phi)
			+ c u_\eps\beta'(u_\eps) \phi \,dx \,dt = \iTR r_\eps\beta'(u_\eps)\phi \,dx \,dt \,.
		\end{split}
		\label{eq:renormalized-u-eps}
	\end{align}
	We let $\eps \to 0$ in this equation with the aim of obtaining (\ref{eq:beta-weak}). Since
	$\beta'$ is bounded and $r_\eps \to 0$ in $L^1_\loc([0,T] \times \R^d)$, the right-hand
	side converges to 0 as $\eps \to 0$. For the $cu_\eps \beta'(u_\eps)$ term, we want to show that
	\begin{align}
		\iTR cu_\eps \beta'(u_\eps) \phi \,dx \,dt \ \to \
		\iTR cu \beta'(u) \phi \,dx \,dt  \qquad \text{as $\eps \to 0$.}
		\label{eq:cu-desired-conv}
	\end{align}
	In the case $p < \infty$, we write
	\begin{align*}
		cu_\eps \beta'(u_\eps) - cu\beta'(u)
		= c(u_\eps - u)\beta'(u_\eps) + cu(\beta'(u_\eps) - \beta'(u)) \,.
	\end{align*}
	Since $u_\eps(t) \to u(t)$ in $L^p_\loc(\R^d)$ for a.e. $t \in (0, T)$ and $\beta'$ is
	bounded, $\int_0^T\int_{\R^d} c(u_\eps - u)\beta'(u_\eps) \,dx\,dt$ converges to 0 as $\eps \to 0$.
	Since $u_\eps \to u$ pointwise almost everywhere, also
	$\int_0^T\int_{\R^d} cu(\beta'(u_\eps) - \beta'(u)) \,dx \,dt$ converges to 0 as $\eps \to 0$.
	In the case $p = \infty$, we conclude (\ref{eq:cu-desired-conv}) directly from dominated convergence.

	For the remaining terms in (\ref{eq:renormalized-u-eps}), if $p < \infty$,
	we use that $\beta(u_\eps) \to \beta(u)$ in $L^p_\loc(\R^d)$ since $\beta'$ is bounded. If
	$p = \infty$, we apply the dominated convergence
	theorem, using that $u_\eps \to u$ almost everywhere on $[0, T] \times \R^d$.
\end{proof}

The same proof shows a similar result for the transport equation with
a right-hand side $f \in L^1(0, T; L^1_\loc(\R^d))$. In this case, there is a
term $f\beta'(u)$ on the right-hand side in (\ref{eq:chain-rule}).

\subsection{Renormalized solutions}

We introduce here a new notion of a solution for the transport equation
(\ref{eq:transport}), motivated by formula (\ref{eq:chain-rule}). First,
we set up some notation. If $A \subset \R^d$ is measurable, we denote by $L(A)$
the set of measurable functions $A \to \R$. Let $f_n \in L(A)$  for $n \in \N$ and $f \in L(A)$.
We say that $f_n \to f$ in $L(A)$
if $f_n|_K$ converges to $f|_K$ in measure for every compact $K \subset A$. Using this notion
of convergence, we also define the space $C([0, T]; L(A))$.

\begin{definition}[Admissible function]
	We call $\beta \in C^1(\R)$ an admissible function, if $\beta$ and $\beta' \cdot (1 + |x|)$ are
	bounded on $\R$.
\end{definition}

\begin{definition}[Renormalized solution]
	Let	$b \in L^1_\loc([0, T] \times \R^d; \R^d)$ and
	$c, \ \div b \in L^1_\loc([0, T] \times \R^d)$.
	Let $u^0\in L(\R^d)$ and $u \in L([0, T] \times \R^d)$.
	If, for every admissible function $\beta$,
	\begin{align}
		\del_t \beta(u) - b \cdot \nabla \beta(u) + cu\beta'(u) = 0, \qquad
		\beta(u)|_{t = 0}= \beta(u^0)
		\label{eq:def-renormalized}
	\end{align}
	holds in the weak sense, then we call $u$ a renormalized solution of (\ref{eq:transport})
	corresponding to the initial condition $u^0$.
	\label{def:renormalized}
\end{definition}

Definition \ref{def:renormalized} provides a notion of a solution for
very general initial data $u^0 \in L(\R^d)$, for which the weak definition (\ref{eq:transport-weak})
is not well defined. This will become important
in the application to ordinary differential equations.
Moreover, renormalized solutions have several
features of smooth solutions, such as the conservation
of $L^p$-norms in the case $\div b = 0$ and $c = 0$ (see Proposition \ref{prop:continuity}).

The following result establishes the relation between weak and renormalized solutions.
It corresponds to part 1) of Theorem II.3 in \cite{DL89}.

\begin{proposition}[Consistency]
	Let $1 \le p \le \infty, \ 1/p + 1/q = 1$. Let $u \in L^\infty(0, T; L^p_\loc(\R^d))$ and
	$u^0\in L^1_\loc(\R^d)$. Assume that $c, \, \div b \in L^1(0, T; L^q_\loc(\R^d))$.
	\begin{enumerate}
		\item If $b \in L^1(0, T; L^{q}_\loc(\R^d; \R^d))$
		      and $u$ is a renormalized solution of (\ref{eq:transport}),
		      then $u$ is a weak solution.
		\item If $b \in L^1(0, T; W^{1, q}_\loc(\R^d; \R^d))$ and $u$ is
		      a weak solution of (\ref{eq:transport}), then $u$ is a renormalized solution.
	\end{enumerate}
	\label{prop:consistency}
\end{proposition}

\begin{proof}
	Part 2) is a consequence of Lemma \ref{lemma:solutions-are-renormalized}.
	For part 1), we choose a sequence of admissible functions $\beta_n$
	such that $\beta_n\to \id_{\R}$, $\beta_n' \to 1$ pointwise, $|\beta_n(\sigma)| \le |\sigma|$ and
	$|\beta_n'(\sigma)| \le 1$ for every $\sigma \in \R$. For example, one may take
	$\beta_n(\sigma) = n \tanh(\sigma/n)$. By (\ref{eq:def-renormalized}),
	for every $\phi \in C^\infty_c([0, T) \times \R^d)$
	\begin{align*}
		-\iTR \beta_n(u)
		 & \del_t\phi \,dx \,dt - \int_{\R^d}\beta_n(u^0) \phi(0, x) \,dx \\
		 & \qquad + \iTR \beta_n(u) (\phi \div b + b \cdot \nabla \phi
		+ c u \beta_n'(u) \phi) \,dx \,dt = 0 \,.
	\end{align*}
	We let $n \to \infty$ and use the dominated convergence theorem to conclude (\ref{eq:transport-weak}).
	For the third integral, $|u|(|\phi \div	b| + |b| |\nabla \phi| + |\phi c|)$ is a majorant, since
	$|\beta_n(u)| \le |u|$ and $|\beta_n'(u)| \le 1$. The majorant is integrable,
	because	$b \in L^1(0, T; L^{q}_\loc(\R^d; \R^d))$ and $c, \, \div b \in L^1(0, T; L^q_\loc(\R^d))$.
\end{proof}

\begin{remark}
	Similarly, one can define renormalized solutions for the transport equation
	with a right-hand side $f \in L^1_\loc([0, T] \times \R^d)$. In this case,
	we require that
	\begin{align*}
		\del_t \beta(u) - b \cdot \nabla \beta(u) + cu\beta'(u) = f\beta'(u),
		\qquad \beta(u)|_{t = 0} = \beta(u^0)
	\end{align*}
	for every admissible function $\beta$. Proposition \ref{prop:consistency} still holds
	for $f \in L^1_\loc([0, T] \times \R^d)$ by the same proof.
	\label{remark:renormalized-rhs}
\end{remark}

\subsection{Uniqueness}
\label{sec:uniqueness}

In this section, we prove the uniqueness of renormalized and weak
solutions in $L^\infty(0, T; L^p(\R^d))$ for $1 \le p \le \infty$.
For all uniqueness results,
we will assume that $b$ satisfies the following growth condition:
\begin{align}
	\frac{|b|}{1 + |x|}\in L^1(0, T; L^1(\R^d)) + L^1(0, T; L^\infty(\R^d))
	\tag{$\star$}
	\label{eq:b-growth-condition}
\end{align}

\subsubsection{The case $p < \infty$}

The following Proposition collects several
properties of renormalized solutions. Uniqueness for $p < \infty$
will follow immediately from it.

\begin{proposition}
	Let	$b \in L^1_\loc([0, T] \times \R^d; \R^d)$ and $c, \ \div b \in L^1_\loc([0, T] \times \R^d)$.
	Let $1 \le p \le \infty$, let $f \in L^1(0, T; L^p(\R^d))$ and let $u \in L^\infty(0, T; L^p(\R^d))$
	be a renormalized solution of
	\begin{align*}
		\del_t u - b \cdot \nabla u + cu = f, \qquad u|_{t = 0} = u^0,
	\end{align*}
	where $u^0 \in L^p(\R^d)$.
	\begin{enumerate}
		\item If $p = \infty$, then $u \in C([0, T]; L^r_\loc(\R^d))$ for every $1 \le r < \infty$.
		\item Assume $p < \infty$, (\ref{eq:b-growth-condition}), and that
		      $c, \, \div b \in L^1(0, T; L^\infty(\R^d))$.
		      Then $u \in C([0, T]; L^p(\R^d))$ and for a.e. $t \in (0, T)$
		      \begin{align}
			      \frac{d}{dt} \int_{\R^d} |u|^p \,dx =
			      -\int_{\R^d} (\div b + pc)|u|^p \,dx
			      + p \int_{\R^d} f |u|^{p - 2}u \,dx \,.
			      \label{eq:ddt-Lp-norm}
		      \end{align}
		      Furthermore, for every admissible $\beta$ such that $0 \notin \supp \beta$,
		      \begin{align}
			      \frac{d}{dt} \int_{\R^d} \beta(u) \,dx =
			      -\int_{\R^d} (\beta(u)\div b + cu\beta'(u)) \,dx
			      + \int_{\R^d} f \beta'(u) \,dx
			      \label{eq:ddt-int-beta}
		      \end{align}
		      for a.e. $t \in (0, T)$.
	\end{enumerate}
	\label{prop:continuity}
\end{proposition}

The proof of Proposition \ref{prop:continuity} uses the following Lemma.

\begin{lemma}
	Let $(X, \mu)$ be a measure space and $1 \le p < \infty$. Let $f_n, f \in L^p(\mu)$ for $n \in \N$.
	If $f_n \to f$ $\mu$-a.e. and $\|f_n\|_{L^p(\mu)} \to \|f\|_{L^p(\mu)}$ as $n \to \infty$, then
	$f_n \to f$ in $L^p(\mu)$.
	\label{lemma:pointwise-ae-Lp}
\end{lemma}

\begin{proof}
	Since the function $x \mapsto |x|^p$ is convex on $\R$, the following pointwise inequality holds.
	\begin{align*}
		|f_n - f|^p \le 2^{p - 1}(|f_n|^p + |f|^p)
	\end{align*}
	Thus, by Fatou's Lemma and the almost everywhere convergence of $f_n$ to $f$,
	\begin{align*}
		\liminf_{n \to \infty} \int_X 2^{p - 1}(|f_n|^p + |f|^p) - |f_n - f|^p \,d\mu
		\ge 2^p \int_X |f|^p \,d\mu \,.
	\end{align*}
	We now use the following fact: If $a_n, b_n$ are sequences of real numbers, and $a_n \to a$ as $n \to \infty$,
	then $\liminf_{n \to \infty}(a_n + b_n) = a + \liminf_{n \to \infty} b_n$.
	By applying this to the sequence $\|f_n\|_{L^p(\mu)}^p$,
	\begin{align*}
		\liminf_{n \to \infty} \int_X 2^{p - 1}(|f_n|^p + |f|^p) - |f_n - f|^p \,d\mu
		= 2^p \int_X |f|^p \,d\mu - \limsup_{n \to \infty} \int_X |f_n - f|^p \,d\mu \,.
		 & \qedhere
	\end{align*}
\end{proof}

\begin{proof}[Proof of Proposition \ref{prop:continuity}]
	Let $\beta$ be an admissible function and $\xi \in C^\infty_c(\R^d), \ \psi \in C^\infty_c([0, T))$.
	Since $u$ is a renormalized solution,
	\begin{align}
		\begin{split}
			- & \iTR \beta(u) \xi \del_t\psi \,dx \,dt
			- \psi(0)\int_{\R^d}\beta(u^0) \xi \,dx                          \\
			  & \qquad + \iTR \psi \beta(u)(\xi \div b + b \cdot \nabla \xi)
			+ \psi \xi c u \beta'(u) \,dx \,dt
			= \iTR \psi \xi f \beta'(u) \,dx \,dt \,.
		\end{split}
		\label{eq:beta-xi-psi}
	\end{align}
	Since $b, c, f \in L^1(0, T; L^1_\loc)$ and $\xi \in C^\infty_c(\R^d)$,
	(\ref{eq:beta-xi-psi}) implies
	\begin{align}
		\int_{\R^d} \beta(u) \xi \,dx \in W^{1, 1}((0, T)) \hookrightarrow C([0, T]) \,.
		\label{eq:beta-xi-w11}
	\end{align}
	Since $\{\beta(u(t)) : t \in [0, T]\}$ is bounded in $L^\infty(\R^d)$, the Banach-Alaoglu theorem and
	(\ref{eq:beta-xi-w11}) imply $\beta(u(s)) \rightharpoonup^* \beta(u(t))$ in $L^\infty(\R^d)$ if $s \to t$ in $[0, T]$.
	Since $\beta^2$ is admissible, also $\beta(u(s))^2 \rightharpoonup^* \beta(u(t))^2$ in
	$L^\infty(\R^d)$ if $s \to t$ in $[0, T]$. Let $K \subset \R^d$ be compact.
	Since $\meas^d(K) < \infty$, the characteristic function $\chi_K$ lies in $L^1(\R^d)$,
	so $\|\beta(u(t))\|_{L^2(K)} \in C([0, T])$. Thus
	$\beta(u) \in C([0, T]; L^2(K))$, since $L^2(K)$ is uniformly convex.
	By iterating this argument for $L^{2^k}_\loc$, $k \in \N$, we conclude
	\begin{align}
		\hspace{8em} \beta(u) \in C([0, T]; L^r_\loc(\R^d))
		\qquad \text{for every $1 \le r < \infty$.}
		\label{eq:beta-in-Lr}
	\end{align}

	\paragraph{Part 1).}

	If $u \in L^\infty([0, T] \times \R^d)$, we can choose $\beta$ admissible
	such that $\beta(u) = u$ almost everywhere. So part 1) follows from (\ref{eq:beta-in-Lr}).

	\paragraph{Part 2).}

	We first show (\ref{eq:ddt-int-beta}), so let $\beta$ be admissible and assume
	that $0 \notin \supp \beta$. We choose $\xi = \eta_R$
	in (\ref{eq:beta-xi-psi}), where $\eta_R$ is defined in (\ref{eq:eta-R}).
	We then let $R \to \infty$ with the goal of obtaining
	\begin{align}
		\begin{split}
			-\iTR \beta(u) & \del_t\psi \,dx \,dt
			- \psi(0)\int_{\R^d}\beta(u^0) \,dx                                             \\
			               & \qquad + \iTR \psi (\beta(u) \div b + c u \beta'(u)) \,dx \,dt
			= \iTR \psi f\beta'(u)  \,.
		\end{split}
		\label{eq:ddt-int-beta-integral}
	\end{align}
	Since $p < \infty$ and $0 \notin \supp \beta$,
	\begin{align}
		\beta(u), \ u\beta'(u), \ \beta'(u)
		\in L^\infty(0, T; L^1(\R^d)) \cap L^\infty(0, T; L^\infty(\R^d)).
		\label{eq:L1-cap-Linfty}
	\end{align}
	Therefore, all terms in (\ref{eq:beta-xi-psi}) except the one with
	$\nabla \eta_R$ converge as desired by dominated convergence.
	For the $\nabla \eta_R$-term, we note that $|\nabla \eta_R| \le C_\eta/R$, that
	$\supp \eta_R \subset B_{2R} \setminus B_R$ and that $1/R \le 2/(1 + |x|)$ for
	$x \in B_{2R} \setminus B_R$. Hence
	\begin{align*}
		\abs*{\iTR \psi \beta(u) \, b \cdot \nabla \eta_R \,dx \,dt}
		\le 2C_\eta \int_0^T \contractint \int_{B_{2R} \setminus B_R}
		|\psi| |\beta(u)| \frac{|b|} {1 + |x|} \,dx \,dt \,.
	\end{align*}
	By the growth condition (\ref{eq:b-growth-condition}) on $b$ and by (\ref{eq:L1-cap-Linfty}), the
	integrand on the right-hand side is in $L^1([0, T] \times \R^d)$. Hence the integral
	converges to 0 as $R \to \infty$.
	Thus (\ref{eq:ddt-int-beta-integral}) holds, and since $\psi \in C^\infty_c([0, T))$ is arbitrary,
	also (\ref{eq:ddt-int-beta}).

	To show (\ref{eq:ddt-Lp-norm}), we choose
	a sequence of admissible functions $\beta_n$ such that $0 \notin \supp \beta_n$,
	$\beta_n(\sigma) \ge 0$, $\sigma \beta_n'(\sigma) \ge 0$ for all $\sigma \in \R$ and
	\begin{align*}
		\beta_n(\sigma), \, \sigma \beta_n'(\sigma) \le C|\sigma|^p,
		\qquad \beta_n(\sigma) \xrightarrow[n]{} |\sigma|^p,
		\quad \beta_n'(\sigma) \xrightarrow[n]{} p\sigma|\sigma|^{p - 2}
		\quad \text{for all $\sigma \in \R \setminus \{0\}$} \,.
	\end{align*}
	Here, $C$ must be independent of $n$ and $\sigma$. Note that because $0 \notin \supp \beta_n$, we
	automatically have $\beta_n(0) \to 0$ and $\beta_n'(0) \to 0$ as $n \to \infty$.
	An example for a sequence with these properties is
	\begin{align*}
		\beta_n(\sigma) = |n\tanh(\sigma/n)|^p (1 - \eta_{1/n}(\sigma)),
	\end{align*}
	where $\eta_{1/n}$ is a smooth cutoff function as in (\ref{eq:eta-R}) with $d = 1$.
	We now choose $\beta = \beta_n$ in (\ref{eq:ddt-int-beta-integral}) and let $n \to \infty$,
	aiming to show
	\begin{align}
		\begin{split}
			-\iTR & |u|^p \del_t\psi \,dx \,dt
			- \psi(0) \int_{\R^d} |u^0|^p \,dx               \\
			      & + \iTR \psi (\div b + pc)|u|^p \,dx \,dt
			= \iTR \psi f |u|^{p - 2}u  \,dx \,dt\,.
		\end{split}
		\label{eq:ddt-Lp-norm-int}
	\end{align}
	Since $u \in L^\infty_t L^p_x$ and we assume $c, \, \div b \in L^1_t L^\infty_x$,
	all integrals on the left-hand side converge by dominated convergence.
	For the right-hand side, notice that $\beta_n'(u) \le C|u|^{p - 1}$.
	Since $|u|^{p - 1} \in L^\infty_t L^{p/(p - 1)}_x$,
	by Hölder's inequality $f|u|^{p - 1} \in L^1_t L^1_x$.
	Thus dominated convergence also applies.
	This concludes the proof of
	(\ref{eq:ddt-Lp-norm-int}) and hence of (\ref{eq:ddt-Lp-norm}).

	To show $u \in C([0, T]; L^p(\R^d))$, let $t_n \to t$ be a convergent sequence in $[0, T]$.
	By (\ref{eq:ddt-Lp-norm}), $\|u(t)\|_{L^p(\R^d)} \in C([0, T])$.
	So by Lemma \ref{lemma:pointwise-ae-Lp},
	we only need to show that $u(t_n)$ converges to $u(t)$ almost everywhere on $\R^d$.
	Furthermore, it suffices to show this
	for a subsequence, since $t_n \to t$ is an arbitrary convergent sequence.
	By a diagonal argument, it suffices to show that for every $R > 0$,
	there exists a subsequence $t_{k_n}$ of $t_n$ such that
	\begin{align}
		u(t_{k_n}) \to u(t) \qquad \text{a.e. on $B_R$.}
		\label{eq:pointwise-local-trunc}
	\end{align}
	Let $\beta$ be an admissible function such that $\beta'(\sigma) > 0$ for all $\sigma \in \R$.
	By (\ref{eq:beta-in-Lr}),
	\begin{align*}
		\beta(u(t_n)) \to \beta(u(t)) \qquad \text{in $L^r(B_R)$ as $n \to \infty$.}
	\end{align*}
	Thus, $\beta(u(t_{k_n})) \to \beta(u(t))$ a.e. on $B_R$ for some subsequence $t_{k_n}$ of $t_n$.
	Since $\beta' > 0$, this implies (\ref{eq:pointwise-local-trunc}).
\end{proof}

\begin{corollary}[Uniqueness for $p < \infty$]
	Let $1 \le p < \infty, \ 1/p + 1/q = 1$. Assume (\ref{eq:b-growth-condition}) and
	\begin{align*}
		b \in L^1(0, T; W^{1, q}_\loc(\R^d; \R^d)),
		\qquad c, \, \div b \in L^1(0, T; L^\infty(\R^d)) \,.
	\end{align*}
	If $u \in L^\infty(0, T; L^p(\R^d))$ is a solution of (\ref{eq:transport}) for
	$u^0 = 0$, then $u = 0$.
	\label{cor:uniqueness-p}
\end{corollary}

\begin{proof}
	By Proposition \ref{prop:consistency}, $u$ is a renormalized solution.
	Thus (\ref{eq:ddt-Lp-norm}) and Grönwall's inequality imply $u = 0$.
\end{proof}

\begin{remark}
	Proposition \ref{prop:continuity} also shows the following property of renormalized solutions,
	which is familiar from smooth solutions: If $c = 0$ and $\div b = 0$, the measures of
	level sets of $u$ are constant. To show this, let $\beta$ be admissible such that $0 \notin \supp \beta$.
	By (\ref{eq:ddt-int-beta}), for almost every $t \in (0, T)$
	\begin{align*}
		\int_{\R^d} \beta(u(t)) \,dx = \int_{\R^d} \beta(u^0) \,dx \,.
	\end{align*}
	Let $0 < a < b$. Then, letting $\beta \uparrow \chi_{(a, b)}$ in the above equation and the monotone
	convergence theorem imply that $\meas^d(\{a < u(t) < b\}) = \meas^d(\{a < u^0 < b\})$
	for almost every $t \in (0, T)$.
\end{remark}

\subsubsection{The case $p = \infty$}

The uniqueness of weak solutions for $p = \infty$ will be a consequence of the following theorem.
It is the same as Theorem II.6 of \cite{DL89}, except that we allow a right-hand side
for both renormalized solutions.

\begin{theorem}[Integration by parts along the flow]
	Assume (\ref{eq:b-growth-condition}) and
	\begin{align*}
		 & b \in L^1(0,T; W^{1, 1}_\loc(\R^d; \R^d)),
		 & c, \ \div b \in L^1(0, T; L^\infty(\R^d)) \,.
	\end{align*}
	Let $1 \le p \le \infty, \, 1/p + 1/q = 1$ and let
	$f \in L^1(0, T; L^p(\R^d))$, $g \in L^1(0, T; L^q(\R^d))$.
	If $u \in L^\infty(0, T; L^p(\R^d))$ and
	$v \in L^\infty(0, T; L^q(\R^d))$
	are renormalized solutions of
	\begin{align*}
		\del_t u - b \cdot \nabla u + cu = f,    \qquad
		\del_t v - b \cdot \nabla v - (\div b + c)v = g,
	\end{align*}
	then
	\begin{align}
		\iTR u g \,dx \,dt = -\iTR fv \,dx \,dt
		+ \int_{\R^d} u(T)v(T) \,dx - \int_{\R^d} u(0)v(0) \,dx \,.
		\label{eq:ibp}
	\end{align}
	\label{theorem:ibp}
\end{theorem}

The boundary terms on right-hand side of (\ref{eq:ibp}) make sense due to the time-continuity
of $u$ and $v$, which was shown in Proposition \ref{prop:continuity}.

\begin{proof}
	This proof is from \cite{DL89}, Theorem II.6.
	Let $\beta$ be an admissible function and set $\tilde u \colonequals \beta(u)$,
	$\tilde v \colonequals \beta(v)$.
	By definition of a renormalized solution, $\tilde u$ and $\tilde v$ satisfy
	\begin{align*}
		\del_t \tilde u - b \cdot \nabla \tilde u + cu\beta'(u)            & = f\beta'(u)     \\
		\del_t \tilde v - b \cdot \nabla \tilde v - (c + \div b)v\beta'(v) & = g\beta'(v) \,.
	\end{align*}
	Let $\tilde u_\eps \colonequals \tilde u * \rho_\eps$ and
	$\tilde v_\eps \colonequals \tilde v * \rho_\eps$. By Lemma
	\ref{lemma:regularization} and Remark \ref{remark:regularization-rhs}
	\begin{align}
		\begin{split}
			\del_t \tilde u_\eps - b \cdot \nabla \tilde u_\eps + c u \beta'(u)
			 & = f\beta'(u) +  r_\eps, \\
			\del_t \tilde v_\eps- b \cdot \nabla \tilde v_\eps- (c + \div b) v \beta'(v)
			 & = g\beta'(v) + s_\eps,
		\end{split}
		\label{eq:tilde-u-v}
	\end{align}
	where $r_\eps, s_\eps\to 0$ in
	$L^1(0, T; L^1_\loc(\R^d))$ as $\eps \to 0$. For $R > 0$, let $\eta_R$
	be a smooth cutoff function defined in (\ref{eq:eta-R}). Due to the regularization,
	the equations in (\ref{eq:tilde-u-v})
	are no longer just equations of distributions, but also of
	$L^1_\loc([0, T] \times \R^d)$-functions.
	Hence, we can multiply the equation for $\tilde u_\eps$ by
	$\tilde v_\eps \eta_R$ and integrate over $[0, T] \times \R^d$:
	\begin{align}
		\begin{split}
			\iTR \eta_R \tilde v_\eps\del_t  \tilde u_\eps \,dx \,dt
			 & + \iTR \eta_R \tilde v_\eps(-b \cdot \nabla \tilde u_\eps + cu\beta'(u)) \,dx \,dt \\
			 & \qquad = \iTR \eta_R \tilde v_\eps (f\beta'(u) + r_\eps) \,dx\,dt
		\end{split}
		\label{eq:int-v-dt-u}
	\end{align}
	If $\eta_R \tilde v_\eps$ would belong to $C^1([0, T] \times \R^d)$, we could conclude from
	the trace theorem that
	\begin{align}
		\begin{split}
			\iTR \eta_R \tilde v_\eps \del_t \tilde u_\eps \,dx\,dt
			= & -\iTR \eta_R \tilde u_\eps \del_t \tilde v_\eps \,dx \,dt   \\
			\addlinespace
			  & + \int_{\R^d} \eta_R \tilde v_\eps(T) \tilde u_\eps(T) \,dx
			- \int_{\R^d} \eta_R \tilde v_\eps(0) \tilde u_\eps(0) \,dx \,.
		\end{split}
		\label{eq:trace-u-v}
	\end{align}
	However, because $\eta_R \tilde v_\eps \notin C^1([0, T] \times \R^d)$, some
	additional arguments are needed. Let $Z = (0, T) \times B_{2R}$ and
	$F \in C^\infty_c([0, T] \times B_{2R})$. Because $\tilde u_\eps \in W^{1, 1}(Z)$,
	by the trace theorem
	\begin{align}
		\int_Z F \del_t \tilde u_\eps \,dx \,dt
		= -\int_Z \tilde u_\eps \del_t F \,dx \,dt
		+ \int_{B_{2R}} F(T, x) \tilde u_\eps(T, x) \,dx
		- \int_{B_{2R}} F(0, x) \tilde u_\eps(0, x) \,dx \,.
		\label{eq:trace-F}
	\end{align}

	Next, we choose a sequence
	$F_n \in C^\infty_c(\R \times B_{2R})$ such that
	\begin{align}
		F_n \to \eta_R \tilde v_\eps \quad \text{in $W^{1, 1}(Z)$}, \qquad
		F_n \rightharpoonup^* \eta_R \tilde v_\eps \quad \text{in $L^\infty(Z)$} \,.
		\label{eq:Fn-Z-seq}
	\end{align}
	This works as follows. Because $Z$ has Lipschitz boundary, by the
	Whitney extension theorem, there exists a $h \in W^{1, 1}(\R^{d + 1})$
	such that $h|_Z = \eta_R \tilde v_\eps$. Since $\supp \eta_R \subset B_{2R}$,
	we may assume $\supp h \subset \R \times B_{2R}$ after multiplying $h$
	by a smooth cutoff function. Since $\eta_R \tilde v_\eps \in L^\infty(Z)$, we may
	also assume $h \in L^\infty(\R^{d + 1})$ by composing with a function in $C^1_b(\R)$ which
	is the identity on a suitably large interval. Then, there exists a sequence
	$F_n \in C^\infty_c(\R \times B_{2R})$, bounded in $L^\infty$,
	such that $F_n \to h$ in $W^{1, 1}(\R^{d + 1})$ as $n \to \infty$.
	After extraction of a subsequence, also $F_n \rightharpoonup^* h$ in $L^\infty(\R^{d + 1})$.

	We now set $F = F_n$ in (\ref{eq:trace-F}) and let $n \to \infty$.
	The left-hand side of (\ref{eq:trace-F}) converges to the left-hand side of (\ref{eq:trace-u-v})
	by the weak-$*$ convergence of $F_n$ and since
	$\del_t \tilde u_\eps \in L^1_\loc([0, T] \times \R^d)$. For the
	$\eta_R \tilde u_\eps \del_t \tilde v_\eps$-term, note that $\tilde u_\eps \in L^\infty(Z)$.
	The boundary terms converge as desired by the continuity of the trace operator
	$W^{1,1}(Z) \to L^1(\del Z)$ and since $\tilde u_\eps(0)$ and $\tilde u_\eps(T)$ are smooth
	and therefore in $L^\infty(B_{2R})$.

	Substituting (\ref{eq:trace-u-v}) into (\ref{eq:int-v-dt-u}) and
	integrating by parts in $x$ yields
	\begin{align*}
		-\iTR \eta_R
		 & \tilde u_\eps \del_t \tilde v_\eps \,dx \,dt
		+ \int_{\R^d} \eta_R \tilde v_\eps(T) \tilde u_\eps(T) \,dx
		- \int_{\R^d} \eta_R \tilde v_\eps(0) \tilde u_\eps(0) \,dx        \\
		 & + \iTR \tilde u_\eps(\eta_R \tilde v_\eps \div b
		+ \tilde v_\eps b \cdot \nabla \eta_R + \eta_R b \cdot \nabla \tilde v_\eps) +
		\eta_R \tilde v_\eps c u \beta'(u) \,dx \,dt                       \\
		 & = \iTR \eta_R \tilde v_\eps (f \beta'(u) + r_\eps) \,dx\,dt \,.
	\end{align*}
	By using equation (\ref{eq:tilde-u-v}) for $\tilde v_\eps$,
	\begin{align*}
		-\iTR \eta_R
		 & \tilde u_\eps(b \cdot \nabla \tilde v_\eps +
		(c + \div b)v\beta'(v) + g\beta'(v) + s_\eps) \,dx \,dt             \\
		 & + \int_{\R^d} \eta_R \tilde v_\eps(T) \tilde u_\eps(T) \,dx
		- \int_{\R^d} \eta_R \tilde v_\eps(0) \tilde u_\eps(0) \,dx         \\
		 & + \iTR \tilde u_\eps(\eta_R \tilde v_\eps \div b
		+ \tilde v_\eps b \cdot \nabla \eta_R + \eta_R b \cdot \nabla \tilde v_\eps) +
		\eta_R \tilde v_\eps c u \beta'(u) \,dx \,dt                        \\
		 & = \iTR \eta_R \tilde v_\eps (f \beta'(u) + r_\eps) \,dx \,dt \,.
	\end{align*}
	The terms $\eta_R \tilde u_\eps b \cdot \nabla \tilde v_\eps$ in the first line and third line
	cancel out.
	Letting $\eps \to 0$, we obtain
	\begin{align}
		- & \iTR \eta_R \tilde u((c + \div b)v\beta'(v) + g\beta'(v)) \,dx \,dt
		+ \int_{\R^d} \eta_R \tilde v(T) \tilde u(T) \,dx
		- \int_{\R^d} \eta_R \tilde v(0) \tilde u(0) \,dx      \notag           \\
		  & + \iTR \tilde u(\eta_R \tilde v \div b
		+ \tilde v b \cdot \nabla \eta_R) + \eta_R \tilde v c u \beta'(u) \,dx \,dt
		= \iTR \eta_R \tilde v f \beta'(u)  \,dx \,dt \,.
		\label{eq:post-eps-0}
	\end{align}
	To justify this, note that $\tilde u_\eps(t) \rightharpoonup^* \tilde u(t)$ and
	$\tilde v_\eps(t) \rightharpoonup^* \tilde v(t)$ in $L^\infty(\R^d)$ for every $t \in [0, T]$.
	Since $\tilde u_\eps(t) \to \tilde u(t)$ and $\tilde v_\eps(t) \to \tilde v(t)$
	in $L^r_\loc(\R^d)$ for all $r < \infty$, we also have
	$\tilde u_\eps(t) \tilde v_\eps(t) \rightharpoonup^* \tilde u(t) \tilde v(t)$ in $L^\infty(\R^d)$.
	Thus, for almost every $t \in [0, T]$, the integrals over
	$x$ converge as desired. To conclude the convergence of the space-time integrals, we apply
	dominated convergence to the time integral.

	From now on, assume that $\beta(0) = 0$. By letting $R \to \infty$
	in (\ref{eq:post-eps-0}), we obtain
	\begin{align}
		\begin{split}
			-\iTR
			 & \tilde u((c + \div b)v\beta'(v) + g\beta'(v)) \,dx \,dt
			+ \int_{\R^d} \tilde v(T) \tilde u(T) \,dx
			- \int_{\R^d} \tilde v(0) \tilde u(0) \,dx                           \\
			 & + \iTR \tilde u\tilde v \div b + \tilde v c u \beta'(u) \,dx \,dt
			= \iTR \tilde v f \beta'(u)  \,dx \,dt \,.
		\end{split}
		\label{eq:post-R-infty}
	\end{align}
	For all terms except the one with $\nabla \eta_R$, this is justified by dominated convergence,
	since
	\begin{align*}
		\tilde u, \in L^\infty_t (L^\infty_x \cap L^p_x),
		\qquad \tilde v \in L^\infty_t (L^\infty_x \cap L^q_x),
		\qquad \tilde u(0), \tilde u(T) \in L^p_x,
		\qquad \tilde v(0), \tilde v(T) \in L^q_x \,.
	\end{align*}
	Since $|\nabla \eta_R| \le \|\nabla \eta\|_{L^\infty(\R^d)}/R$ and
	$\supp \nabla \eta_R \subset B_{2R} \setminus B_R$, the $\nabla \eta_R$-term can be estimated by
	\begin{align*}
		\abs*{ \iTR \tilde u \tilde v b \cdot \nabla \eta_R \,dx \,dt}
		\le C \int_0^T\contractint \int_{B_{2R} \setminus B_R}
		|\tilde u \tilde v| \frac{|b|}{1 + |x|}\,dx \,dt \,.
	\end{align*}
	Since $u$ and $v$ have dual Hölder exponents and $\beta(0) = 0$,
	$\tilde u \tilde v \in L^\infty_t L^1_x \cap L^\infty_t L^\infty_x$.
	Thus, by the growth condition (\ref{eq:b-growth-condition}), the right-hand side
	converges to 0 as $R \to \infty$.

	Finally, we choose a sequence of admissible functions $\beta_n$ such that
	$\beta_n\to \id_{\R}$, $\beta_n' \to 1$ pointwise, $|\beta_n(\sigma)| \le |\sigma|$ and
	$|\beta_n'(\sigma)| \le 1$ for every $\sigma \in \R$.
	Taking $\beta = \beta_n$ in (\ref{eq:post-R-infty}) and
	using dominated convergence as $n \to \infty$, we get
	\begin{align*}
		\begin{split}
			-\iTR
			 & u((c + \div b)v + g) \,dx \,dt
			+ \int_{\R^d}  v(T)  u(T) \,dx
			- \int_{\R^d}  v(0)  u(0) \,dx           \\
			 & + \iTR  u v \div b +  v c u \,dx \,dt
			= \iTR v f  \,dx \,dt \,.
		\end{split}
	\end{align*}
	The $\div b$ and $c$ terms cancel out, so (\ref{eq:ibp}) follows.
\end{proof}

\begin{corollary}[Uniqueness for $p = \infty$]
	Assume (\ref{eq:b-growth-condition}) and
	\begin{align}
		b \in L^1(0, T; W^{1, 1}_\loc(\R^d; \R^d)),
		\qquad c, \, \div b \in L^1(0, T; L^\infty(\R^d))
		\tag{{$\dagger$}}
		\label{eq:bc-hypotheses}
	\end{align}
	If $u \in L^\infty([0, T] \times \R^d)$ is a solution of (\ref{eq:transport}) for $u^0 = 0$,
	then $u = 0$.
	\label{cor:uniqueness-infty}
\end{corollary}

\begin{proof}
	It suffices to show that for every $\phi \in C^\infty_c((0, T) \times \R^d)$
	\begin{align}
		\iTR u \phi \,dx \,dt = 0 \,.
		\label{eq:int-u-phi-0}
	\end{align}
	To do this, we consider a solution $v$ of
	\begin{align}
		\del_tv - b \cdot \nabla v - (c + \div b)v = \phi, \qquad v|_{t = T}= 0 \,.
		\label{eq:v-equation}
	\end{align}
	By Proposition \ref{prop:existence} and Remark \ref{remark:existence-rhs}, there exists a solution
	to this equation in $L^\infty([0, T] \times \R^d)$. An initial condition at time $T$ can be imposed
	by reversing time.

	To show that also $v \in L^\infty(0, T; L^1(\R^d))$, we look again at the proof of
	Proposition \ref{prop:existence}. We constructed a sequence of solutions
	$v_n \in L^\infty(0,T; C^1_b(\R^d))$ to
	\begin{align*}
		\del_t v_n - b_n \cdot \nabla v_n - (c_n + \div b_n)v_n = \phi_n \,,
		\quad v_n|_{t = 0} = v_n^0
	\end{align*}
	where
	\begin{align*}
		b_n \to b, \ \ c_n \to c, \ \ \div b_n \to \div b, \ \ \phi_n \to \phi
		\quad \text{in $L^1_\loc([0, T] \times \R^d)$}
	\end{align*}
	and $c_n, \, \div b_n$ are bounded in $L^1(0, T; L^\infty(\R^d))$. As indicated in Remark \ref{remark:existence-rhs},
	the proof of Proposition \ref{prop:existence} can be modified to incorporate a right-hand side by
	setting $\phi_n \colonequals (\phi * \rho_{1/n})\eta_n$, where $\rho_{1/n}$ and $\eta_n$ are defined in
	(\ref{eq:rho-eps}) and (\ref{eq:eta-R}).
	Then, we used the estimate (\ref{eq:u-Linfty-estimate}) to show
	that $v_n$ is bounded in $L^\infty([0, T] \times \R^d)$ and obtained
	$v$ as the weak-$*$ limit of a subsequence of $v_n$ in $L^\infty([0, T] \times \R^d)$.
	By Step 1 of Proposition \ref{prop:existence}, $v_n$ is also compactly supported. Thus,
	by (\ref{eq:ddt-Lp-norm}) and Grönwall's inequality
	\begin{align*}
		\norm*{v_n(t)}_{L^1(\R^d)}
		\le \exp\paren*{\norm*{c_n}_{L^1(0, T; L^\infty(\R^d))}}
		\norm*{\phi_n}_{L^1(0, T; L^1(\R^d))} \,.
	\end{align*}
	This shows that $v_n$ is bounded in $L^\infty(0, T; L^1(\R^d))$.

	Let $\psi \in L^1([0, T])$, $R > 0$ and $\zeta(t, x) \colonequals \psi(t) \chi_{B_R}(x) \sgn(v)$.
	Since $v_n \rightharpoonup^* v$ in $L^\infty([0, T] \times \R^d)$,
	\begin{align}
		\iTR v_n \zeta \,dx \,dt \ \xrightarrow[n \to \infty]{} \
		\iTR v \zeta \,dx \,dt = \int_0^T \!\!\psi(t) \int_{B_R} |v| \,dx \,dt
		\label{eq:vn-weak-star}
	\end{align}
	By Hölder's inequality,
	\begin{align*}
		\iTR v_n \zeta \,dx \,dt\le \norm*{v_n}_{L^\infty(0, T; L^1(\R^d))} \norm*{\psi}_{L^1([0, T])} \,.
	\end{align*}
	Thus, by letting $R \to \infty$ in (\ref{eq:vn-weak-star}) and using that $\psi \in L^1([0, T])$ is
	arbitrary, we conclude $v \in L^\infty(0, T; L^1(\R^d))$.

	By Proposition \ref{prop:consistency} (and Remark \ref{remark:renormalized-rhs}),
	both $u$ and $v$ are renormalized solutions.
	Thus (\ref{eq:int-u-phi-0}) follows from Theorem \ref{theorem:ibp}.
\end{proof}

\subsection{Applications to the 2d Euler equations}
\label{sec:euler}

We pause the general discussion and briefly discuss applications of the theory
developed so far to the two-dimensional Euler equations.
In vorticity formulation, the Euler equations read
\begin{align}
	\begin{split}
		\del_t \omega + b \cdot \nabla \omega & = 0          \hspace{4em} \text{on $(0, T) \times \R^2$} \\
		b                                     & = K * \omega
	\end{split}
	\label{eq:euler}
\end{align}
The convolution kernel $K : \R^2 \setminus \{0\} \to \R^2$ (also called the Biot-Savart kernel) is defined by
\begin{align*}
	K(x) \colonequals \frac 1 {2\pi |x|^2}
	\begin{pmatrix}
		-x_2 \\
		x_1
	\end{pmatrix} \,.
\end{align*}
We assume that $\omega \in L^\infty(0, T; L^p(\R^2))$ for some $1 \le p \le \infty$
and $b \in L^1(0, T; L^q_\loc(\R^2))$, where $1/p + 1/q = 1$.
Then, the transport equation in (\ref{eq:euler}) can be understood in the weak sense.\footnote{
	If $4/3 \le p < 2$, then $\omega \in L^\infty(0, T; L^p(\R^2))$ already
	implies $b \in L^1(0, T; L^q_\loc(\R^2))$.
	Since $K$ is homogeneous of degree $-1$, by the Hardy-Littlewood-Sobolev theorem
	$b \in L^\infty(0, T; L^r(\R^2))$, where $1/r = 1/p - 1/2$ (\cite{Ste70}, p.~119, Theorem 1).
	However, for $p > 2$, such a priori estimates are not sufficient to conclude even
	$K * \omega \in L^1_\loc([0, T] \times \R^2)$. Consider for example
	$\omega(t, x) = \chi_{x_1 > 0}/(1 + |x|)$.
}

Since $\nabla K$ is homogeneous of degree $-2$, smooth on $\R^2 \setminus \{0\}$, and
$\int_{S^1} \nabla K \, d\mathcal H^1 = 0$,
by the Calderón-Zygmund theory for singular integrals,
the map $f \mapsto \nabla K * f$ is a bounded linear operator
$L^p(\R^2) \to L^p(\R^2; \R^{2 \times 2})$ if $1 < p < \infty$ (see \cite{Ste70}, p.~39, Theorem 3).
Thus $\nabla b \in L^\infty(0, T; L^p(\R^2))$ for such $p$.

If $p \ge 2$, this implies $b \in L^1(0, T; W^{1, q}_\loc)$. Combining this fact with
part 2) of Proposition \ref{prop:consistency} yields the following theorem.

\begin{theorem}
	Let $2 \le p < \infty$, and let $\omega \in L^\infty(0, T; L^p(\R^2))$ and $b \in L^1(0, T; L^q_\loc(\R^2))$
	satisfy the Euler equations (\ref{eq:euler}) in the weak sense. Then $\omega$ is a renormalized
	solution.
	\label{theorem:euler}
\end{theorem}

The result in Theorem \ref{theorem:euler} has been observed previously as a consequence
of Theorem II.3 of \cite{DL89}, see Proposition 1 of \cite{FML05}.

For $1 < p < 2$, it was shown in \cite{CS15} that solutions of (\ref{eq:euler})
obtained as a limit in the vanishing viscosity approximation are renormalized.
The argument in \cite{CS15} uses the integration by parts formula (\ref{eq:ibp}).
They show that every such solution satisfies, for all $\chi \in C^\infty_c((0, T) \times \R^2)$,
\begin{align}
	\int_0^T \contractint \int_{\R^2} \chi \omega \,dx \,dt
	= \int_{\R^2} \phi(0, x) \omega(0, x) \,dx \,,
	\label{eq:chi-phi}
\end{align}
where $\phi \in L^\infty(0, T; L^1 \cap L^\infty)$ is the renormalized solution of
\begin{align*}
	\del_t \phi + b \cdot \nabla \phi = -\chi, \qquad \phi|_{t = T} = 0 \,.
\end{align*}
Equation (\ref{eq:chi-phi}) may be used to deduce that $\omega$ is renormalized.
Note that $\div b = 0$, since $\div K = 0$.
Thus by Theorem \ref{theorem:ibp}, (\ref{eq:chi-phi}) also holds for the renormalized
solution of (\ref{eq:euler}). Since $\chi \in C^\infty_c((0, T) \times \R^2)$ is arbitrary,
$\omega$ equals the renormalized solution.

\subsection{Stability}
\label{sec:stability}

\subsubsection{The case $c = 0$}

In this section, we prove stability results for renormalized solutions
of the transport equation. We consider the case $c = 0$ here
and more general $c$ in §\ref{sec:stab-general-c}.
A consequence of the stability results will be the existence of renormalized
solutions for $u^0 \in L(\R^d)$, see Corollary \ref{cor:existence}.

\begin{theorem}[Stability]
	Assume $b$ satisfies (\ref{eq:bc-hypotheses}) and (\ref{eq:b-growth-condition}) with $c = 0$.
	For $n \in \N$, let $b_n \in L^1_\loc([0, T] \times \R^d; \R^d)$ and $c_n \in L^1_\loc([0, T] \times \R^d)$
	such that
	\begin{align}
		b_n\to b \quad \text{in $L^1_\loc([0, T] \times \R^d; \R^d)$},
		\qquad \div b_n\to \div b, \ c_n\to 0 \quad \text{in $L^1_\loc([0, T] \times \R^d)$}
		\label{eq:stability-conv-assumptions}
	\end{align}
	as $n \to \infty$.
	Let $u_n^0, u^0 \in L(\R^d)$ and let $u_n \in L([0, T] \times \R^d)$ be renormalized
	solutions of (\ref{eq:transport}) with $(b_n, c_n, u_n^0)$.
	\begin{enumerate}
		\item If $u_n^0 \to u^0$ in $L(\R^d)$, then there exists a renormalized solution
		      $u \in C([0, T]; L(\R^d))$ for $(b, 0, u^0)$, and $u_n \to u$ in
		      $C([0, T]; L(\R^d))$.
		\item Let $1 \le p < \infty$ and $u_n \in L^\infty(0, T; L^p(\R^d))$. Assume that
		      $c_n, \, \div b_n \in L^1(0, T; L^\infty(\R^d))$ and that (\ref{eq:b-growth-condition}) holds for all $b_n$.
		      If $u_n^0 \to u^0$ in $L^p(\R^d)$ and there exists a function $g \in L^1((0, T))$
		      such that
		      \begin{align}
			      \hspace{4em}
			      \sup_{n \in \N} \, \norm*{\div b_n(t) + p c_n(t)}_{L^\infty(\R^d)} \le g(t)
			      \qquad \text{for a.e. $t \in [0, T]$,}
			      \label{eq:majorant}
		      \end{align}
		      then $u_n \to u$ in $C([0, T]; L^p(\R^d))$.
	\end{enumerate}
	\label{theorem:stability}
\end{theorem}

\begin{proof}
	This proof corresponds to the proofs of Theorems II.3 and II.4 in \cite{DL89}.
	For the first four steps of the proof, we assume the hypotheses of part 1)
	of Theorem \ref{theorem:stability}.

	\paragraph{Step 1: Weak-$*$ convergence of $\beta(u_n)$.}

	Let $\beta$ be an admissible function and set
	\begin{align*}
		v_n= \beta(u_n), \quad v_n^0= \beta(u_n^0), \quad v^0= \beta(u^0) \,.
	\end{align*}
	We first show that $v_n$ converges weak-$*$ in
	$L^\infty([0, T] \times \R^d)$ to the unique solution of (\ref{eq:transport})
	with $c = 0$ and initial data $v^0$. Since $v_n$ is bounded in
	$L^\infty([0, T] \times \R^d)$, there exists a weak-$*$ convergent
	subsequence $v_{k_n}$, whose limit we denote by $v$.

	Since $u_{k_n}$ is a renormalized solution, $v_{k_n}$ satisfies
	\begin{align*}
		\del_tv_{k_n} - b_{k_n} \! \cdot \nabla v_{k_n}+ c_{k_n}u_{k_n}\beta'(u_{k_n}) = 0
		\quad \text{in $(0, T) \times \R^d$},
		\qquad v_{k_n}|_{t = 0}= v_{k_n}^0 \quad \text{in $\R^d$}\,.
	\end{align*}
	This means that for every $\phi \in C^\infty_c([0, T) \times \R^d)$
	\begin{align*}
		-\iTR & v_{k_n} \del_t\phi \,dx \,dt -\int_{\R^d}v^0_{k_n}\phi(0, x) \,dx    \\
		      & + \iTR v_{k_n} (\phi \div b_{k_n} \! + b_{k_n} \! \cdot \nabla \phi)
		+ c_{k_n}u_{k_n} \beta'(u_{k_n}) \, \phi \, dx \,dt = 0 \,.
	\end{align*}
	By letting $n \to \infty$ in this equation, we now show that $v$
	satisfies (\ref{eq:transport-weak}) with $c = 0$ and initial data
	$v^0$. By (\ref{eq:stability-conv-assumptions}) and the boundedness of
	$u_{k_n}\beta'(u_{k_n})$ in $L^\infty([0, T] \times \R^d)$,
	the first and third integral converge as desired. For the second integral,
	note that $v_{k_n}^0, v^0 \le \|\beta\|_{L^\infty(\R)}$.
	Therefore, for every $\eps \ge 0$
	\begin{align*}
		\int_{\R^d} |v_{k_n}^0 - v^0| |\phi(0, x)| \,dx
		\le \|\phi(0, \cdot)\|_{L^1(\R^d)}
		\paren*{
			\eps + 2\|\beta\|_{L^\infty(\R)} \,
			\meas^d(\{|v_{k_n}^0 - v^0| \ge \eps\} \cap \supp \phi)
		}
	\end{align*}
	Since $v_{k_n}^0\to v^0$ in $L(\R^d)$ by the
	Lipschitz-continuity of $\beta$, the right-hand side can be made arbitrarily small
	by choosing $\eps$ appropriately. Hence $v$ is a solution of the transport
	equation with $(b, 0, v^0)$.

	So far, we have only shown the weak-$*$ convergence of the subsequence
	$v_{k_n}$ to $v$. However, by uniqueness (Corollary \ref{cor:uniqueness-infty}) and the above
	argument, every subsequence of $v_n$ has a subsequence
	converging to $v$ weak-$*$ in
	$L^\infty([0, T] \times \R^d)$, which implies the convergence
	of the whole sequence.

	\paragraph*{Step 2: $L^2_\loc([0, T] \times \R^d)$ convergence of $\beta(u_n)$.}

	We now show that $v_n\to v$ in $L^2_\loc([0, T] \times \R^d)$. Let
	$K \subset [0, T] \times \R^d$ be compact. By Step 1,
	$v_n\rightharpoonup v$ in $L^2(K)$, so it suffices to show
	$\|v_n\|_{L^2(K)}\to \|v\|_{L^2(K)}$, because $L^2(K)$ is uniformly convex.

	Let $w$ be the unique solution of (\ref{eq:transport}) with $(b, 0, \beta(u^0)^2)$.
	By Step 1 with $\beta$ replaced by $\beta^2$, which is also admissible,
	\begin{align*}
		v_n^2\rightharpoonup^{*}w \quad \text{in $L^\infty([0, T] \times \R^d)$}\,.
	\end{align*}
	By Proposition \ref{prop:consistency}, $v$ is renormalized, hence $v^2$
	is also a solution of (\ref{eq:transport}) with $c = 0$ and initial data $\beta(u^0)^2$.
	Note that $v$ is bounded, so we can choose a suitably cut-off
	version of $x^2$ as an admissible function. Thus $v^2 = w$ by Corollary
	\ref{cor:uniqueness-infty}. We conclude $\|v_n\|_{L^2(K)}\to \|v\|_{L^2(K)}$,
	because $\meas^{d + 1}(K) < \infty$.

	\paragraph{Step 3: Uniformity in time.}

	The next goal is to show $v_n\to v$ in
	$C([0, T]; L^2_\loc(\R^d))$. We first show that for every $\xi
		\in C^\infty_c(\R^d)$
	\begin{align}
		\int_{\R^d}v_n(t, x) \xi(x) \,dx \to \int_{\R^d} v(t, x)\xi(x) \,dx
		\qquad \text{in $W^{1, 1}((0, T))$.}
		\label{eq:weak-conv-w11}
	\end{align}
	Since $u_n$ is a renormalized solution, for every $\psi \in C^\infty_c((0, T))$
	\begin{align*}
		\iTR v_n\xi \del_t\psi(t) \,dx
		= \iTR (v_n(\xi \div b_n+ b_n \cdot \nabla \xi)
		+ c_nu_n\beta'(u_n)\xi) \psi(t) \,dx \,dt \,.
	\end{align*}
	Because $\psi$ is arbitrary, we deduce that
	\begin{align}
		\frac{d}{dt} \int_{\R^d}v_n\xi \,dx
		= -\int_{\R^d}v_n(\xi \div b_n+ b_n\cdot \nabla \xi)
		+ c_nu_n\beta'(u_n)\xi \,dx
		\label{eq:dt-int-vn}
	\end{align}
	in the sense of distributions on $(0, T)$. However, the right-hand side
	is in $L^1((0, T))$, so (\ref{eq:dt-int-vn}) shows
	$\int_{\R^d}v_n\xi \,dx \in W^{1, 1}((0, T))$. Similarly,
	\begin{align}
		\frac{d}{dt} \int_{\R^d}v \xi \,dx
		= -\int_{\R^d}v(\xi \div b
		+ b \cdot \nabla \xi) \, dx \,.
		\label{eq:dt-int-v}
	\end{align}
	Thus, to prove (\ref{eq:weak-conv-w11}), by (\ref{eq:dt-int-vn}) and (\ref{eq:dt-int-v})
	it suffices to show that
	\begin{align*}
		\iTR \xi (v_n\div b_n - v \div b) + (v_nb_n - vb) \cdot \nabla \xi
		+ c_nu_n\beta'(u_n) \xi \,dx \,dt \, \to \, 0 \,.
	\end{align*}
	The integral of $c_nu_n\beta'(u_n) \xi$ converges to zero since $c_n\to 0$ in
	$L^1_\loc([0, T] \times \R^d)$ and $u_n\beta'(u_n)$ is bounded in
	$L^\infty([0, T] \times \R^d)$. The first two terms converge to 0 since
	$b_n \to b$ and $\div b_n \to \div b$ in $L^1_\loc([0, T] \times \R^d)$ and
	$v_n \rightharpoonup^* v$ in $L^\infty([0, T] \times \R^d)$.

	By Proposition \ref{prop:continuity}, $v \in C([0, T]; L^2_\loc(\R^d))$.
	Hence, to conclude $v_n\to v$ in $C([0, T]; L^2_\loc(\R^d))$, it suffices
	to show that for every convergent sequence $t_n\to t$ in $[0, T]$
	\begin{align}
		v_n(t_n, \cdot) \to v(t, \cdot) \qquad \text{in $L^2_\loc(\R^d)$.}
		\label{eq:vn-tn-to-v-t}
	\end{align}
	Fix such a sequence $t_n$. Because $W^{1, 1}((0, T))$ embeds
	continuously into $C([0, T])$, (\ref{eq:weak-conv-w11}) implies for every
	$\xi \in C^\infty_c(\R^d)$
	\begin{align}
		\int_{\R^d}v_n(t_n, x) \xi(x) \,dx \to
		\int_{\R^d}v(t, x)\xi(x) \,dx
		\qquad \text{as $n \to \infty$.}
		\label{eq:vn-tn}
	\end{align}
	Let $K \subset \R^d$ be compact.
	Since $v_n(t_n)$ is bounded in $L^2(K)$, (\ref{eq:vn-tn}) implies
	$v_n(t_n) \rightharpoonup v(t)$ in $L^2(K)$. Since $\beta^2$ is also
	an admissible function, we also have $v_n(t_n)^2 \rightharpoonup v(t)^2$ in $L^2(K)$.
	Since $\chi_{K} \in L^2(K)$, we conclude
	$\|v_n(t_n, \cdot)\|_{L^2(K)}\to \|v(t, \cdot)\|_{L^2(K)}$, and hence
	$v_n(t_n, \cdot) \to v(t, \cdot)$ in $L^2(K)$.

	\paragraph{Step 4: Part 1).}

	Choose $\beta = \tanh$ and let $v$ be the limit of $\beta(u_n)$ in
	$C([0, T]; L^2_\loc(\R^d))$. Define $u = \beta^{-1}(v)$. To make
	sense of this definition, we now show that
	\begin{align}
		m^{d}(\{x \in \R^d: |v(t, x)| = 1\}) = 0 \quad \text{for every $t \in [0, T]$}\,.
		\label{eq:v-less-1}
	\end{align}
	For every $t \in [0, T]$, we already know that $|v(t, x)| \le 1$
	for a.e. $x \in \R^d$, since $v$ is the weak-$*$ limit of
	$v_n$ in $L^\infty([0, T] \times \R^d)$
	and $v \in C([0, T]; L^2_\loc(\R^d))$.
	To prove (\ref{eq:v-less-1}), for $k \in \N$, let $\theta_k \in C^\infty_c(\R)$ be such that
	\begin{align*}
		0 \le \theta_k \le 1, \quad \theta_k|_{[1 - 1/k, 1]} = 1,
		\quad \supp \theta_k \subset (1 - 2/k, \infty) \,.
	\end{align*}
	Define $\alpha_k(\sigma) = \theta_k(|\sigma|)$ for $\sigma \in \R$ and
	note that $\alpha_k$ is admissible.
	Since $|v^0(x)| < 1$ for all $x \in \R^d$, $\alpha_k(v^0) \to 0$ in
	$L(\R^d)$ as $k \to \infty$. Furthermore, since $v$ is renormalized,
	$\alpha_k(v)$ is the unique renormalized solution of (\ref{eq:transport}) with
	$c = 0$ and initial data $\alpha_k(v^0)$. By applying Steps 1-3 to the sequence
	$\alpha_k(v)$, $b_k = b$ and $c_k = 0$ for $k \in \N$,
	we deduce that for every admissible function $\gamma$,
	\begin{align*}
		\gamma(\alpha_k(v)) \to \gamma(0)
		\quad \text{in $C([0, T]; L^2_\loc(\R^d))$ as $k \to \infty$.}
	\end{align*}
	But $\alpha_k$ only takes values in $[0, 1]$, hence we may choose
	$\gamma$ such that $\gamma|_{[0, 1]} = \id$. Thus
	\begin{align}
		\alpha_k(v) \to 0
		\quad \text{in $C([0, T]; L^2_\loc(\R^d))$ as $k \to \infty$.}
		\label{eq:alpha-to-0}
	\end{align}
	By the definition of $\alpha_k$, this implies (\ref{eq:v-less-1}).

	We now show $u_n\to u$ in $C([0, T]; L(\R^d))$. Let
	$K \subset \R^d$ be compact and $\delta > 0$. By (\ref{eq:alpha-to-0}), there
	exists an $M \in \N$ such that
	\begin{align*}
		\sup_{t \in [0, T]}\meas^{d}(\{x \in K : |v(t, x)| \ge 1 - 1/M\}) < \delta \,.
	\end{align*}
	Let $\eps > 0$. Then
	\begin{align}
		\sup_{t \in [0, T]}\meas^{d}
		 & (\{x \in K : |u_n(t, x) - u(t, x)| \ge \eps\})
		\label{eq:uniform-meas}                           \\
		 & \le \sup_{t \in [0, T]}
		\meas^{d}\big(\{ x \in K : |u_n(t, x) - u(t, x)| \ge \eps \ \text{and}
		\ |v(t, x)| < 1 - 1/M\}\big) + \delta \,. \notag
	\end{align}
	The supremum on the right-hand side converges to 0 as
	$n \to \infty$, since $\beta^{-1}$ is Lipschitz-continuous on any compact
	interval $I \subset (-1, 1)$. Explicitly, let
	$\lambda = \beta(\beta^{-1}(1 - 1/M) + \eps)$ and $L$ be the Lipschitz constant of
	$\beta^{-1}|_{[-\lambda, \lambda]}$. If $x \in K$ and $|v(t, x)| < 1 - 1/M$, then
	\begin{align*}
		\{|u_n(t, x) - u(t, x)| \ge \eps\} \quad \text{implies}\quad \{|v_n(t, x) - v(t, x)| \ge \eps/L\} \,.
	\end{align*}
	Since $v_n\to v$ in $C([0, T]; L(\R^d))$ and $\delta > 0$ is arbitrary, we conclude
	$u_n\to u$ in $C([0, T]; L(\R^d))$.

	To prove that $u$ is a renormalized solution of (\ref{eq:transport})
	with $c = 0$ and initial data $u^0$, let $\beta$ be an admissible
	function. Because $\beta$ is Lipschitz-continuous, $\beta(u_n)$ converges
	to $\beta(u)$ in $C([0, T]; L(\R^d))$. By Step 3,
	$\beta(u_n)$ also converges in $C([0, T]; L^2_\loc(\R^d))$ to the unique renormalized
	solution of (\ref{eq:transport}) with $c = 0$ and initial data $\beta(u^0)$.

	\paragraph{Step 5: Part 2).}

	By (\ref{eq:majorant}) and (\ref{eq:ddt-Lp-norm}), $u_n$ is bounded in $L^\infty(0, T; L^p(\R^d))$.
	By Step 4, $u$ is the limit of $u_n$ in $C([0,T]; L(\R^d))$, hence $u$ belongs to
	$L^\infty(0, T; L^p(\R^d))$. Thus $b$ and $u$ satisfy the hypothesis of part 2) of
	Proposition \ref{prop:continuity}, which implies that $u \in C([0, T]; L^p(\R^d))$.
	So, we only need to show that
	$u_n(t_n) \to u(t)$ in $L^p(\R^d)$ for every convergent sequence $t_n \to t$ in $[0, T]$.
	Since $u_n \to u$ in $C([0, T]; L(\R^d))$, we may assume that $u_n(t_n)$
	converges to $u(t)$ almost everywhere on $\R^d$.
	Therefore, by Lemma \ref{lemma:pointwise-ae-Lp}, it suffices to show that
	$\|u_n(t_n)\|_{L^p(\R^d)} \to \|u(t)\|_{L^p(\R^d)}$ as $n \to \infty$.

	By Fatou's Lemma,
	\begin{align}
		\liminf_{n \to \infty} \int_{\R^d} |u_n(t_n)|^p \,dx \ge \int_{\R^d} |u(t)|^p \,dx \,.
		\label{eq:un-tn-fatou}
	\end{align}
	For the other inequality, let $a \in W^{1, 1}((0, T))$.
	By (\ref{eq:ddt-Lp-norm}), for a.e. $t \in (0, T)$
	\begin{align}
		\frac{d}{dt} \paren*{e^{-a(t)} \int_{\R^d} |u(t)|^p \,dx }
		= -e^{-a(t)} \int_{\R^d} (\div b(t) + a'(t))|u(t)|^p \,dx \,.
		\label{eq:ddt-e-at}
	\end{align}
	Integrating (\ref{eq:ddt-e-at}) from 0 to $t$ yields
	\begin{align}
		e^{-a(t)} \int_{\R^d} |u(t)|^p \, dx
		= \int_{\R^d} |u^0|^p \,dx
		- \int_0^t \contractintsmall \int_{\R^d} (\div b + a') e^{-a(s)} |u|^p \,dx \,ds \,.
		\label{eq:e-a-u-t}
	\end{align}
	Since the assumptions of Proposition \ref{prop:continuity} are also satisfied for $u_n$,
	(\ref{eq:ddt-e-at}) also holds with $u_n, b_n, c_n$ instead of $u, b, 0$.
	We integrate from 0 to $t_n$:
	\begin{align}
		e^{-a(t_n)} \int_{\R^d} |u_n(t_n)|^p \, dx
		= \int_{\R^d} |u_n^0|^p \,dx
		- \int_0^{t_n} \contractintsmall \int_{\R^d} (\div b_n + pc_n + a') e^{-a(s)} |u_n|^p \,dx \,ds
		\label{eq:e-a-un-tn}
	\end{align}
	Since $b_n \to b$ and $c_n \to 0$ in $L^1_\loc([0, T] \times \R^d)$, after extracting
	a subsequence, we may assume they converge pointwise almost everywhere on $[0, T] \times \R^d$.
	We now choose $a$ such that $a$ is continuous, $a(0) = 0$ and $a' = g$, where $g \in L^1((0, T))$
	is from the assumptions of Theorem \ref{theorem:stability}, part 2). Due to our choice of $a'$,
	the integrand on the right-hand side in (\ref{eq:e-a-un-tn}) is nonnegative. Thus, by
	the assumption that $u^0_n \to u^0$ in $L^p(\R^d)$ and another application of Fatou's Lemma,
	\begin{align}
		\begin{split}
			\limsup_{n \to \infty}
			 & \paren*{ e^{-a(t_n)} \int_{\R^d} |u_n(t_n)|^p \,dx}   \\
			 & \ = \int_{\R^d} |u^0|^p \,dx - \liminf_{n \to \infty}
			\int_0^{t_n} \contractintsmall \int_{\R^d}
			(\div b_n + pc_n + a') e^{-a(s)} |u_n|^p \,dx \,ds       \\
			 & \ \le  e^{-a(t)} \int_{\R^d} |u(t)|^p \,dx \,.
		\end{split}
		\label{eq:un-tn-limsup}
	\end{align}
	Since $a$ is continuous, the left-hand side is equal to
	$e^{-a(t)}\limsup_{n \to \infty} \int_{\R^d} |u_n(t_n)|^p \,dx$.
	Thus we conclude $\|u_n(t_n)\|_{L^p(\R^d)} \to \|u(t)\|_{L^p(\R^d)}$.
\end{proof}

\subsubsection{More general $c$}
\label{sec:stab-general-c}

We now prove stability results for
$c \in L^1(0, T; L^\infty(\R^d))$ by reducing to the case $c = 0$.
Let $a \in L^1(0, T; L^\infty(\R^d))$ and consider the equation
\begin{align*}
	\del_t\Phi - b \cdot \nabla \Phi = a, \qquad\qquad \Phi|_{t = 0}= 0 \,.
\end{align*}
By Proposition \ref{prop:existence} and Remark \ref{remark:existence-rhs},
if $b \in L^1_\loc([0, T] \times \R^d; \R^d)$ and
$\div b \in L^1(0, T; L^\infty(\R^d))$, there exists a solution $\Phi \in L^\infty([0, T] \times \R^d)$
to this equation. If $b \in L^1(0, T; W^{1, 1}_\loc(\R^d))$,
then $\Phi$ is renormalized by Proposition \ref{prop:consistency}.

\begin{lemma}
	Assume $b \in L^1(0, T; W^{1, 1}_\loc(\R^d))$ and $c \in L^1(0, T; L^1_\loc(\R^d))$.
	Let $u^0\in L(\R^d)$. Then, $u \in L([0, T] \times \R^d)$ is a renormalized solution
	of (\ref{eq:transport}) with $(b, c, u^0)$ if and only if $e^{-\Phi}u$ is a
	renormalized solution of (\ref{eq:transport}) with $(b, c + a, u^0)$.
	\label{lemma:c+a}
\end{lemma}

\begin{proof}
	This proof is from \cite{DL89}, p.~526, Lemma II.2.

	It suffices to show just one direction due to the symmetry of the statement.
	Assume that $u$ is a renormalized solution of the transport equation for
	the initial condition $u^0$.

	Formally, the lemma would follow from the chain rule. To make this
	argument rigorous, we invoke the regularization result, Lemma \ref{lemma:regularization}.
	Let $\beta$ be an admissible function and set $v = \beta(u)$. Then
	by Lemma \ref{lemma:regularization}
	\begin{align}
		\begin{split}
			\begin{alignedat}{2}
				\del_tv_\eps- b \cdot \nabla v_\eps+ cu \beta'(u)
				 & = r_\eps     & \qquad v_\eps|_{t = 0}    & = v_\eps^0 \\
				\del_t\Phi_\eps- b \cdot \nabla \Phi_\eps
				 & = a + s_\eps & \qquad \Phi_\eps|_{t = 0} & = 0 \,
			\end{alignedat}
		\end{split}
		\label{eq:v-phi-regularized}
	\end{align}
	Here, $v_\eps= v * \rho_\eps, \, v_\eps^0= \beta(u^0) * \rho_\eps, \, \Phi_\eps = \Phi * \rho_\eps$
	and $r_\eps, s_\eps\to 0$ in $L^1_\loc([0, T] \times \R^d)$. Let $\gamma$ be an admissible function
	and define $\omega^\eps \colonequals \gamma\big(e^{-\Phi_\eps} v_\eps\big)$.
	By the product and chain rules for Sobolev functions,
	\begin{align*}
		\del_t \omega^\eps
		 & = \gamma'{\big(e^{-\Phi_\eps} v_\eps\big)}
		e^{-\Phi_\eps} \big({-}\del_t \Phi_\eps v_\eps+ \del_t v_\eps\big) \\
		 & = \gamma'{\big(e^{-\Phi_\eps} v_\eps\big)}
		e^{-\Phi_\eps} \big({-}(b \cdot \nabla \Phi_\eps + a + s_\eps)v_\eps
		+ b \cdot \nabla v_\eps- c u \beta'(u) + r_\eps\big) \,.
	\end{align*}
	Even though the derivative of the exponential function is not bounded on $\R$, the chain rule
	for Sobolev functions applies here since $\Phi_\eps \in L^\infty([0, T] \times \R^d)$.
	Since
	\begin{align*}
		\nabla \omega^\eps = \gamma'{\big(e^{-\Phi_\eps} v_\eps\big)}
		e^{-\Phi_\eps} \big({-}\nabla \Phi_\eps + \nabla v_\eps \big) \,,
	\end{align*}
	we have
	\begin{align}
		\del_t \omega^\eps - b \cdot \nabla \omega^\eps
		= \gamma'{\big(e^{-\Phi_\eps} v_\eps\big)} e^{-\Phi_\eps}
		\big({-}(a + s_\eps)v_\eps - cu\beta'(u) + r_\eps\big) \,.
		\label{eq:phi-eps-v-eps}
	\end{align}
	Let $\phi \in C^\infty_c([0, T) \times \R^d)$. By integrating (\ref{eq:phi-eps-v-eps})
	over $[0, T] \times \R^d$ against $\phi$ and using integration by parts, we get
	\begin{align}
		\begin{split}
			-  \iTR \gamma{\big(e^{-\Phi_\eps}v_\eps\big)} & \del_t \phi \,dx \,dt
			- \int_{\R^d}\gamma{\big(v_\eps^0\big)} \phi(0, x) \,dx
			+ \iTR \gamma{\big(e^{-\Phi_\eps}v_\eps\big)} \div(b\phi)   \,dx \,dt  \\
			                                               & = \iTR
			\gamma'{\big(e^{-\Phi_\eps} v_\eps\big)} e^{-\Phi_\eps}
			\big({-}(a + s_\eps) v_\eps - cu\beta'(u) + r_\eps \big) \phi \,dx \,dt \,.
		\end{split}
		\label{eq:v-phi-after-ibp}
	\end{align}
	We let $\eps \to 0$ in (\ref{eq:v-phi-after-ibp}). Since $r_\eps, s_\eps \to 0$ in
	$L^1_\loc([0, T] \times \R^d)$ and $e^{-\Phi_\eps}, v_\eps$ are bounded uniformly in $\eps$,
	\begin{align*}
		\iTR \gamma'{\big(e^{-\Phi_\eps} v_\eps\big)} e^{-\Phi_\eps} s_\eps v_\eps \phi \,dx \,dt \to 0,
		\qquad
		\iTR \gamma'{\big(e^{-\Phi_\eps} v_\eps\big)} e^{-\Phi_\eps} r_\eps \phi \,dx \,dt \to 0
	\end{align*}
	as $\eps \to 0$.
	For all other terms, we notice that $v_\eps \to v$ and $\Phi_\eps \to \Phi$ almost everywhere
	on $[0, T] \times \R^d$. Thus, by the dominated convergence theorem
	\begin{align}
		\begin{split}
			- \iTR \gamma{\big(e^{-\Phi}v\big)}
			 & \del_t\phi \,dx \,dt
			- \int_{\R^d} \gamma{\big(v^0\big)} \phi(0, x) \,dx
			+ \iTR \gamma{\big(e^{-\Phi} v\big)} \div(b\phi) \,dx \,dt                                   \\
			 & + \iTR \gamma'{\big(e^{-\Phi} v\big)} e^{-\Phi} (av + cu\beta'(u)) \phi \,dx \,dt = 0 \,.
		\end{split}
		\label{eq:beta-v}
	\end{align}

	Next, we choose a sequence of admissible functions $\beta_n$
	such that $\beta_n \to \id_\R$, $\beta_n' \to 1$ pointwise and
	\begin{align*}
		|\beta_n(\sigma)| \le |\sigma|, \quad |\sigma \beta_n'(\sigma)| \le |\beta_n(\sigma)|
		\qquad \text{for every $\sigma \in \R$.}
	\end{align*}
	For instance, $\beta_n(\sigma) = n\tanh(\sigma/n)$ has these properties. Since $v = \beta(u)$,
	by choosing $\beta = \beta_n$ in (\ref{eq:beta-v}) and letting $n \to \infty$, we get
	\begin{align}
		\begin{split}
			- \iTR \gamma \paren*{e^{-\Phi} u}
			 & \del_t \phi \,dx \,dt
			- \int_{\R^d}\gamma(u^0) \phi(0, x) \,dx
			+ \iTR \gamma(e^{-\Phi}u) \div (b\phi)        \,dx \,dt                         \\
			 & + \iTR \gamma'\paren*{e^{-\Phi}u} e^{-\Phi} u (a + c) \phi \,dx \,dt = 0 \,.
		\end{split}
		\label{eq:beta-u-c+a}
	\end{align}
	This is justified by the dominated convergence theorem, using the fact that
	$\gamma(\sigma)$ and $\gamma'(\sigma) (1 + |\sigma|)$ are bounded on $\R$. To get an integrable majorant
	for the term with $cu\beta'(u)$, we use that $|\sigma\beta_n'(\sigma)| \le |\beta(\sigma)|$:
	\begin{align*}
		\abs*{\gamma'\paren*{e^{-\Phi} \beta_n(u)} e^{-\Phi} cu\beta_n'(u)}
		\le \abs*{\gamma'\paren*{e^{-\Phi} \beta_n(u)} e^{-\Phi} \beta_n(u) \, c}
		\le C_\gamma |c| \,.
	\end{align*}
	Since $\gamma$ is an arbitrary admissible function, (\ref{eq:beta-u-c+a}) implies that
	$e^{-\Phi} u$ is a renormalized solution of the transport equation with $(b, c + a, u^0)$.
\end{proof}

The following is a stability result for $c \in L^1(0, T; L^\infty(\R^d))$. The hypotheses
are the same as in Theorem \ref{theorem:stability}, except that we
require $b_n \in L^1(0, T; W^{1, 1}_\loc(\R^d))$ in order to use Lemma \ref{lemma:c+a}.

\begin{corollary}
	Assume $b$ and $c$ satisfy (\ref{eq:bc-hypotheses}) and (\ref{eq:b-growth-condition}). We keep
	the notations and assumptions of Theorem \ref{theorem:stability}, except that we replace
	$c_n \to 0$ by $c_n \to c$ in $L^1_\loc([0, T] \times \R^d)$, and assume in addition
	that $b_n \in L^1(0, T; W^{1, 1}_\loc(\R^d))$ for every $n \in \N$.
	\begin{enumerate}
		\item If $u_n^0 \to u^0$ in $L(\R^d)$, then there exists a renormalized solution
		      $u \in C([0, T]; L(\R^d))$ for $(b, c, u^0)$, and $u_n \to u$ in
		      $C([0, T]; L(\R^d))$.
		\item Part 2) of Theorem \ref{theorem:stability} remains true under its additional
		      assumptions.
	\end{enumerate}
	\label{cor:stability-c-ne-0}
\end{corollary}

\begin{proof}
	Let $\Phi \in L^\infty([0, T] \times \R^d)$ be a solution to
	\begin{align*}
		\del_t\Phi - b \cdot \nabla \Phi = -c, \qquad \Phi|_{t = 0}= 0 \,.
	\end{align*}
	By Proposition \ref{prop:existence}, such a $\Phi$ exists.
	By Lemma \ref{lemma:c+a}, $\tilde u_n\colonequals e^{-\Phi}u_n$ is
	a renormalized solution of
	\begin{align*}
		\del_t\tilde u_n - b_n \cdot \nabla \tilde u_n
		+ (c_n- c) \tilde u_n = 0, \qquad \tilde u_n|_{t = 0}= u_n^0\,.
	\end{align*}
	By part 1) of Theorem \ref{theorem:stability}, $\tilde u_n$ converges in
	$C([0, T]; L(\R^d))$ to a renormalized solution $\tilde u$ of
	(\ref{eq:transport}) with $c = 0$ and initial data $u^0$. Using again Lemma
	\ref{lemma:c+a}, we deduce that $u \colonequals e^{\Phi}\tilde u$ is a renormalized solution of
	\begin{align*}
		\del_t u - b \cdot \nabla u + c u = 0, \qquad u|_{t = 0}= u^0\,.
	\end{align*}
	To conclude the proof of part 1) of Corollary \ref{cor:stability-c-ne-0},
	we need to show that $e^{\Phi}\tilde u_n\to e^{\Phi}\tilde	u$ in $C([0, T]; L(\R^d))$.
	Since $\Phi \in L^\infty([0, T] \times \R^d)$, we also have $e^{\Phi}\in L^\infty([0, T] \times \R^d)$.
	Thus for every compact $K \subset \R^d$ and $\eps > 0$
	\begin{align*}
		\sup_{t \in [0, T]}\meas^{d}
		 & \paren*{
			\braces*{
				x \in K : |e^{\Phi(t, x)} \tilde u_n(t, x) - e^{\Phi(t, x)} \tilde u(t, x)| \ge \eps
			}
		}                                            \\
		 & \le \sup_{t \in [0, T]} \meas^{d}\paren*{
			\braces*{
				x \in K : |\tilde u_n(t, x) - \tilde u(t, x)| \ge \eps /
				\big\| e^{\Phi} \big\|_{L^\infty([0, T] \times \R^d)}
			}
		}\,.
	\end{align*}
	Since $\tilde u_n\to \tilde u$ in $C([0, T]; L(\R^d))$, the right-hand
	side converges to 0 as $n \to \infty$.

	For part 2), by assumption, we know that (\ref{eq:majorant}) holds for some $g \in L^1((0, T))$.
	Then, since $c(t) \in L^\infty(\R^d)$ for almost every $t \in [0, T]$,
	\begin{align*}
		\hspace{2em}
		\sup_{n \in \N} \, \norm*{\div b_n(t) + p (c_n - c)(t)}_{L^\infty(\R^d)}
		\le g(t) + p\|c(t)\|_{L^\infty(\R^d)}
		\qquad \text{for a.e. $t \in [0, T]$.}
	\end{align*}
	Therefore, the sequence $\tilde u_n$ satisfies
	the assumptions of Theorem \ref{theorem:stability} part 2), so
	$\tilde u_n \to \tilde u$ in $C([0, T]; L^p(\R^d))$. Since $e^{\Phi} \in L^\infty([0, T] \times \R^d)$,
	multiplication by $e^{\Phi}$ preserves convergence in $C([0, T]; L^p(\R^d))$, hence also
	$u_n \to u$ in $C([0, T]; L^p(\R^d))$.
\end{proof}

\begin{corollary}[Existence \& Uniqueness]
	Assume (\ref{eq:bc-hypotheses}) and (\ref{eq:b-growth-condition}).
	If $u^0\in L(\R^d)$, then there exists a unique renormalized
	solution $u$ of (\ref{eq:transport}). In addition,
	\begin{enumerate}
		\item $u \in C([0, T]; L(\R^d))$.
		\item If $u^0 \in L^\infty(\R^d)$, then $u \in L^\infty([0, T] \times \R^d)$
		      and $u \in C([0, T]; L^r_\loc(\R^d))$ for every $r < \infty$.
		\item If $u^0 \in L^p(\R^d)$ for some $1 \le p < \infty$, then
		      $u \in C([0, T]; L^p(\R^d))$.
	\end{enumerate}
	\label{cor:existence}
\end{corollary}

\begin{proof}
	For the uniqueness, by Lemma \ref{lemma:c+a}, we may assume that $c = 0$.
	Then, uniqueness follows from Corollary \ref{cor:uniqueness-infty} by
	choosing an injective admissible function.

	For existence, let $\beta_n(\sigma) = n \tanh(\sigma/n)$ and $u^0_n = \beta_n(u^0)$. By
	Proposition \ref{prop:existence}, there exists a weak solution $u_n \in L^\infty([0, T] \times \R^d)$
	of the transport equation with $(b, c, u_n^0)$ for every $n \in \N$. Since
	$b \in L^1(0, T; W^{1, 1}_\loc(\R^d))$, $u_n$ is renormalized by
	Proposition \ref{prop:consistency}. Furthermore, $u_n^0 \to u^0$ in $L(\R^d)$,
	so the existence of $u$ and part 1) follow from Corollary \ref{cor:stability-c-ne-0}.

	Part 2) follows from Propositions \ref{prop:existence} and \ref{prop:continuity}.

	For part 3), we choose a sequence $b_n \in L^1(0, T; C^1_b(\R^d))$ as in Lemma \ref{lemma:bn-seq} and set
	$c_n \colonequals (c *\rho_{1/n})\eta_n$, $u^0_n \colonequals (u^0 * \rho_{1/n})\eta_n$
	for $n \in \N$.
	Then, by Step 1 of Proposition \ref{prop:existence}, there exists a sequence
	$u_n \in L^\infty(0,T;C^1_c(\R^d))$ satisfying
	\begin{align*}
		\del_t u_n - b_n \cdot \nabla u_n + c_n u_n = 0, \qquad u_n|_{t = 0} = u_n^0 \,.
	\end{align*}
	By Proposition \ref{prop:consistency}, the $u_n$ are renormalized solutions. Since $u_n^0 \to u^0$ in $L^p(\R^d)$
	and (\ref{eq:majorant}) is satisfied by (\ref{eq:div-bn-bounded}),
	we conclude by part 2) of Corollary \ref{cor:stability-c-ne-0}.
\end{proof}

\subsection{More properties of renormalized solutions}

\begin{lemma}
	Let $b \in L^1(0, T; W^{1,1}_\loc(\R^d))$, $c = 0$
	and $u_1, \dots, u_m \in L([0, T] \times \R^d)$ be renormalized solutions of (\ref{eq:transport})
	with initial conditions $u_1^0, \dots, u_m^0 \in L(\R^d)$. If $F : \R^m \to \R$
	is continuous, then $F(u_1, \dots, u_m)$ is a renormalized solution of (\ref{eq:transport})
	with initial condition $F(u_1^0, \dots, u_m^0)$.
	\label{lemma:composition}
\end{lemma}

\begin{proof}
	Let $\beta$ be an admissible function and $v_i \colonequals \beta(u_i)$, $v_i^0 \colonequals \beta(u_i^0)$.
	Then, by Lemma \ref{lemma:regularization}, $v_{i,\eps} \colonequals v_i * \rho_\eps$ satisfies
	\begin{align*}
		\del_t v_{i, \eps}- b \cdot \nabla v_{i, \eps}= r_{i, \eps}\,,
		\qquad v_{i,\eps}|_{t = 0} = v_{i,\eps}^0
	\end{align*}
	where $r_{i,\eps} \to 0$ in $L^1(0, T; L^1_\loc(\R^d))$. In the following, we
	write $u = (u_i)_{i = 1}^m$, $v_\eps = (v_{i,\eps})_{i = 1}^m$ etc. for the
	vector-valued functions.

	For now, assume that $F \in C^1_b(\R^m)$ and let $\gamma$ be admissible. Then, by
	the chain rule for Sobolev functions
	\begin{align*}
		\del_t \big(\gamma(F(v_\eps))\big)
		 & =  \gamma'(F(v_\eps)) \, \nabla F(v_\eps) \cdot (b \cdot \nabla v_\eps + r_\eps)                        \\
		 & =  \nabla\big(\gamma(F(v_\eps))\big)  \cdot b + \gamma'(F(v_\eps)) \, \nabla F(v_\eps) \cdot r_\eps \,.
	\end{align*}
	We integrate against a test function $\phi \in C^\infty_c([0, T) \times \R^d)$:
	\begin{align*}
		- \iTR \gamma
		 & (F(v_\eps)) \, \del_t \phi \,dx \,dt
		- \int_{\R^d} \gamma(F(v_\eps^0)) \, \phi(0, x) \,dx            \\
		 & + \iTR \gamma(F(v_\eps)) (\phi \div b + b \cdot \nabla \phi)
		= \iTR \gamma'(F(v_\eps)) \, \nabla F(v_\eps) \cdot r_\eps \, \phi\,dx \,dt
	\end{align*}
	Next, we let $\eps \to 0$ to obtain
	\begin{align*}
		- \iTR \gamma
		 & (F(v)) \, \del_t \phi \,dx \,dt
		- \int_{\R^d} \gamma(F(v^0)) \, \phi(0, x) \,dx
		+ \iTR \gamma(F(v)) (\phi \div b + b \cdot \nabla \phi) = 0 \,.
	\end{align*}
	On the left-hand side of the equation, this is justified by dominated convergence, since
	$v_\eps \to v$ almost everywhere. The right-hand side converges to 0, because
	$r_\eps \to 0$ in $L^1_\loc([0, T] \times \R^d; \R^d)$, and $\gamma'$
	and $\nabla F$ are bounded.

	By letting $\beta$ converge pointwise to
	the identity on $\R$, using dominated convergence and the boundedness of $\gamma$,
	we get
	\begin{align*}
		- \iTR \gamma
		 & (F(u)) \, \del_t\phi \,dx \,dt
		- \int_{\R^d} \gamma(F(u^0)) \, \phi(0, x) \,dx
		+ \iTR \gamma(F(u)) (\phi \div b + b \cdot \nabla \phi) = 0 \,.
	\end{align*}
	This concludes the proof for $F \in C^1_b(\R^m)$. For $F \in C(\R^m)$, we choose a
	sequence $F_n \in C^1_b(\R^m)$ converging pointwise to $F$ everywhere on $\R^m$.
	Then we set $F = F_n$ in the last equation, let $n \to \infty$, and use dominated convergence.
\end{proof}

Lemma \ref{lemma:composition} implies that renormalized solutions of the transport equation are
``additive'': If
\begin{align*}
	b \in L^1(0, T; W^{1, 1}_\loc(\R^d)), \qquad c \in L^1(0, T; L^\infty(\R^d)),
\end{align*}
and $u, v$ are renormalized solutions of (\ref{eq:transport}) for initial conditions $u^0, v^0$,
then $u + v$ is a renormalized solution of (\ref{eq:transport})
for the initial condition $u^0 + v^0$. By Lemma \ref{lemma:composition}, this holds
for $c = 0$, and via Lemma \ref{lemma:c+a}, also for $c \in L^1(0, T; L^\infty(\R^d))$.

\section{Regular Lagrangian flows}
\label{sec:rlf}

\subsection{Existence}

We now prove the existence of a regular Lagrangian flow for a vector field
$b$ satisfying
\begin{align}
	b \in W^{1, 1}_\loc(\R^d; \R^d), \qquad \div b \in L^\infty(\R^d),
	\qquad
	\frac{|b|}{1 + |x|} \in L^1(\R^d) + L^\infty(\R^d) \,.
	\label{eq:b-hypotheses}
\end{align}
For simplicity, we restrict ourselves to the autonomous case. This restriction is not necessary, but it
allows us to work with a flow depending on a single time variable instead of two time variables,
which simplifies the proof.

\begin{theorem}
	Assume (\ref{eq:b-hypotheses}). Then there exists a map $X \in C(\R; L(\R^d; \R^d))$ with the
	following properties.
	\begin{enumerate}
		\item For almost all $x \in \R^d$, $X(\cdot, x) \in W^{1, 1}_\loc(\R; \R^d)$ and
		      \begin{align}
			      \del_t X(t, x) = b(X(t, x)) \quad \text{on $\R$}, \qquad X(t, x)|_{t = 0} = x \,.
		      \end{align}
		\item For all $t \in \R$
		      \begin{align}
			      e^{-|t|D} \meas^d \le X(t, \cdot)_\# \meas^d \le e^{|t|D} \meas^d \,,
			      \qquad \text{where} \ D \colonequals \norm*{\div b}_{L^\infty(\R^d)} \,.
			      \label{eq:near-incompressibility}
		      \end{align}
		\item If $u^0\in L(\R^d)$, then $u^0(X(t, x))$
		      is the unique renormalized solution of (\ref{eq:transport}) with
		      $c = 0$ for the initial data $u^0$.
		\item The following group property holds: For all $s, t \in \R$
		      \begin{align}
			      \hspace{2em}
			      X(s + t, x) = X(s, X(t, x))
			      \qquad \text{a.e. on $\R^d$.}
			      \label{eq:group}
		      \end{align}
	\end{enumerate}
	\label{theorem:ode-existence}
\end{theorem}

In particular, Theorem \ref{theorem:ode-existence} implies that for every $T > 0$, the
restriction $X|_{[0, T] \times \R^d}$ is a regular Lagrangian flow for $b$.

In the proof of Theorem \ref{theorem:ode-existence}, we will use the
following estimate of the ``superlevels'' of the flow, due to Crippa and DeLellis \cite{CL08}.
If $X$ is a regular Lagrangian flow, we define
\begin{align}
	G^X_\lambda \colonequals \{x \in \R^d : |X(t, x)| \le \lambda \ \ \text{for all $t \in [0, T]$}\}.
	\label{eq:superlevels}
\end{align}
In Lemma \ref{lemma:superlevels}, we assume that the growth condition (\ref{eq:b-growth-condition})
holds. We write $|b|/(1 + |x|) = b^1 + b^2$, where $b^1 \in L^1(0, T; L^1(\R^d))$ and
$b^2 \in L^1(0, T; L^\infty(\R^d))$.

\begin{lemma}[\cite{CL08}, Proposition 3.2]
	Let $b \in L^1_\loc([0, T] \times \R^d)$, assume (\ref{eq:b-growth-condition}),
	and let $X$ be a regular Lagrangian flow for $b$. Then for every $R > 0$
	\begin{align*}
		\meas^d\paren*{B_R(0) \setminus G_\lambda^X} \le g(R, \lambda) \,.
	\end{align*}
	The function $g$ only depends on $\|b^1\|_{L^1(0, T; L^1(\R^d))}$,
	$\|b^2\|_{L^1(0, T; L^\infty(\R^d))}$ and the compressibility constant of $X$.
	For fixed $R$, $\lim_{\lambda \to \infty} g(R, \lambda) = 0$.
	\label{lemma:superlevels}
\end{lemma}

\begin{proof}
	By property \ref{enum:F1} of a regular Lagrangian flow,
	\begin{align*}
		\int_{B_R} \sup_{t \in [0, T]} \log \paren*{
			\frac{1 + |X(t, x)|} {1 + R}
		} \,dx
		 & \le \int_{B_R} \int_0^T \frac{|\del_t X(t, x)|}{1 + |X(t, x)|} \,dt \,dx  \\
		 & = \int_{B_R} \int_0^T \frac{|b(t, X(t, x))|}{1 + |X(t, x)|} \,dt \,dx \,.
	\end{align*}
	Using the decomposition $|b|/(1 + |x|) = b^1 + b^2$, we obtain
	\begin{align*}
		\int_{B_R} \sup_{t \in [0, T]} \log \paren*{
			\frac{1 + |X(t, x)|} {1 + R}
		} \,dx
		 & \le \int_0^T \contractint \int_{B_R} |b^1(t, X(t, x))| + |b^2(t, X(t, x))| \,dt \,dx \\
		 & \le K \norm*{b^1}_{L^1_tL^1_x} + \meas^d(B_R) \norm*{b^2}_{L^1_tL^\infty_x} \,.
	\end{align*}
	Here $K$ is the compressibility constant of $X$ from Definition \ref{def:lrf}.
	By Chebyshev's inequality,
	\begin{align*}
		\meas^d\paren*{B_R \setminus G_\lambda^X}
		\le \paren*{\log\paren*{\frac {1 + \lambda}{1 + R}}}^{\!-1}
		\paren*{
			K \norm*{b^1}_{L^1_tL^1_x}
			+ \meas^d(B_R) \norm*{b^2}_{L^1_tL^\infty_x}
		} \,.
	\end{align*}
	Now, the function on the right-hand side has the required properties of $g$.
\end{proof}

\begin{proof}[Proof of Theorem \ref{theorem:ode-existence}]
	By Corollary \ref{cor:existence}, there exists a $X \in C(\R; L(\R^d; \R^d))$ such that
	for each $1 \le j \le d$, $X_j$ is the unique renormalized solution of
	\begin{align}
		\del_t X_j - b \cdot \nabla X_j = 0 \quad \text{on $\R \times \R^d$},
		\qquad X_j|_{t = 0} = x_j \,.
		\label{eq:Xj-transport}
	\end{align}
	Even though Corollary \ref{cor:existence} is formulated for a finite time interval $[0, T]$,
	we get a solution $X \in C(\R; L(\R^d; \R^d))$ as follows: For every $T > 0$, there exists
	a unique solution $X_j^T \in C([0, T]; L(\R^d; \R^d))$ of (\ref{eq:Xj-transport})
	on $[0, T] \times \R^d$. Therefore, if $S < T$, then $X_j^T|_{[0, S] \times \R^d} = X_j^S$ almost
	everywhere. Hence, by inductively choosing $X_j^n$ for $n \in \N$ such that
	$X_j^n|_{[0, m] \times \R^d} = X_j^m$ for $m < n$ everywhere on $[0, m] \times \R^d$, we can define
	$X_j(t, x) \colonequals X_j^n(t, x)$ for some $n > t$. The construction for negative
	times is the same.

	\paragraph*{Part 2).}

	By Lemma \ref{lemma:composition}, for every $u^0 \in C(\R^d)$, $u^0(X)$ is the unique
	renormalized solution of (\ref{eq:transport}) with $c = 0$ for the initial data $u^0$.
	Combining this with (\ref{eq:ddt-Lp-norm}), we see that
	for every $u^0 \in C(\R^d) \cap L^1(\R^d)$
	\begin{align}
		\frac d {dt} \int_{\R^d} \abs*{u^0(X(t, x))} \,dx
		= -\int_{\R^d} (\div b) \abs*{u^0(X(t, x))} \,dx
		\qquad \text{for a.e. $t \in \R$.}
		\label{eq:L1-evo}
	\end{align}
	Let $D \colonequals \norm*{\div b}_{L^\infty}$. By Grönwall's inequality and (\ref{eq:L1-evo}),
	for every $t \in \R$
	\begin{align}
		e^{-|t|D} \int_{\R^d} \abs*{u^0} \,dx
		\, \le \, \int_{\R^d} \abs*{u^0(X(t, x))} \,dx
		\, \le \, e^{|t|D} \int_{\R^d} \abs*{u^0} \,dx \,.
		\label{eq:u0-near-incompressibility}
	\end{align}
	Let $A \subset \R^d$ be a measurable set. Letting $u^0 \uparrow \chi_A$ in (\ref{eq:u0-near-incompressibility})
	and the monotone convergence theorem imply (\ref{eq:near-incompressibility}).

	\paragraph{Part 3).}

	We now show that also for $u^0 \in L(\R^d)$, $u^0(X)$ is the unique renormalized solution of
	(\ref{eq:transport}) with $c = 0$ and initial data $u^0$.
	By Lusin's theorem, $C(\R^d) \subset L(\R^d)$ is dense with respect to
	local convergence in measure.
	Let $u^0_n\in C(\R^d)$ be a sequence converging to $u^0$ in $L(\R^d)$. By passing to
	a subsequence, we may in addition assume that $u_n^0 \to u^0$ almost everywhere on $\R^d$.
	By part 1) of Theorem \ref{theorem:stability}, $u^0_n(X) \to \tilde u$ in $L(\R \times \R^d)$,
	where $\tilde u$ is the unique renormalized solution of (\ref{eq:transport})
	with $(b, 0, u^0)$. To show that $\tilde u = u^0(X)$, it suffices to show that
	$u_n^0(X) \to u^0(X)$ a.e. on $\R^d$ for a.e. every $t \in \R$. This is the case because for a.e.
	$t \in \R$
	\begin{align*}
		\meas^{d}\paren*{ X(t, \cdot)^{-1} \paren*{\braces*{x \in \R^d : u_n^0(x) \nrightarrow u^0(x)}} }= 0
	\end{align*}
	by (\ref{eq:near-incompressibility}).

	\paragraph{Part 4).}

	Let $s \in \R$. By the previous paragraph,
	$X(s, X(t, x))$ is the unique renormalized solution of
	\begin{align*}
		\del_t u - b \cdot \nabla u = 0, \qquad u|_{t = 0} = X(s, \cdot) \,.
	\end{align*}
	However, $X(s + t, x)$ is also a renormalized solution of this equation. So by uniqueness
	(Corollary \ref{cor:existence}), $X(s + t, x) = X(s, X(t, x))$ for almost
	every $x \in \R^d$.

	\paragraph{Part 1).}

	Let $b_n \colonequals b * \rho_{1/n}$, where $\rho_{1/n}$ is a regularizing kernel
	defined in (\ref{eq:rho-eps}). By the growth condition in (\ref{eq:b-hypotheses}), we can
	write $|b| = (1 + |x|)(b^1 + b^2)$, where $b^1 \in L^1(\R^d)$, $b^2 \in L^\infty(\R^d)$ and
	$b^1, b^2 \ge 0$. Using that $\rho_{1/n}$ is supported in $B_{C_\rho/n}(0)$ for some $C_{\rho} > 0$, we compute
	\begin{align}
		\begin{split}
			\frac {|b_n|}{1 + |x|}
			 & = \frac 1 {1 + |x|} \abs*{\int_{\R^d} b(y)\rho_{1/n}(y - x) \,dy}                  \\
			 & \le \frac 1 {1 + |x|} \int_{\R^d} (b^1(y) + b^2(y))(1 + |y|)\rho_{1/n}(y - x) \,dy \\
			 & \le \tilde C_\rho \int_{\R^d} (b^1(y) + b^2(y))\rho_{1/n}(y - x) \,dy              \\
			 & = \tilde C_\rho \, (b^1 * \rho_{1/n} + b^2 * \rho_{1/n}) \,.
		\end{split}
		\label{eq:bn-growth}
	\end{align}
	Here $\tilde C_\rho$ is some constant only depending on $\rho$.

	We claim that there exists a flow $X_n \in C^1(\R \times \R^d; \R^d)$ for $b_n$.
	By the local version of the Cauchy-Lipschitz theorem, for every $x^0 \in \R^d$, there exists
	a $\tilde T > 0$ and a solution $\gamma \in C^1([0, \tilde T]; \R^d)$ to the ODE (\ref{eq:ode})
	with $b$ replaced by $b_n$.
	By (\ref{eq:bn-growth}), $b_n/(1 + |x|) \in L^\infty(\R^d)$, since
	$\norm{b^1 * \rho_{1/n}}_{L^\infty} \le \norm{b^1}_{L^1} \norm{\rho_{1/n}}_{L^\infty}$.
	Therefore the solution $\gamma$
	can be extended to a maximal solution on $[0, \infty)$. The same argument applies for
	negative times. Thus, there exists a global flow map $X_n : \R \times \R^d \to \R^d$.
	Since $b_n \in C^1(\R^d)$, by Proposition \ref{prop:smooth-dependence},
	$X_n(t, \cdot) \in C^1(\R^d; \R^d)$ for every $t \in \R$. Since $b$ is autonomous, we also
	have $X(\cdot, x) \in C^1(\R; \R^d)$ for every $x \in \R^d$,
	and hence $X \in C^1(\R \times \R^d; \R^d)$.

	Since $b_n$ is locally Lipschitz continuous, $X_n$ has the group property (\ref{eq:group}).
	We compute
	\begin{align*}
		0 = \frac{d}{dt} \big(X_n(t, X_n(-t, x))\big)
		= \del_t X_n(t, X_n(-t, x)) - DX_n(t, X_n(-t, x)) \, \del_t X_n(-t, x) \,.
	\end{align*}
	Since $\del_t X_n(-t, x) = b_n(X_n(-t, x))$, we see that $X_n$
	satisfies the transport equation with $c = 0$ in the classical sense:
	\begin{align*}
		\del_t X_n - b_n \!\cdot \nabla X_n = 0, \qquad X_n|_{t = 0} = \id_{\R^d}
	\end{align*}
	By the chain rule, for every $\beta \in C^1(\R^d)$,
	\begin{align}
		\del_t \beta(X_n) - b_n \!\cdot \nabla \big(\beta(X_n)\big) = 0,
		\qquad \beta(X_n)|_{t = 0} = \beta \,.
		\label{eq:Xn-transport}
	\end{align}
	Thus, for every $1 \le j \le d$, the $j$-th component $X_{n, j}$ is a renormalized solution
	of (\ref{eq:transport}) with $(b_n, 0, x_j)$. By part 1) of Theorem \ref{theorem:stability},
	for every $T > 0$, $X_n \to X$ in $C([-T, T]; L(\R^d; \R^d))$. In particular, $X_n \to X$
	in $L(\R \times \R^d; \R^d)$.

	We now show that for every $\beta \in C^1(\R^d)$ such that $\beta$ and $|\nabla \beta|(1 + |x|)$ are
	bounded, the following equation holds in the weak sense:
	\begin{align}
		\del_t \beta(X) = b(X) \cdot \nabla \beta(X) \quad \text{on $\R \times \R^d$},
		\qquad \beta(X)|_{t = 0} = \beta \quad \text{on $\R^d$}
		\label{eq:ode-renormalized}
	\end{align}
	For simplicity, we only prove this on $\lbrack 0, \infty \rparen \times \R^d$.
	The argument for negative times is the same.
	Since $\del_t X_n = b_n(X_n)$ and $X_n \in C^1(\R \times \R^d; \R^d)$, (\ref{eq:ode-renormalized})
	holds with $X$ and $b$ replaced by $X_n$ and $b_n$. Thus, for every
	$\phi \in C^\infty_c([0, \infty) \times \R^d)$
	\begin{align*}
		-\int_0^\infty \contractint \int_{\R^d} \beta(X_n(t, x)) \, \del_t \phi \,dx \,dt
		-\int_{\R^d} \phi(0) \beta \,dx
		= \int_0^\infty \contractint \int_{\R^d}
		b_n(X_n(t, x)) \cdot \nabla \beta(X_n(t, x)) \, \phi \,dx \,dt \,.
	\end{align*}
	Our goal is to prove
	\begin{align}
		-\int_0^\infty \contractint \int_{\R^d} \beta(X(t, x)) \, \del_t \phi \,dx \,dt
		-\int_{\R^d} \phi(0) \beta \,dx
		= \int_0^\infty \contractint \int_{\R^d}
		b(X(t, x)) \cdot \nabla \beta(X(t, x)) \, \phi \,dx \,dt \,.
		\label{eq:ode-renormalized-weak}
	\end{align}
	Since $X_n \to X$ in $L(\R \times \R^d)$ and $\beta$ is bounded and Lipschitz continuous,
	\begin{align*}
		\int_0^\infty \contractint \int_{\R^d} \beta(X_n(t, x)) \, \del_t \phi \,dx \,dt
		\ \xrightarrow[n \to \infty]{} \
		\int_0^\infty \contractint \int_{\R^d} \beta(X(t, x)) \, \del_t \phi \,dx \,dt \,.
	\end{align*}
	Thus, it remains to show that
	\begin{align}
		\int_0^\infty \contractint \int_{\R^d} b_n(X_n(t, x)) \cdot \nabla \beta(X_n(t, x)) \, \phi \,dx \,dt
		\ \xrightarrow[n \to \infty]{} \
		\int_0^\infty \contractint \int_{\R^d} b(X(t, x)) \cdot \nabla \beta(X(t, x)) \, \phi \,dx \,dt \,.
		\label{eq:bn-Xn-convergence}
	\end{align}

	By (\ref{eq:Xn-transport}), $X_n$ is a renormalized solution of the transport
	equation with $(b_n, 0, \id_{\R^d})$. Since $\div b_n \in L^\infty(\R^d)$, by the
	computation in the proof of part 2), the measure $X_n(t, \cdot)_\# \meas^d$
	is absolutely continuous with respect to $\meas^d$ and
	\begin{align}
		e^{-|t|D_n} \meas^d \le X_n(t, \cdot)_\# \meas^d \le e^{|t|D_n} \meas^d \,,
		\qquad \text{where} \ D_n \colonequals \norm*{\div b_n}_{L^\infty(\R^d)} \,.
		\label{eq:Xn-near-incompressibility}
	\end{align}
	Let $h_n(t)$ be the density of  $X_n(t, \cdot)_\# \meas^d$, and let
	$T > 0$ such that $\supp \phi \subset [0, T) \times \R^d$. Since
	\begin{align*}
		\norm*{\div b_n}_{L^\infty}
		= \norm*{(\div b) * \rho_{1/n}}_{L^\infty} \le \norm*{\div b}_{L^\infty} \,,
	\end{align*}
	$\{h_n\}_{n \in \N}$ is bounded in $L^\infty([0, T] \times \R^d)$ by (\ref{eq:Xn-near-incompressibility}).
	By (\ref{eq:near-incompressibility}), also $X(t, \cdot)_\#$ is absolutely continuous with respect
	to $\meas^d$, and we denote the density of $X(t, \cdot)_\#$ by $h(t, \cdot)$.

	By changing variables and using the group property, (\ref{eq:bn-Xn-convergence}) is equivalent to
	\begin{align}
		\int_0^\infty \contractint \int_{\R^d}
		b_n \!\cdot \nabla \beta \, \phi(t, X_n(-t, x)) h_n \,dx \,dt
		\ \xrightarrow[n \to \infty]{} \
		\int_0^\infty \contractint \int_{\R^d}
		b \cdot \nabla \beta \, \phi(t, X(-t, x)) h \,dx \,dt \,.
		\label{eq:bn-Xn-convergence-cov}
	\end{align}
	Since $\phi$ is bounded and $h_n$ is bounded in $L^\infty([0, T] \times \R^d)$, by the Banach-Alaoglu
	theorem, there exists a subsequence $\phi(t, X_{k_n}(-t, x)) h_{k_n}$ of $\phi(t, X_n(-t, x)) h_n$
	converging weak-$*$ in $L^\infty([0, T] \times \R^d)$.
	To identify the weak-$*$ limit, we test against $\psi \in C^\infty_c([0, T] \times \R^d)$.
	\begin{align*}
		\iTR & \psi \phi(t, X_{k_n}(-t, x)) h_{k_n} \,dx \,dt
		= \iTR \psi(t, X_{k_n}(t, x)) \phi \,dx \,dt          \\
		     & \xrightarrow{n \to \infty} \
		\iTR \psi(t, X(t, x)) \phi \,dx \,dt
		= \iTR \psi \phi(t, X(-t, x)) h \,dx \,dt \,.
	\end{align*}
	The convergence follows from $X_{k_n} \to X$ in $L(\R \times \R^d; \R^d)$, the compact support of $\phi$,
	and because $\psi$ is bounded and Lipschitz. Since this argument also applies to any subsequence
	of $\phi(t, X_n(-t, x)) h_n$, we conclude
	\begin{align}
		\phi(t, X_n(-t, x)) h_n \rightharpoonup^* \phi(t, X(-t, x)) h
		\qquad \text{in $L^\infty([0, T] \times \R^d)$.}
		\label{eq:phi-h-weak-star}
	\end{align}
	Let $\lambda > 0$. We write
	\begin{align}
		 & \abs*{
			\int_0^\infty \contractint \int_{\R^d}
			b_n \!\cdot \nabla \beta \, \phi(t, X_n(-t, x)) h_n \,dx \,dt
			- \int_0^\infty \contractint \int_{\R^d}
			b \cdot \nabla \beta \, \phi(t, X(-t, x)) h \,dx \,dt
		}
		\label{eq:local-split}
		\\
		 & \qquad \le \abs*{
			\int_0^\infty \contractint \int_{B_\lambda}
			b_n \!\cdot \nabla \beta \, \phi(t, X_n(-t, x)) h_n
			- b \cdot \nabla \beta \, \phi(t, X(-t, x)) h \,dx \,dt
		}
		+ r_n(\lambda) + r(\lambda) \,, \notag
	\end{align}
	where
	\begin{align*}
		r_n(\lambda) & \colonequals \abs*{
			\int_0^\infty \contractint \int_{\R^d \setminus B_\lambda}
			b_n \!\cdot \nabla \beta \, \phi(t, X_n(-t, x)) h_n \,dx \,dt
		} \qquad \text{for $n \in \N$,}    \\
		r(\lambda)   & \colonequals \abs*{
			\int_0^\infty \contractint \int_{\R^d \setminus B_\lambda}
			b \cdot \nabla \beta \, \phi(t, X(-t, x)) h \,dx \,dt
		} \,.
	\end{align*}
	To conclude (\ref{eq:bn-Xn-convergence-cov}), we show that the right-hand side of
	(\ref{eq:local-split}) can be made arbitrarily small by choosing $\lambda$ appropriately
	and $n$ large enough.
	By (\ref{eq:phi-h-weak-star}) and since $b_n \to b$ in $L^1_\loc(\R^d)$,
	for every fixed $\lambda > 0$, the first term in the second line of (\ref{eq:local-split})
	converges to 0 as $n \to \infty$.

	Since $|\nabla \beta| (1 + |x|)$ is bounded and (\ref{eq:b-growth-condition}) holds,
	$b \cdot \nabla \beta \in L^1(\R^d) + L^\infty(\R^d)$. By (\ref{eq:near-incompressibility})
	and since $\phi$ and $h$ are bounded on $[0, T] \times \R^d$,
	\begin{align*}
		\phi(t, X(-t, x)) h \in L^1(0, T; L^1) \cap L^1(0, T; L^\infty) \,.
	\end{align*}
	Thus $r(\lambda) \to 0$ as $\lambda \to \infty$. We finish the proof of (\ref{eq:bn-Xn-convergence-cov})
	by showing that $r_n(\lambda) \to 0$ uniformly in $n$ as $\lambda \to \infty$. By (\ref{eq:bn-growth}),
	\begin{align*}
		r_n(\lambda)
		 & \le C_{\rho,\beta} \int_0^\infty \contractint \int_{\R^d \setminus B_\lambda}
		(b^1 * \rho_{1/n} + b^2 * \rho_{1/n}) |\phi(t, X_n(-t, x))| h_n \,dx \,dt                         \\
		 & \le C_{\rho,\beta} \norm*{\phi}_{L^\infty_{t, x}} \norm*{h_n}_{L^\infty([0, T] \times \R^d)} T
		\int_{\R^d \setminus B_{\lambda}} b^1 * \rho_{1/n} \,dx                                           \\
		 & \quad + \ C_{\rho,\beta} \norm*{b^2}_{L^\infty(\R^d)}
		\int_0^T \contractint \int_{X_n(-t, \R^d \setminus B_\lambda)} |\phi| \,dx \,dt \,.
	\end{align*}
	In the last line, we changed variables $x' = X(-t, x)$.
	Since $\supp \rho_{1/n} \subset B_{C_\rho/n}$ for some $C_\rho > 0$, and since
	$h_n$ is bounded in $L^\infty([0, T] \times \R^d)$, the first term on the right-hand
	side converges to 0 as $\lambda \to \infty$, uniformly in $n$.
	Let $R > 0$ such that $\supp \phi \subset [0, T] \times B_R(0)$. By the group
	property, $X_n(-t, x) = X_n(t, \cdot)^{-1}(x)$. Therefore
	\begin{align*}
		\int_0^T \contractint \int_{X_n(-t, \R^d \setminus B_\lambda)} |\phi| \,dx \,dt
		 & \le \norm*{\phi}_{L^\infty_{t, x}}
		\int_0^T \meas^d {\big(\{x \in B_R(0) : |X_n(t, x)| \ge \lambda\}\big)} \,dt                      \\
		 & \le  \norm*{\phi}_{L^\infty_{t, x}} T \, \meas^d\big(B_R(0) \setminus G^{X_n}_\lambda\big) \,.
	\end{align*}
	Here $G_\lambda^{X_n}$ is the set defined in (\ref{eq:superlevels}).
	We now want to apply Lemma \ref{lemma:superlevels}. By (\ref{eq:bn-growth}),
	\begin{align*}
		\frac{|b_n|}{1 + |x|} \le \tilde C_\rho (b^1 * \rho_{1/n} + b^2 * \rho_{1/n}) \,,
	\end{align*}
	and $b^1 * \rho_{1/n} \in L^1(\R^d)$, $b^2 * \rho_{1/n} \in L^\infty(\R^d)$.
	Thus, by Lemma \ref{lemma:superlevels},
	\begin{align*}
		\meas^d\big(B_R(0) \setminus G^{X_n}_\lambda\big) \le g_n(R, \lambda)
	\end{align*}
	where the function $g_n$ depends only on the compressibility constant $K_n$ of $X_n$ as well as the norms
	$\norm{b^1 * \rho_{1/n}}_{L^1(0, T; L^1)}$ and $\norm{b^2 * \rho_{1/n}}_{L^1(0, T; L^\infty)}$.
	Since $\{K_n\}_{n \in \N}$ is bounded by (\ref{eq:Xn-near-incompressibility}) and
	\begin{align*}
		\norm*{b^1 * \rho_{1/n}}_{L^1(0, T; L^1)}
		\le \norm*{b^1}_{L^1(0, T; L^1)}, \qquad
		\norm*{b^2 * \rho_{1/n}}_{L^1(0, T; L^\infty)}
		\le \norm*{b^2}_{L^1(0, T; L^\infty)} \,,
	\end{align*}
	we conclude that $g_n(R, \lambda) \to 0$ as $\lambda \to \infty$, uniformly in $n$.
	Therefore $r_n(\lambda) \to 0$ as $\lambda \to \infty$, uniformly in $n$.

	We now show that for almost every $x \in \R^{d}$, $X(\cdot, x) \in W^{1, 1}_\loc(\R)$
	and $\del_t X(t, x) = b(X(t,x))$. Since $b \cdot \nabla \beta \in L^1(\R^d) + L^\infty(\R^d)$,
	by (\ref{eq:near-incompressibility})
	\begin{align*}
		b(X) \cdot \nabla \beta(X) \in L^1(-T, T; L^1(\R^d)) + L^1(-T, T; L^\infty(\R^d))
	\end{align*}
	for every $T > 0$. In particular, $b(X) \cdot \nabla \beta(X) \in L^1_\loc(\R \times \R^d)$.

	Thus
	\begin{align}
		b(X(t, x)) \cdot \nabla \beta(X(t, x)) \in L^1_\loc(\R)
		\qquad \text{for a.e. $x \in \R^d$.}
		\label{eq:bX-L1loc-ae}
	\end{align}
	This is true for the ``$\loc$''-spaces, because both $\R$ and $\R^d$ can be
	covered by a countable number of compact sets.

	Let $\psi \in C^\infty_c([0, \infty))$ and $\xi \in C^\infty_c(\R^d)$. We choose
	$\phi(t, x) = \xi(x) \psi(t)$ in (\ref{eq:ode-renormalized-weak}).
	\begin{align*}
		-\int_0^\infty \contractint \int_{\R^d} \beta(X) \, \xi \, \del_t\psi \,dx \,dt
		- \psi(0) \int_{\R^d}\beta \xi \,dx
		= \int_0^\infty \contractint \int_{\R^d} b(X) \cdot \nabla \beta(X) \, \xi \psi \,dx \,dt \,.
	\end{align*}
	Since $\xi$ is arbitrary, for a.e. $x \in \R^d$,
	\begin{align}
		-\int_0^\infty \beta(X(t, x)) \del_t\psi(t) \,dt
		- \psi(0) \beta(x)
		= \int_0^\infty b(X(t, x)) \cdot \nabla \beta(X(t, x)) \, \psi(t) \,dt \,.
		\label{eq:ode-renormalized-weak-pointwise}
	\end{align}
	Fix $x \in \R^d$ such that both (\ref{eq:bX-L1loc-ae}) and (\ref{eq:ode-renormalized-weak-pointwise})
	hold. Then $\beta(X(t, x)) \in W^{1, 1}_\loc(\R)$. We choose $\beta = \tanh(x_i)$, where $x_i$
	is the $i$-th coordinate function. Thus, $\tanh(X_i(t, x)) \in C(\R)$ after modifying
	$X_i(\cdot, x)$ on a $\meas^1$-negligible set. It is possible to achieve this via a modification of $X_i(\cdot, x)$,
	since $\tanh' > 0$ on $\R$. Since $\tanh$ is strictly increasing, this implies $X_i(t, x) \in C(\R)$.

	Let $T > 0$ and $1 \le i \le d$. By the time-continuity of $X$ just shown,
	$X([-T, T], x) \subset \R^d$ is compact,
	so we may choose $\beta$ such that $\beta(X(t, x)) = X_i(t, x)$ for all $t \in [-T, T]$.
	Then, (\ref{eq:ode-renormalized-weak-pointwise}) shows that
	$X(\cdot, x) \in W^{1, 1}_\loc(\R; \R^d)$ and $\del_t X(t, x) = b(X(t, x))$.
\end{proof}

\subsection{Uniqueness}

Instead of working with regular Lagrangian flows, DiPerna and Lions \cite{DL89}
use a different characterization of the flow map. For time-dependent vector fields,
they introduce a map $X : [0, T]^2 \times \R^d \to \R^d$, where (informally speaking)
$X(s, t, x)$ is the position at time $s$ of a particle that was at $x$ at time $t$.
This flow belongs to $C([0,T]^2; L(\R^d; \R^d))$ and is uniquely determined by
the following three properties:

\begin{enumerate}
	\item Let $\beta \in C^1(\R^d)$ such that $\beta$ and $|\nabla \beta|(1 + |x|)$ are bounded. Then
	      for every $t \in [0, T]$, the following equation holds in the weak sense:
	      \begin{align*}
		      \del_s \beta(X) = b(s, X) \cdot \nabla \beta(X) \quad \text{on $(0, T) \times \R^d$},
		      \qquad \beta(X)|_{s = t} = \beta \quad \text{on $\R^d$}
	      \end{align*}
	\item There exists a $K > 0$ such that
	      $K^{-1} \meas^d \le X(t, \cdot)_\# \meas^d \le K\meas^d$ for all $t \in [0, T]$.
	\item For all $s, t, r \in [0, T]$,
	      \begin{align}
		      X(s, t, x) = X(s, r, X(r, t, x)) \quad \text{for a.e. $x \in \R^d$.}
		      \label{eq:group-time-dependent}
	      \end{align}
\end{enumerate}

The third property is the analogue of the group property (\ref{eq:group}) if $b$ is time-dependent.
However, subsequently it has been shown that (\ref{eq:group-time-dependent}) and
the lower bound on the pushforward measure $X(t, \cdot)_\# \meas^d$ are not required for uniqueness of the flow.
This is a consequence of more general principles connecting uniqueness for the continuity equation
with uniqueness of regular Lagrangian flows.

The continuity equation for $b$ is the following PDE. It is a special case of the transport equation
(\ref{eq:transport}) with $c = \div b$ and $b$ replaced by $-b$.
\begin{align}
	\del_t u + \div(bu) = 0 \quad \text{on $(0, T) \times \R^d$,}
	\qquad
	u|_{t = 0} = u^0 \quad \text{on $\R^d$}
	\label{eq:continuity}
\end{align}
These principles were introduced by Ambrosio in \cite{Amb04}.
A comprehensive treatment of them can be found in \cite{AC08}, p.~11--19.
Below, we use the connection with the continuity equation to give a direct
proof of the uniqueness of regular Lagrangian flows.

\begin{lemma}
	Assume that (\ref{eq:b-growth-condition}) holds
	and let $X$ be a regular Lagrangian flow for $b$.
	Let $u^0 \in C^\infty_c(\R^d)$ and $g(t, \cdot)$ be the density of
	$X(t, \cdot)_\# (u^0 \,d\meas^d)$ with respect to $\meas^d$. Then $g$ is a weak solution
	of the continuity equation (\ref{eq:continuity}) with initial data $u^0$.
	\label{lemma:continuity-eq}
\end{lemma}

By the second property of a regular Lagrangian flow, the measure $X(t, \cdot)_\# (u^0 \,d\meas^d)$
is absolutely continuous with respect to $\meas^d$ and its density belongs to
$L^\infty([0, T] \times \R^d)$. Hence the weak formulation of the continuity equation
is well defined with $u = g$.

\begin{proof}
	Let $\phi \in C^\infty_c([0, T) \times \R^d)$. Since $X(\cdot, x) \in W^{1,1}(0, T; \R^d)$ for almost
	every $x \in \R^d$, by the chain rule for Sobolev functions
	\begin{align*}
		\frac{d}{dt}\big(\phi(t, X(t, x))\big)
		= \del_t \phi(t, X(t, x)) +  b(t, X(t, x)) \cdot \nabla \phi(t, X(t, x))
		\qquad \text{for a.e. $x \in \R^d$.}
	\end{align*}
	Integrating over $[0, T]$ yields
	\begin{align}
		-\phi(0, x)
		= \int_0^T \del_t \phi(t, X(t, x)) \,dt
		+ \int_0^T b(t, X(t, x)) \cdot \nabla \phi(t, X(t, x)) \,dt
		\qquad \text{for a.e. $x \in \R^d$.}
		\label{eq:continuity-pointwise}
	\end{align}
	Since $b \cdot \nabla \phi \in L^1(0, T; L^1 + L^\infty)$, by property \ref{enum:F2} of a
	regular Lagrangian flow, also
	\begin{align*}
		b(t, X(t, x)) \cdot \nabla \phi(t, X(t, x)) \in L^1(0, T; L^1 + L^\infty) \,.
	\end{align*}
	Thus, we may integrate (\ref{eq:continuity-pointwise}) over $\R^d$ after multiplying with $u^0$,
	\begin{align*}
		-\int_{\R^d}  u^0 \phi(0, x) \,dx
		 & = \iTR u^0 \del_t \phi(t, X(t, x)) \,dx \,dt                                    \\
		 & \quad\qquad + \iTR u^0 \, b(t, X(t, x)) \cdot \nabla \phi(t, X(t, x)) \,dx \,dt
	\end{align*}
	By definition of $g$, this equation is the weak form of (\ref{eq:continuity}) with $u = g$.
\end{proof}

\begin{theorem}
	If $b \in L^1(0, T; W^{1, 1}_\loc(\R^d; \R^d))$, $\div b \in L^1(0, T; L^\infty(\R^d))$ and
	(\ref{eq:b-growth-condition}) holds, then the regular Lagrangian flow for $b$ is unique.
	\label{theorem:ode-uniqueness}
\end{theorem}

\begin{proof}
	Let $X, Y : [0, T] \times \R^d \to \R^d$ be two regular Lagrangian flows for $b$.
	Let $u^0 \in C^\infty_c(\R^d)$
	and $g_X, g_Y \in L^\infty([0, T] \times \R^d)$ be the densities of $X(t, \cdot)_\# (u^0\,d\meas^d)$
	and $Y(t, \cdot)_\# (u^0 \,d\meas^d)$.
	By Lemma \ref{lemma:continuity-eq}, both $g_X$ and $g_Y$ are weak solutions
	of the continuity equation (\ref{eq:continuity}) for the initial data $u^0$. Since
	weak solutions of (\ref{eq:continuity}) are unique by Corollary \ref{cor:uniqueness-infty},
	we have $g_X = g_Y$ a.e. on $[0, T] \times \R^d$.

	By Proposition \ref{prop:continuity}, $g_X, g_Y \in C([0, T]; L^r_\loc(\R^d))$ for every $r < \infty$.
	Thus $g_X(t, \cdot) = g_Y(t, \cdot)$ a.e. on $\R^d$ for every $t \in [0, T]$.
	Hence, for every $t \in [0, T]$
	\begin{align*}
		X(t, \cdot)_\# (u^0 \,d\meas^d) = Y(t, \cdot)_\# (u^0 \,d\meas^d)
	\end{align*}
	as measures on $\R^d$.
	Testing against $\xi \in C^\infty_c(\R^d)$ yields
	\begin{align*}
		\int_{\R^d} u^0 \xi(X(t, x)) \,dx = \int_{\R^d} u^0 \xi(Y(t, x)) \,dx \,.
	\end{align*}
	Since $u^0$ is arbitrary, $\xi(X(t, x)) = \xi(Y(t, x))$ a.e. on $\R^d$.
	Since $\xi$ is arbitrary, we conclude $X(t, x) = Y(t, x)$ a.e. on $\R^d$ for every $t \in [0, T]$.
	Because $X(\cdot, x)$ and $Y(\cdot, x)$ are continuous on $[0, T]$ for a.e. $x \in \R^d$, this
	implies $X(\cdot, x) = Y(\cdot, x)$ on $[0, T]$ for a.e. $x \in \R^d$.
\end{proof}

\section{Counterexamples}
\label{sec:counterexamples}

\subsection{The growth condition}

To illustrate the necessity of the growth condition (\ref{eq:b-growth-condition})
for uniqueness of solutions to the transport equation,
consider the following example in two dimensions. Let
\begin{align*}
	b(x, y) =
	\begin{cases}
		\begin{pmatrix}
			x^2 \\
			-2xy
		\end{pmatrix} & \quad x > 0     \\
		0               & \quad x \le 0
	\end{cases}
\end{align*}
\begin{wrapfigure}[9]{r}{0.3\textwidth}
	\captionsetup{margin=0.5em}
	\vspace{-1.5em}
	\begin{center}
		\begin{tikzpicture}
			\begin{axis}[
					width=0.3\textwidth,
					height=0.3\textwidth,
					axis lines=middle,
					xmin=0,
					xmax=5,
					ymin=0,
					ymax=5,
					samples=200,
					xlabel={$t$},
					ylabel={$x$},
					xticklabels=\empty,
					yticklabels=\empty,
					xlabel style={ at={(current axis.right of origin)}, right },
					ylabel style={ at={(current axis.above origin)}, left },
				]
				\addplot[black, domain=4:10]{1/(x - 4)}; \addplot[black,
					domain=3.5:10]{1/(x - 3.5)}; \addplot[black, domain=3:10]{1/(x - 3)};
				\addplot[black, domain=2.5:10]{1/(x - 2.5)}; \addplot[black,
					domain=2:10]{1/(x - 2)}; \addplot[black, domain=1.5:10]{1/(x - 1.5)};
				\addplot[black, domain=1:10]{1/(x - 1)}; \addplot[black,
					domain=0.5:10]{1/(x - 0.5)}; \addplot[black]{1/x}; \addplot[black]{1/(x + 0.5)};
				\addplot[black]{1/(x + 1)}; \addplot[black]{1/(x + 1.5)}; \addplot[black]{1/(x + 2)};
				\addplot[black]{1/(x + 2.5)};
			\end{axis}
		\end{tikzpicture}
	\end{center}
	\vspace{-1em}
	\caption{The characteristic curves of $-b$, projected to the $x$-$t$-plane.}
\end{wrapfigure}
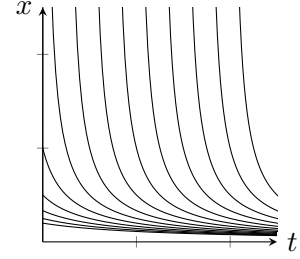

Since $b$ is smooth, $b \in W^{1,1}_\loc(\R^2)$. Furthermore, $\div b = 0$.
Note that $u \in C^1([0, \infty) \times \R^2)$ is a solution of the transport equation
(\ref{eq:transport}) with $c = 0$
if and only if it is constant along all characteristics of $-b$.

The flow map of $-b$ is
\begin{align*}
	X(t, x, y) = \begin{cases}
		             \begin{pmatrix}
			1/(x^{-1}+ t) \\
			y(1 + tx)^2
		\end{pmatrix}
		                             & \quad x > 0 \\
		             \addlinespace
		             \begin{pmatrix}
			x \\
			y
		\end{pmatrix} & \quad x \le 0
	             \end{cases}
\end{align*}
Every characteristic curve intersecting the plane $t = 0$ stays in the region
$\{(t, x, y) \in [0, \infty) \times \R^2 : x < 1/t\}$ for all $t \ge 0$. Hence, the
initial condition at $t = 0$ does not affect the region
$\{x \ge 1/t\}$, so there are infinitely many solutions to the transport equation.
For instance, let $h \in C^1(\R)$ such that $\supp h \subset (0, \infty)$. Then
\begin{align*}
	u(t, x, y) =
	\begin{cases}
		h(t - x^{-1}) & \quad x \ge 1/t \\
		0             & \quad x < 1/t
	\end{cases}
\end{align*}
is a solution in $C^1([0, \infty) \times \R^2)$.

\subsection{Integrability of the derivatives of the vector field}

We will construct an autonomous vector field $b$ on $\R^2$ such that
\begin{gather}
	\begin{split}
		\begin{gathered}
			b \in W^{s, 1}_\loc(\R^2) \quad \text{for every $s \in (0, 1)$,} \\
			b \in L^p(\R^2) + L^\infty(\R^2) \quad \text{for every $1 \le p < 2$,}
			\qquad \text{$\div b = 0$ on $\R^2$ as a distribution,}
		\end{gathered}
	\end{split}
	\label{eq:ex-dint-b-properties}
\end{gather}
and there are two distinct regular Lagrangian flows for $b$. Here, $W^{s, 1}$ is a
fractional Sobolev space which will be explained later.
This example is from section IV.2 in \cite{DL89}, except that DiPerna and Lions
define $b$ slightly differently for $|x| \ge |y|$.

We define $b$ as follows:
\begin{align}
	b(x, y) \colonequals
	\begin{pmatrix}
		-\frac {x} {y^2} \sgn(y) \\
		-\frac 1 {|y|}
	\end{pmatrix}
	\quad \text{if $|x| < |y|$,} \qquad\quad
	b(x, y) = 0 \quad \text{if $|x| \ge |y|$.}
\end{align}
Note that $b$ comes from the Hamiltonian
\begin{align*}
	H(x, y) =
	\begin{cases}
		-\frac{x}{|y|} & \quad |x| < |y|  \\
		-1             & \quad x \ge |y|  \\
		1              & \quad x \le -|y|
	\end{cases}
	\qquad \text{i.e.} \qquad
	b_x = -\frac{\del H}{\del y}, \quad b_y = \frac{\del H}{\del x} \,.
\end{align*}
Therefore, $\div b = 0$ as a distribution on $\R^2$.

Before verifying the other properties listed in (\ref{eq:ex-dint-b-properties}),
we construct the regular Lagrangian flows.
For $(x, y) \in \R \times [0, \infty)$, define
\begin{align}
	 & Z(t, x, y) \colonequals
	\begin{pmatrix}
		\frac x y \sqrt{y^2 - 2t} \\
		\addlinespace
		\sqrt{y^2 - 2t}
	\end{pmatrix}
	\qquad \text{if $|x| < y$ and $t \le y^2/2$.}
	\label{eq:ex-X-def}
\end{align}
If $|x| \ge y$, we define $Z(t, x, y) = (x, y)$ for all $t \in \R$. We do not define $Z$ for
$|x| < y$ and $t > y^2/2$. Note that
\begin{align*}
	\del_t Z(t, x, y) = b(Z(t, x, y)) \qquad \text{if $t < y^2/2$.}
\end{align*}
Hence, for each $(x, y) \in \R \times [0, \infty)$, the curve $t \mapsto Z(t, x, y)$
is the trajectory along $b$ until the origin is reached.
Moreover,
\begin{align}
	\det DZ(t, x, y) = 1 \qquad \text{if $y > 0, \ |x| < y, \ t < y^2/2$.}
	\label{eq:Z-measure-preserving}
\end{align}

Beyond the origin, we extend the curves in the following two ways:
We define the $x$-components of two flows $\tilde X, \tilde Y$
for $(x, y) \in \R \times [0, \infty)$ by
\begin{align*}
	\tilde X_x(t, x, y) & =
	\begin{cases}
		Z_x(t, x, y)        & \quad t \le y^2/2 \\
		-Z_x(y^2 - t, x, y) & \quad t > y^2 / 2
	\end{cases} \\
	\addlinespace
	\tilde Y_x(t, x, y) & =
	\begin{cases}
		Z_x(t, x, y)                  & \quad t \le y^2/2 \\
		Z_x(y^2 - t, x, y)\phantom{-} & \quad t > y^2 / 2
	\end{cases}
\end{align*}
The $y$-components of $\tilde X$ and $\tilde Y$ are defined identically:
\begin{align*}
	\tilde X_y(t, x, y) = \tilde Y_y(t, x, y) =
	\begin{cases}
		Z_y(t, x, y)        & \quad t \le y^2/2 \\
		-Z_y(y^2 - t, x, y) & \quad t > y^2 / 2
	\end{cases}
\end{align*}

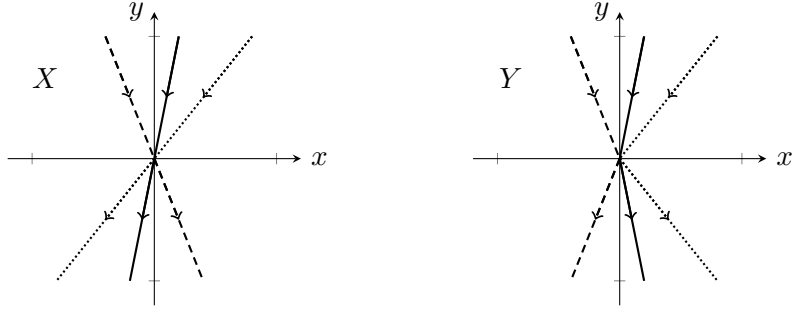
\begin{figure}
	\begin{center}
		\begin{tikzpicture}
			\begin{axis}[
					width=0.35\textwidth,
					height=0.35\textwidth,
					axis lines=middle,
					xmin=-1.2,
					xmax=1.2,
					ymin=-1.2,
					ymax=1.2,
					samples=200,
					xlabel={$x$},
					ylabel={$y$},
					xticklabels=\empty,
					yticklabels=\empty,
					xlabel style={ at={(current axis.right of origin)}, right },
					ylabel style={ at={(current axis.above origin)}, left },
				]
				\draw[->, thick] (0.2, 1) -- (0.1, 0.5);
				\draw[-, thick] (0.2, 1) -- (0, 0);
				\draw[->, thick] (0, 0) -- (-0.1, -0.5);
				\draw[-, thick] (0, 0) -- (-0.2, -1);

				\draw[densely dotted, ->, thick] (0.8, 1) -- (0.4, 0.5);
				\draw[densely dotted,-, thick] (0.8, 1) -- (0, 0);
				\draw[densely dotted,->, thick] (0, 0) -- (-0.4, -0.5);
				\draw[densely dotted,-, thick] (0, 0) -- (-0.8, -1);

				\draw[densely dashed, ->, thick] (-0.4, 1) -- (-0.2, 0.5);
				\draw[densely dashed,-, thick] (-0.4, 1) -- (0, 0);
				\draw[densely dashed,->, thick] (0, 0) -- (0.2, -0.5);
				\draw[densely dashed,-, thick] (0, 0) -- (0.4, -1);
			\end{axis}
			\node at (0.5, 3) {$X$};
		\end{tikzpicture}
		\hspace{4em}
		\begin{tikzpicture}
			\begin{axis}[
					width=0.35\textwidth,
					height=0.35\textwidth,
					axis lines=middle,
					xmin=-1.2,
					xmax=1.2,
					ymin=-1.2,
					ymax=1.2,
					samples=200,
					xlabel={$x$},
					ylabel={$y$},
					xticklabels=\empty,
					yticklabels=\empty,
					xlabel style={ at={(current axis.right of origin)}, right },
					ylabel style={ at={(current axis.above origin)}, left },
				]
				\draw[->, thick] (0.2, 1) -- (0.1, 0.5);
				\draw[-, thick] (0.2, 1) -- (0, 0);
				\draw[->, thick] (0, 0) -- (0.1, -0.5);
				\draw[-, thick] (0, 0) -- (0.2, -1);

				\draw[densely dotted, ->, thick] (0.8, 1) -- (0.4, 0.5);
				\draw[densely dotted,-, thick] (0.8, 1) -- (0, 0);
				\draw[densely dotted,->, thick] (0, 0) -- (0.4, -0.5);
				\draw[densely dotted,-, thick] (0, 0) -- (0.8, -1);

				\draw[densely dashed, ->, thick] (-0.4, 1) -- (-0.2, 0.5);
				\draw[densely dashed,-, thick] (-0.4, 1) -- (0, 0);
				\draw[densely dashed,->, thick] (0, 0) -- (-0.2, -0.5);
				\draw[densely dashed,-, thick] (0, 0) -- (-0.4, -1);
			\end{axis}
			\node at (0.5, 3) {$Y$};
		\end{tikzpicture}
	\end{center}
	\vspace{-1em}
	\caption{Flow lines of $X$ and $Y$.}
\end{figure}

Next, we extend $\tilde X$ and $\tilde Y$ to $y < 0$ using the symmetry of $b$.
Because $b(x, y) = b(-x, -y)$, we set
\begin{align*}
	X(t, x, y) \colonequals
	\begin{cases}
		\tilde X(t, x, y)      & \quad y \ge 0 \\
		- \tilde X(-t, -x, -y) & \quad y < 0
	\end{cases} \,.
\end{align*}
We also extend $\tilde Y$ to $y < 0$ by the same formula and call the resulting flow $Y$.

Let $(x, y) \in \R^2$ such that $|x| < |y|$.
Then, for all but one $t \in \R$, $X(\cdot, x, y)$ and $Y(\cdot, x, y)$
are differentiable and solve the ODE (\ref{eq:ode}) with initial data $(x, y)$.
Due to (\ref{eq:Z-measure-preserving}), $X$ and $Y$ are measure-preserving.

We now check the properties in (\ref{eq:ex-dint-b-properties}),
beginning with $b \in L^p(\R^2) + L^\infty(\R^2)$ for every $1 \le p < 2$.
Since $b$ is bounded on $\R^2 \setminus B_1$, it suffices
to show that $b \in L^p(B_1)$ for every $1 \le p < 2$.
First, consider $b_x = -x\sgn(y)/|y|^2$.
Since $|b_x|$ is invariant under negation of $y$,
it suffices to show that $b_x \in L^p(B_1 \cap \{y > 0\})$.
Because $b = 0$ if $|x| > |y|$,
\begin{align*}
	\|b_x\|_{L^p(B_1 \cap \{y > 0\})}^p
	 & \le \int_0^1 \contractint \int_{-y}^y |b_x|^p \,dx \,dy
	= \int_0^1 \contractint \int_{-y}^y \frac{|x|^p}{y^{2p}} \,dx \,dy \\
	 & \le \int_0^1 \contractint \int_{-y}^y \frac 1 {y^p} \,dx \,dy
	= 2 \int_0^1 \frac 1 {y^{p - 1}} \,dy < \infty \,.
\end{align*}
The computation for $b_y = -1/|y|$ is similar.

Next, we show that $b \in W^{s, 1}_\loc(\R^2)$ for every $0 < s < 1$.
For an open set $\Omega \subset \R^d$, the fractional Sobolev space
$W^{s, 1}(\Omega)$ is defined as (see \cite{DPV12}, p.~524)
\begin{align*}
	W^{s, 1}(\Omega) = \braces*{
		u \in L^1(\Omega) \ : \ \frac{|u(z)- u(w)|}{|z - w|^{d + s}} \in L^1(\Omega \times \Omega)
	} \,.
\end{align*}
We also write
\begin{align*}
	[u]_{W^{s, 1}(\Omega)} \colonequals \norm*{\frac{|u(z)- u(w)|}{|z - w|^{d + s}}}_{L^1(\Omega \times \Omega)}
	\quad \text{and} \quad
	\|u\|_{W^{s, 1}(\Omega)} \colonequals \|u\|_{L^1(\Omega)} + [u]_{W^{s, 1}(\Omega)} \,.
\end{align*}
The notation $b \in W^{s, 1}_\loc(\R^2)$ means that $b \in W^{s, 1}(\Omega)$ for every
bounded, open set $\Omega \subset \R^2$.
Because $b \in L^1(\R^2) + L^\infty(\R^2)$, we have $b \in L^1_\loc(\R^2)$, so we
only need to show that $[u]_{W^{s, 1}(B_R)} < \infty$ for every $R > 0$.

First, we show that $b$ is of bounded variation on $B_4 \setminus B_1$.
For simplicity, we will only consider $(B_4 \setminus B_1) \cap \{y > 0\}$.
In fact, since $b$ is symmetric around $y = 0$ up to signs,
\begin{align*}
	\|b\|_{BV(B_4 \setminus B_1)} = 2\|b\|_{BV((B_4 \setminus B_1) \cap \{y > 0\})} \,.
\end{align*}
The total variation of $b$ splits into the following two parts:
\begin{align}
	\|b\|_{BV((B_4 \setminus B_1) \cap \{y > 0\})}
	= \|b\|_{BV((B_4 \setminus B_1) \cap \{|x| < |y|, y > 0\})}
	+ 2 \sqrt 2 \int_{1/\sqrt 2}^{4 / \sqrt 2} |b(x, x)| \,dx
	\label{eq:b-BV-split}
\end{align}
The second term comes from the discontinuities at the lines $x = y$ and $x = -y$.
Since
\begin{align*}
	b \in C^1_b\big((B_4 \setminus B_1) \cap \{|x| < |y|, \ y > 0\}; \, \R^2 \big) \,,
\end{align*}
the first term in (\ref{eq:b-BV-split}) is finite. The integral in (\ref{eq:b-BV-split}) is
\begin{align*}
	\int_{1/\sqrt 2}^{4/\sqrt 2} |b(x, x)| \,dx
	= \sqrt 2 \int_{1/\sqrt 2}^{4 / \sqrt 2} \frac 1 x \,dx < \infty \,.
\end{align*}
Therefore $b \in BV(B_4 \setminus B_1)$.

For every bounded open set $\Omega \subset \R^d$ with Lipschitz
boundary, there is a continuous linear embedding $BV(\Omega) \hookrightarrow W^{s, 1}(\Omega)$
for all $0 < s < 1$, see Proposition \ref{prop:BV-Ws1}. Therefore also $b \in W^{s, 1}(B_4 \setminus B_1)$.

For $k \in \Z$, let $A_k \colonequals \{z \in \R^2 : 2^k < |z| < 2^{k + 1}\}$
and define for $k, \ell \in \Z$
\begin{align}
	I_{k, \ell} \colonequals \int_{A_k} \int_{A_\ell} \frac{|b(z) - b(w)|}{|z - w|^{2 + s}} \,dw \,dz \,.
	\label{eq:Ikl-def}
\end{align}
Then, for all $m \in \Z$
\begin{align}
	[b]_{W^{s, 1}(B_{2^m})} = \sum_{k < m} \sum_{\ell < m} I_{k, \ell} \,.
	\label{eq:b-Ws1-Ak}
\end{align}
We will now estimate $I_{k, \ell}$. If $k = \ell$, we have
\begin{align*}
	I_{k, k} =	[b]_{W^{s, 1}(A_k)} = \int_{A_k} \int_{A_k} \frac{|b(z) - b(w)|}{|z - w|^{2 + s}} \,dw \,dz \,.
\end{align*}
Using that $b$ is homogeneous of degree $-1$, we can express this in terms of $[b]_{W^{s, 1}(A_0)}$.
By changing variables $\zeta = 2^{-k}z$ and $\eta = 2^{-k} w$,
\begin{align*}
	I_{k, k} = \int_{A_0} \int_{A_0}
	\frac{2^{-k}|b(\zeta) - b(\eta)|}{2^{k(2 + s)}|\zeta - \eta|^{2 + s}}
	\, 2^{4k} \,d \zeta  \, d \eta
	= 2^{k(1 - s)} [b]_{W^{s, 1}(A_0)} \,.
\end{align*}
By the above discussion, we know that $b \in W^{s, 1}(A_0)$.
Moreover, the sum $\sum_{k < m} 2^{k(1 - s)}$ is finite since $s < 1$. Therefore, the contribution
of the diagonal terms $I_{k, k}$ to (\ref{eq:b-Ws1-Ak}) is finite.

Next, we look at the case $\ell = k + 1$. We note that
$I_{k, k + 1} \le [b]_{W^{s,1}(A_k \cup A_{k + 1})}$. Again by homogeneity of $b$,
\begin{align*}
	[b]_{W^{s,1}(A_k \cup A_{k + 1})} = 2^{k(1 - s)} [b]_{W^{s, 1}(A_0 \cup A_1)} \,.
\end{align*}
Because $A_0 \cup A_1 = B_4 \setminus B_1$ and $b \in W^{s, 1}(B_4 \setminus B_1)$,
the contribution of the terms $I_{k, k + 1}$ to (\ref{eq:b-Ws1-Ak})
is finite. By symmetry, also the contribution of the terms $I_{k, k - 1}$ is finite.

The remaining case is when $|k - \ell| \ge 2$.
Assume that $k < \ell$. If $z \in A_k$ and $w \in A_\ell$, we have
\begin{align*}
	|z - w| \ge 2^{\ell - 1}, \quad |b(z)| \le 2^{-k + 1},
	\quad |b(w)| \le 2^{-\ell + 1} \le 2^{-k+1} \,.
\end{align*}
Therefore
\begin{align*}
	I_{k, \ell}
	 & = \int_{A_k} \int_{A_\ell} \frac{|b(z) - b(w)|}{|z - w|^{2 + s}} \,dw \,dz
	\le \int_{A_k} \int_{A_\ell} \frac {2^{-k+2}} {2^{(\ell - 1)(2 + s)}} \,dw \,dz \\
	\addlinespace
	 & \le C_s \, 2^{-\ell(2 + s) - k} \cdot 2^{2k} \cdot 2^{2\ell}
	= C_s \, 2^{-\ell s + k}\,.
\end{align*}
Since $s < 1$,
\begin{align*}
	\sum_{\ell < m} \sum_{k \le \ell - 2} 2^{-\ell s + k}
	= \sum_{\ell < m} 2^{-\ell s + \ell - 1} < \infty \,,
\end{align*}
so the contribution of the terms $I_{k,\ell} \ (k \le \ell - 2)$ to (\ref{eq:b-Ws1-Ak})
is finite.
By symmetry, also the contribution from $k \ge \ell + 2$ is finite.
This concludes the proof that $b \in W^{s, 1}_\loc(\R^2)$.

\subsection{Integrability of the solution}

\colorlet{color1}{DarkOrchid2}

\colorlet{ct0}{color1!60}
\colorlet{ct1}{color1!150}
\colorlet{ctinfty}{color1!12}

\newcommand{\bb}{{\hat a}}
\newcommand{\bbd}{{\hat u}}
\newcommand{\ind}{{\mathcal K}}

\newcommand{\cubelen}{s}
\newcommand{\boxlen}{\ell}

Consider the transport equation with $c = 0$.
\begin{align}
	\del_t u - b \cdot \nabla u = 0, \qquad u|_{t = 0} = u^0
	\label{eq:transport-c-0}
\end{align}
For simplicity, we set $T = 1$ in this section. Let $1 \le p, q, r \le \infty$
such that $1/p + 1/q = 1$ and
\begin{align*}
	b \in L^1(0, 1; W^{1, r}_\loc(\R^d; \R^d) \cap L^q_\loc), \quad
	\div b \in L^1(0, 1; L^\infty(\R^d))
\end{align*}
such that (\ref{eq:b-growth-condition}) holds.
Then, by Corollaries \ref{cor:uniqueness-p} and \ref{cor:uniqueness-infty}, weak solutions
of (\ref{eq:transport-c-0}) in $L^\infty(0, 1; L^p(\R^d))$ are unique provided that
\begin{align}
	\frac 1 p + \frac 1 r \le 1 \,.
	\label{eq:dl-range}
\end{align}
Outside the range (\ref{eq:dl-range}), uniqueness of (\ref{eq:transport-c-0}) may fail.
In this section, we demonstrate this by a counterexample of Bruè, Colombo and Kumar \cite{BCK26}.
They showed that if $d \ge 2$ and
\begin{align}
	\frac 1 p + \frac {d - 1} {dr} > 1 \,,
	\label{eq:nonuniqueness-range}
\end{align}
then there exists a compactly supported vector field
\begin{align*}
	b \in C([0, 1]; W^{1, r}(\R^d; \R^d)) \cap C([0, 1]; L^q(\R^d; \R^d)) \,,
\end{align*}
for which there exist two nonnegative solutions of (\ref{eq:transport-c-0}) in $C([0, 1]; L^p(\R^d))$.

The full construction for the range (\ref{eq:nonuniqueness-range}) is not presented here,
but only the simpler version from section 2 of \cite{BCK26}. It covers the case $p = 1$ and $r < d$.

\begin{theorem}[\cite{BCK26}, Theorem 2.1]
	Let $d \ge 2$ and $r < d$. Then there exists $0 < \alpha < 1$ and a compactly supported
	vector field
	\begin{align*}
		b \in C([0, 1]; W^{1, r}(\R^d; \R^d)) \cap C^\alpha([0, 1] \times \R^d; \R^d),
		\qquad \div b = 0 \,,
	\end{align*}
	for which there are two nonnegative solutions of (\ref{eq:transport-c-0})
	in $L^\infty(0, 1; L^1(\R^d))$ with the same initial condition.
	\label{theorem:nonuniqueness}
\end{theorem}

Here, $C^\alpha$ is the space of
$\alpha$-Hölder continuous functions with the seminorm
\begin{align*}
	\norm*{f}_{C^\alpha} \colonequals
	\sup_{x \ne y} \frac{|f(x) - f(y)|}{|x - y|^\alpha} \,.
\end{align*}

\subsubsection{The self-similarity of the vector field}

First, let us introduce some notation. For $\ell > 0$ and $x \in \R^d$, we denote by $Q_\ell(x)$ the
closed cube of side length $\ell$ centered at $x$, that is
\begin{align}
	Q_\ell(x) \colonequals \braces*{
		y \in \R^d : \max{\braces*{|x_j - y_j| : 1 \le j \le d}} \le \ell/2
	} \,.
\end{align}
Let $g \in \N_{\ge 1}$ and $k \in \{-1, 1\}^g$. We define the centers of dyadic cubes of the $g$-th
generation in one dimension by
\begin{align*}
	c^1_g(k) \colonequals \frac 1 2 \sum_{j = 1}^g k_j 2^{-j} \,.
\end{align*}
The numbers $c^1_g(k)$ are evenly spaced in $[-1/2, 1/2]$ with distance $2^{-g}$. For example,
if $g = 2$ they are $-\frac 3 8, -\frac 1 8, \frac 1 8, \frac 3 8$.
For $k \in (\{-1, 1\}^g)^d$, we define the centers of dyadic cubes in $d$ dimensions by
\begin{align}
	c_g(k) \colonequals \paren*{c^1_g(k_j)}_{j = 1}^d \,.
\end{align}

For notational convenience, instead of directly constructing $b$, we will construct a vector field
$a(t, \cdot) = -b(1-t, \cdot)$ and two solutions to
\begin{align}
	\del_t u - a \cdot \nabla u = 0, \qquad u|_{t = 1} = u^1 \,.
	\label{eq:a-transport}
\end{align}
Their common value at $t = 1$ is $u^1 \colonequals \chi_{Q_1(0)}$. Since $a$ will be supported in $Q_1(0)$,
one solution to (\ref{eq:a-transport}) is the constant solution $\chi_{Q_1(0)}(x)$.
The second solution is denoted by $u$ and will satisfy
\begin{align}
	u|_{t = 0} = u^0 \colonequals 2^{\nu d} \chi_{P(2^{-(\nu + 1)}, 1)} \,,
	\qquad \text{where} \ \,
	P(\cubelen,\boxlen) \colonequals \bigcup_{k \in \{-1, 1\}^d} Q_s(\boxlen c_1(k)) \,.
	\label{eq:u0-def}
\end{align}
for some $\nu > 0$. Note that $\norm*{u^0}_{L^1} = \norm*{u^1}_{L^1} = 1$.
Figure \ref{fig:evo} shows the initial and final state of $u$ for $d = 2$.

\begin{figure}[ht!]
	\centering
	\begin{tikzpicture}[>=latex]
		\node at (2,4.4) {$t=0$};
		\node at (9,4.4) {$t=1$};

		\draw (0,0) rectangle (4,4);

		\draw[dashed] (2,0) -- (2,4);
		\draw[dashed] (0,2) -- (4,2);

		\fill[ct0] (0.4,2.4) rectangle (1.6,3.6);
		\fill[ct0] (2.4,2.4) rectangle (3.6,3.6);
		\fill[ct0] (0.4,0.4) rectangle (1.6,1.6);
		\fill[ct0] (2.4,0.4) rectangle (3.6,1.6);

		\node at (2,-0.5) {$u^0$};

		\draw[->,line width=0.4pt] (4.5,2.4) -- (6.5,2.4);
		\node at (5.5,2.75) {$a$};

		\draw[<-,line width=0.4pt] (4.5,1.6) -- (6.5,1.6);
		\node at (5.5,1.25) {$b$};

		\fill[ctinfty] (7,0) rectangle (11,4);
		\draw (7,0) rectangle (11,4);

		\node at (9,-0.5) {$u^1$};

		\node at (-0.3,4) {$\frac 1 2$};
		\node at (-0.3,2) {$0$};
		\node at (-0.45,0) {$-\frac 1 2$};

		\node at (11.5,4) {$\frac 1 2$};
		\node at (11.5,2) {$0$};
		\node at (11.5,0) {$-\frac 1 2$};
	\end{tikzpicture}
	\caption{
		The solution $u$ at $t = 0$ and $t = 1$. The intensity of
		the color indicates the magnitude of $u$. (From \cite{BCK26}, Figure 2)
	}
	\label{fig:evo}
\end{figure}
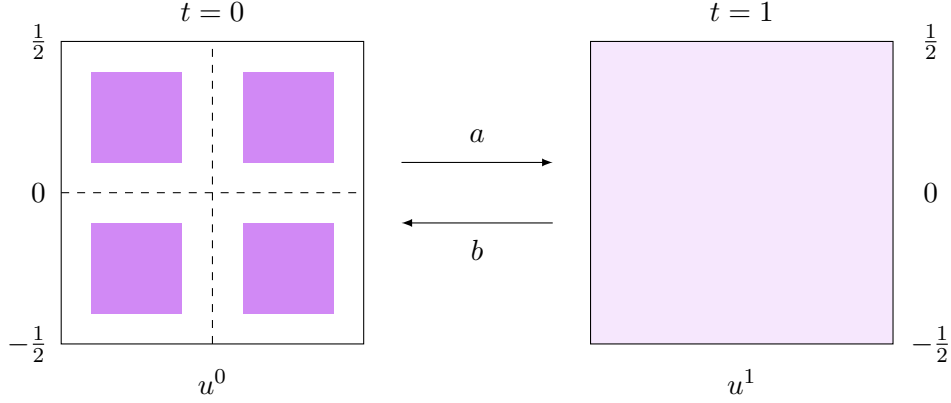

The vector field $a$ ``diffuses'' the $2^d$ cubes of $u^0$ by repeatedly breaking them into smaller
cubes while preserving the $L^1$-norm.
To formulate this precisely, we define
\begin{align*}
	t_i \colonequals 1 - 2^{-(i - 1)\beta},
	\qquad t_i^m \colonequals \frac{t_i + t_{i + 1}} 2
	\qquad\qquad \text{for $i \in \N_{\ge 1}$.}
\end{align*}
At time $t_i$, the solution $u$ is concentrated on $2^{di}$ cubes of side length $2^{-(\nu + 1)i}$,
centered at $c_i(k)$ for $k \in (\{-1, 1\}^i)^d$.
The value of $u$ on these cubes is $2^{\nu di}$ and 0 everywhere else, so that $\norm*{u(t_i)}_{L^1} = 1$.
From time $t_i$ to $t_i^m$, each cube is broken into $2^d$ smaller cubes of
side length $2^{-(\nu + 1)(i + 1)}$, and the magnitude of $u$ increases so that the $L^1$-norm is preserved.
At time $t_i^m$, the centers of the $2^{d(i + 1)}$ cubes do not align with the centers
of the dyadic cubes of generation $i + 1$. This is fixed by moving the cubes to
the centers from time $t_i^m$ to $t_{i + 1}$. The process is shown in Figure \ref{fig:break-and-move}.

Translating cubes which are separated by a positive distance is relatively straightforward.
In §\ref{sec:translating-cubes}, we explicitly construct a vector field $\bb$ for this task, and
we define $a$ in terms of $\bb$ on $[t_i^m, t_{i + 1})$.

The crucial part is subdividing each cube into smaller cubes.
The key observation of \cite{BCK26} is that breaking a cube into $2^d$ smaller cubes is
exactly the original problem, reversed in time. Thus, $a$ should be defined on $[t_i, t_i^m)$ as a collection of
$2^{di}$ time-reversed and scaled copies of itself.

\begin{figure}[ht!]
	\centering
	\begin{tikzpicture}[>=latex,scale=0.85]
		\begin{scope}
			\node at (2.1,4.5) {$t = t_1$};

			\draw (0,0) rectangle (4,4);
			\draw[dashed] (2,0)--(2,4);
			\draw[dashed] (0,2)--(4,2);

			\foreach \x/\y in {0.4/2.4,2.4/2.4,0.4/0.4,2.4/0.4}
				{
					\fill[ct0] (\x,\y) rectangle ++(1.2,1.2);
				}
		\end{scope}

		\draw[->] (4.4,2) -- (5.6,2);

		\pgfmathsetmacro{\sidelen}{0.3}

		\begin{scope}[shift={(6,0)}]
			\node at (2.17,4.5) {$t=t_1^m$};

			\draw (0,0) rectangle (4,4);

			\foreach \x in {1,2,3} \draw[dashed] (\x,0)--(\x,4);
			\foreach \y in {1,2,3} \draw[dashed] (0,\y)--(4,\y);

			\foreach \x in {1, 3}
			\foreach \y in {1, 3}
			\foreach \dx in {-0.3, 0.3}
			\foreach \dy in {-0.3, 0.3}
				{
					\fill[ct1] (\x+\dx - \sidelen/2,\y+\dy - \sidelen/2)
					rectangle ++(\sidelen,\sidelen);
				}
		\end{scope}

		\draw[->] (10.4,2) -- (11.6,2);

		\begin{scope}[shift={(12,0)}]
			\node at (2.1,4.5) {$t=t_2$};

			\draw (0,0) rectangle (4,4);

			\foreach \x in {1,2,3} \draw[dashed] (\x,0)--(\x,4);
			\foreach \y in {1,2,3} \draw[dashed] (0,\y)--(4,\y);

			\foreach \x in {0.5,1.5,2.5,3.5}
			\foreach \y in {0.5,1.5,2.5,3.5}
				{
					\fill[ct1] (\x-\sidelen/2,\y-\sidelen/2)
					rectangle ++(\sidelen,\sidelen);
				}

			\draw[dashed] (2,0)--(2,4);
			\draw[dashed] (0,2)--(4,2);
		\end{scope}
	\end{tikzpicture}
	\caption{
		The evolution from $t = t_1$ to $t = t_1^m$ consists of four scaled and time-reversed copies
		of the original problem.
		From time $t_1^m$ to $t_2$, the vector field $\bb$ is used to translate the cubes, so their
		centers align with $c_2(k)$ for $k \in (\{-1, 1\}^2)^2$.
		(From \cite{BCK26}, Figure 5)
	}
	\label{fig:break-and-move}
\end{figure}
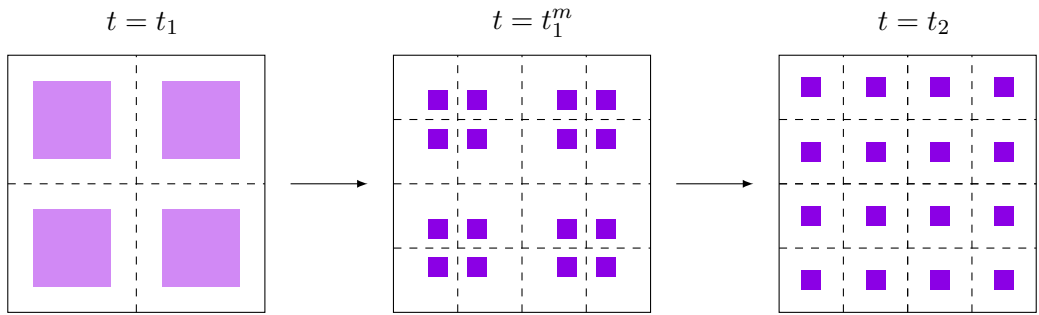

\subsubsection{The contraction map}

The self-similar nature of $a$ suggests constructing $a$ as the fixed point of a
contraction on a suitable complete metric space.
Let $A$ be the space of all
\begin{align}
	a \in C([0, 1]; W^{1, r}(\R^d; \R^d)) \cap C^\alpha([0, 1] \times \R^d; \R^d)
\end{align}
such that
\begin{align}
	a|_{t = 0} = a|_{t = 1} = 0, \qquad \div a = 0, \qquad \supp a \subset [0, 1] \times Q_1(0) \,.
	\label{eq:A-properties}
\end{align}
We equip $A$ with the norm
\begin{align}
	\norm*{a}_A \colonequals \norm*{\nabla a}_{C([0, 1]; L^r(\R^d))}
	+ \norm*{a}_{C^\alpha([0, 1] \times \R^d)} \,.
	\label{eq:A-norm}
\end{align}

\begin{lemma}
	If $r < d$ and $0 < \alpha < 1$, then $A$ is a Banach space.
\end{lemma}

\begin{proof}
	Since $r < d$, by Sobolev inequalities and $\supp a \subset [0, 1] \times Q_1(0)$,
	\begin{align}
		\norm*{a}_{L^r(\R^d)} \le C_{d, r} \norm*{\nabla a}_{L^r(\R^d)} \,.
		\label{eq:sobolev-inequalities}
	\end{align}
	Thus, if $\norm*{a}_A = 0$, then $a = 0$. Since the triangle inequality and homogeneity of $\norm*{\cdot}_A$
	are clear from the definition, $\norm*{\cdot}_A$ indeed defines a norm.

	To show that $A$ is complete, let $a_n \in A$ be a Cauchy sequence with respect to (\ref{eq:A-norm}).
	By (\ref{eq:sobolev-inequalities}) and the completeness of $C([0, 1]; W^{1,r}(\R^d))$, $a_n$ converges
	in $C([0, 1]; W^{1,r}(\R^d))$ to some $a_W$.

	Next, we observe that
	\begin{align}
		A \hookrightarrow C^{0, \alpha}([0, 1] \times Q_1(0); \R^d) \,,
		\label{eq:A-embedding}
	\end{align}
	where $C^{0, \alpha}$ is the space of functions $f$
	for which the norm
	\begin{align*}
		\norm*{f}_{C^{0, \alpha}} \colonequals
		\norm*{f}_{C_b} + \norm*{f}_{C^\alpha}
	\end{align*}
	is finite. The embedding (\ref{eq:A-embedding}) is continuous, since for every $a \in A$
	and $(t, x) \in (0, 1] \times Q_1(0)$
	\begin{align*}
		|a(t, x)| \le \frac{|a(t, x) - a(0, x)|}{t^\alpha} \le \norm*{a}_{C^\alpha([0, 1] \times \R^d)} \,.
	\end{align*}
	The first inequality holds because $a|_{t = 0} = 0$ and $0 < t \le 1$.
	Since $[0, 1] \times Q_1(0)$ is closed and bounded, $C^{0, \alpha}([0, 1] \times Q_1(0); \R^d)$ is a Banach
	space (see \cite{Eva10}, p.~255). Hence, $a_n$ also converges to some $a_\alpha$ in
	$C^{0, \alpha}([0, 1] \times Q_1(0); \R^d)$.

	The two limits $a_W$ and $a_\alpha$ must agree almost everywhere, since $a_n$ converges
	to $a_{\alpha}$ uniformly on $[0, 1] \times \R^d$, hence also in $L^1_\loc([0, 1] \times \R^d; \R^d)$.

	Finally, $a_W = a_\alpha$ also satisfies (\ref{eq:A-properties}) because the maps
	\begin{alignat*}{2}
		C^{0,\alpha}([0, 1] \times Q_1(0); \R^d) & \to C^{0, \alpha}(Q_1(0); \R^d)
		                                         & \hspace{4em}
		                                         & a \mapsto a|_{t = 0} \text{ or } a \mapsto a|_{t = 1} \,, \\
		W^{1, r}(\R^d; \R^d)                     & \to L^r(\R^d)
		                                         &
		                                         & a              \mapsto \div a
	\end{alignat*}
	are continuous.
\end{proof}

We now define a contraction $F$ on $A$. To this end, we introduce the following notations.
\begin{align}
	\Delta_i \colonequals t_i^m - t_i = 2^{-i\beta} \paren*{\frac{2^\beta - 1} 2}, \qquad
	\ind_i^d \colonequals (\{-1, 1\}^i)^d
	\label{eq:Delta-index-def}
\end{align}
If $t \in [t_i, t_i^m)$, we define
\begin{align}
	F(a)(t, x) \colonequals -\frac 1 {2^{(\nu + 1)i} \Delta_i} \,
	\sum_{k \in \ind_i^d}
	a{\paren*{
				\frac{t_i^m - t}{\Delta_i}, \,
				2^{(\nu + 1)i}(x - c_i(k))
			}
		} 
	\label{eq:F-def-recursive}
\end{align}
The scaling factor in space is $2^{-(\nu + 1)i}$, since each block at time $t_i$ has
side length $2^{-(\nu + 1)i}$.
If $t \in [t_i^m, t_{i + 1})$, we define
\begin{align}
	F(a)(t, x) \colonequals \frac 1 {2^i \Delta_i} \,
	\sum_{k \in \ind_i^d}
	\bb{\paren*{
			\frac {t - t_i^m}{\Delta_i}, \,
			2^i\big(x - c_i(k)\big); \,
			2^{-\nu(i + 1) - 1}, \,
			2^{-\nu i}
		}
	}
	\label{eq:F-def-bb}
\end{align}
Here $\bb(t, x; \cubelen, \boxlen)$ is the vector field from Proposition \ref{prop:bb-ex} below,
which translates $2^d$ cubes of side length $\cubelen$, evenly arranged in $Q_\boxlen(0)$,
to the centers of the first-generation dyadic cubes, namely $c_1(k)$ for $k \in \ind_1^d$.
The cubes should be moved to the centers of the $(i + 1)$-th generation dyadic cubes,
hence the space variable is scaled by $2^i$. Since our cubes have side length $2^{-(\nu + 1)(i + 1)}$,
we set the parameter $s$ of $\bb$ to $2^{-(\nu + 1)(i + 1)} / 2^{-i} = 2^{-\nu(i + 1) - 1}$.
Finally, we define $F(a)(1, x) = 0$.

To prove that $F$ maps $A$ to $A$ and is a contraction, we need estimates for $\bb$.
These are collected in Proposition \ref{prop:bb-ex}, together with properties of the corresponding
solution $\bbd$. The proof is deferred to §\ref{sec:translating-cubes}. As described above,
the values of $\bbd$ at $t = 0$ and $t = 1$ are
\begin{align}
	\bbd^0 = \chi_{P(\cubelen,\boxlen)} \,, \quad \bbd^1 = \chi_{P(\cubelen, 1)} \,,
	\qquad \text{where} \ \,
	P(\cubelen,\boxlen) \colonequals \bigcup_{k \in \{-1, 1\}^d} Q_s(\boxlen c_1(k)) \,.
\end{align}

\begin{proposition}
	Let $0 < 4\cubelen < \boxlen < 1$. Then there exists
	\begin{align*}
		\bb(t, x; \cubelen, \boxlen) \in C^\infty_c([0, 1] \times \R^d; \R^d), \quad \text{such that} \ \
		\div \bb = 0, \ \supp \bb \subset [1/3, 2/3] \times Q_1(0)
	\end{align*}
	and $\bbd(t, x; \cubelen, \boxlen) \in L^\infty([0, 1] \times \R^d)$, $\bbd \ge 0$ such that
	\begin{align}
		\del_t \bbd - \bb \cdot \nabla \bbd = 0 \quad \text{on $[0, 1] \times \R^d$}, \qquad
		\bbd|_{t \in [0, 1/3]} = \bbd^0, \quad \bbd|_{t \in [2/3, 1]} = \bbd^1 \,.
		\label{eq:bb-transport}
	\end{align}
	Furthermore
	\begin{equation}
		\begin{split}
			\begin{gathered}
				\meas^d \paren*{\supp \bb(t, \cdot)} \le 4^d \cubelen^d \quad
				\text{for all $t \in [0, 1]$,} \\
				\addlinespace
				\norm*{\bb}_{L^\infty_{t, x}} \le C_d, \quad
				\norm*{\nabla_{(t, x)} \bb}_{L^\infty_{t, x}} \le C_d / \cubelen \,.
			\end{gathered}
			\label{eq:building-block-meas}
		\end{split}
	\end{equation}
	\label{prop:bb-ex}
\end{proposition}

\begin{proposition}
	Let $d \ge 2$ and $0 < \alpha, \beta < 1$. If
	\begin{align}
		\nu > 1 - \log_2(2^\beta - 1)\,, \qquad
		r < \frac{\nu d}{\nu + \beta}\,,\qquad
		\alpha < \frac{1 - \beta}{1 + \nu}\,,
		\label{eq:bb-contraction-assumptions}
	\end{align}
	then $F$ maps $A$ to $A$ and is a contraction.
	\label{prop:F-contraction}
\end{proposition}

\begin{proof}
	Let $a \in A$. We first estimate $\norm*{\nabla F(a)(t)}_{L^r_x}$.
	If $t \in [t_i, t_i^m)$, by the chain rule
	\begin{align*}
		\nabla F(a)(t) = \frac 1 {\Delta_i} \sum_{k \in \ind_i^d}
		\nabla a{\paren*{
					\frac{t_i^m - t}{\Delta_i}, \,
					2^{(\nu + 1)i}(x - c_i(k))
				}
			} \,.
	\end{align*}
	Since the summands have disjoint support,
	\begin{align*}
		\norm*{\nabla F(a)(t)}_{L^r}
		= \frac 1 {\Delta_i} \paren*{
			\sum_{k \in \ind_i^d} \int_{\R^d}
			\abs*{\nabla a{\paren*{
							\frac{t_i^m - t}{\Delta_i}, \,
							2^{(\nu + 1)i}(x - c_i(k))
						}
					}
			}^{r} \,dx
		}^{\!1/r} \!.
	\end{align*}
	By changing variables $x' = 2^{(\nu + 1)i}(x - c_i(k))$ and using that $|\ind_i^d| = 2^{id}$, we obtain
	\begin{align}
		\norm*{\nabla F(a)(t)}_{L^r}
		= \frac 1 {\Delta_i} \paren*{
			\sum_{k \in \ind_i^d} \frac 1 {2^{(\nu + 1)id}}
			\norm*{\nabla a((t_i^m - t)/\Delta_i)}_{L^r_x}^r
		}^{\!1/r} \!\!
		= \frac 1 {\Delta_i 2^{\nu di / r}} \norm*{\nabla a}_{C_tL^r_x}  \,.
		\label{eq:Lr-estimate}
	\end{align}
	Substituting the definition of $\Delta_i$ (\ref{eq:Delta-index-def}), the constant on
	the right-hand side is
	\begin{align}
		\frac 1 {\Delta_i 2^{\nu di / r}}
			= 2^{i(\beta - \nu d / r)} \frac 2 {2^{\beta} - 1}
		< 2^{-i\nu} \frac 2 {2^{\beta} - 1}
		< 1 \,.
		\label{eq:Lr-estimate-constant}
	\end{align}
	Here, we used that $\nu d / r > \nu + \beta$ and $\nu > 1 - \log_2(2^\beta - 1)$.

	Let $a, a' \in A$. By definition, $F(a)(t) = F(a')(t)$ for every
	$t$ that is not in $[t_i, t_i^m)$ for some $i \in \N_{\ge 1}$. Therefore,
	by (\ref{eq:Lr-estimate}) and (\ref{eq:Lr-estimate-constant})
	\begin{align}
		\norm*{\nabla (F(a) - F(a'))}_{C_tL^r_x}
		= \norm*{\nabla F(a - a')}_{C_tL^r_x}
		< \norm*{\nabla(a - a')}_{C_tL^r_x} \,.
		\label{eq:Lr-contraction}
	\end{align}

	If $t \in [t_i^m, t_{i + 1})$, then by a similar computation as above and estimates
	for the building block vector field (\ref{eq:building-block-meas})
	\begin{align}
		\begin{split}
			\norm*{\nabla F(a)(t)}_{L^r}
			 & \le \frac 1 {\Delta_i}
			\norm*{\nabla \bb \paren*{
					t, x ; 2^{-\nu(i + 1) - 1}, \,	2^{-\nu i}
				}
			}_{C_t L^r_x}
			\le C_{\nu, d, r} \frac{2^{\nu i(1 - d/r)}}{\Delta_i}                         \\
			 & = C_{\nu, d, r} \, 2^{i(\nu - \nu d / r + \beta)} \, \frac 2 {2^\beta - 1}
			< C_{\nu, d, r} \, \frac 2 {2^\beta - 1} \,.
		\end{split}
		\label{eq:F-bb-Lr-estimate}
	\end{align}
	In the last line, we used that $r < \frac {\nu d}{\nu + \beta}$. From (\ref{eq:Lr-estimate}),
	(\ref{eq:Lr-estimate-constant}) and (\ref{eq:F-bb-Lr-estimate}), we conclude
	\begin{align}
		\norm*{\nabla F(a)}_{C_tL^r_x} \le C_{\nu, d, r, \beta} + \norm*{\nabla a}_{C_tL^r_x} \,.
		\label{eq:bounded-nabla-Lr}
	\end{align}

	We now turn to the Hölder part of the norm (\ref{eq:A-norm}). For this, we note that for
	any two functions $f, g$,
	\begin{align}
		\supp f \cap \supp g = \emptyset \quad \Longrightarrow \quad
		\norm*{f + g}_{C^\alpha} \le \max\{\norm*{f}_{C^\alpha}, \norm*{g}_{C^\alpha}\} \,.
		\label{eq:alpha-supp}
	\end{align}
	Since $\supp_t \bb \subset [1/3, 2/3]$,
	\begin{align}
		\norm*{F(a)}_{C^\alpha([0, 1] \times \R^d)}
		\le \sup_{i \in \N} \paren*{ \max \braces*{
				\norm*{F(a)}_{C^\alpha([t_i, t_i^m] \times \R^d)}, \,
				\norm*{F(a)}_{C^\alpha([t_i^m, t_{i + 1}] \times \R^d)}
			}
		} \,.
		\label{eq:alpha-time-split}
	\end{align}
	We first look at the intervals $[t_i, t_i^m]$. Since $\nu > 0$, the summands in (\ref{eq:F-def-recursive})
	have disjoint support in space. Hence it suffices to estimate them separately by (\ref{eq:alpha-supp}).
	For every $k \in \ind_i^d$,
	\begin{align*}
		 & \hspace{-1em} \norm*{
			a \paren*{\frac{t_i^m - t}{\Delta_i}, 2^{(\nu + 1)i}(x - c_i(k))}
		}_{C^\alpha([t_i, t_i^m] \times \R^d)}                                \\
		\addlinespace
		 & \hspace{1em} = \!\! \sup_{\substack{s, t \in [t_i, t_i^m], s \ne t \\ x \ne y}} \!\!
		\frac{
			\abs*{
				a\paren*{\frac{t_i^m - t}{\Delta_i}, 2^{(\nu + 1)i}(x - c_i(k))} -
				a\paren*{\frac{t_i^m - s}{\Delta_i}, 2^{(\nu + 1)i}(y - c_i(k))}
			}
		}
		{\big(|t - s|^2 + |x - y|^2\big)^{\alpha/2}}
		\\ \addlinespace
		 & \hspace{1em} = \sup_{\substack{s, t \in [0, 1], s \ne t            \\ x \ne y}}
		\frac{
			\abs*{
				a\paren*{t, x} - a\paren*{s, y}
			}
		}
		{\big(\Delta_i^2 |t - s|^2 + 2^{-2(\nu + 1)i} |x - y|^2\big)^{\alpha/2}}
		\\ \addlinespace
		 & \hspace{1em} \le \max\braces*{
			\frac 1 {\Delta_i}, 2^{(\nu + 1)i}
		}^{\!\alpha} \,
		\norm*{a}_{C^\alpha([0, 1] \times \R^d)} \,.
	\end{align*}
	Therefore, by the definition of $F$ on $[t_i, t_i^m)$,
	\begin{align}
		\norm*{F(a)}_{C^\alpha([t_i, t_i^m] \times \R^d)}
		\le \frac 1 {2^{(\nu + 1)i}\Delta_i} \max\braces*{
			\frac 1 {\Delta_i}, 2^{(\nu + 1)i}
		}^{\!\alpha} \,
		\norm*{a}_{C^\alpha([0, 1] \times \R^d)} \,.
		\label{eq:alpha-max-estimate}
	\end{align}
	We estimate the two arguments of the maximum separately. For $1/\Delta_i$, by
	the definition of $\Delta_i$ in (\ref{eq:Delta-index-def}) and since
	$2^\nu > 2/(2^\beta - 1)$,
	\begin{align*}
		\frac 1 {2^{(\nu + 1)i}\Delta_i^{1 + \alpha}}
		< 2^{-(1 + \nu)i} \, 2^{\beta(1 + \alpha)i} \, 2^{\nu(1 + \alpha)} \,.
	\end{align*}
	Note that $\nu > 1$ since $\beta < 1$. Since $\beta(1 + \alpha) < 2$,
	the above expression is largest for $i = 1$. Using that $\beta\alpha < \alpha$, we obtain
	\begin{align}
		\frac 1 {2^{(\nu + 1)i}\Delta_i^{1 + \alpha}} <
		2^{-(1 + \nu) + \beta + \beta\alpha + \nu + \nu \alpha}
		< 2^{-1 + \beta + \alpha + \nu \alpha}
		< 1 \,.
		\label{eq:max-first-arg}
	\end{align}
	The last inequality follows from $\alpha < \frac{1 - \beta}{1 + \nu}$.
	For the argument $2^{(\nu + 1)i}$ of the maximum in (\ref{eq:alpha-max-estimate}),
	by $\alpha < \frac{1 - \beta}{1 + \nu}$ we have
	\begin{align}
		\frac 1 {2^{(\nu + 1)(1 - \alpha)i}\Delta_i}
		< 2^{(1 - \beta)i} \, 2^{-(\nu + 1)i} \, 2^{\beta i} \, \frac 2 {2^\beta-1}
		= 2^{-\nu i} \, \frac 2 {2^\beta-1}
		< 1 \,.
		\label{eq:max-second-arg}
	\end{align}
	From (\ref{eq:alpha-max-estimate}), (\ref{eq:max-first-arg}) and (\ref{eq:max-second-arg}) we conclude
	\begin{align}
		\norm*{F(a)}_{C^\alpha([t_i, t_i^m] \times \R^d)} \le C_{\nu, \beta} \norm*{a}_{C^\alpha([0, 1] \times \R^d)},
		\label{eq:alpha-rec-contraction}
	\end{align}
	where $C_{\nu,\beta} < 1$ is independent of $i$.

	Let $a, a' \in A$. Because $F(a)(t) = F(a')(t)$ for every $t$ that is not in $[t_i, t_i^m)$ for some $i$,
	\begin{align}
		\begin{split}
			\norm*{F(a) - F(a')}_{C^\alpha_{t, x}}
			 & = \sup_{i \in \N} \ \norm*{F(a) - F(a')}_{C^\alpha([t_i, t_i^m] \times \R^d)} \\
			 & = \sup_{i \in \N} \ \norm*{F(a - a')}_{C^\alpha([t_i, t_i^m] \times \R^d)}
			< \norm{a - a'}_{C^\alpha([0, 1] \times \R^d)}
		\end{split}
		\label{eq:alpha-contraction}
	\end{align}

	It remains to estimate the $C^\alpha$-norm of $F(a)$ on
	the intervals $[t_i^m, t_{i + 1}]$. By definition of $C^\alpha$ and (\ref{eq:F-def-bb}),
	\begin{align*}
		\norm*{F(a)}_{C^\alpha([t_i^m, t_{i + 1}] \times \R^d)}
		 & \le 2\norm*{F(a)}_{L^\infty([t_i^m, t_{i + 1}] \times \R^d)}^{1 - \alpha} \,
		\norm*{\nabla_{(t, x)} F(a)}_{L^\infty([t_i^m, t_{i + 1}] \times \R^d)}^\alpha  \\
		\addlinespace
		 & \hspace{-5em} \le \frac 2 {2^i \Delta_i}
		\norm*{\bb}_{L^\infty_{t, x}}^\alpha \,
		\max\braces*{ \frac 1 {\Delta_i}, 2^i}^\alpha
		\norm*{\nabla_{(t, x)}
			\bb \paren*{t, x; 2^{-\nu(i + 1) - 1}, 2^{-\nu i}}
		}_{L^\infty_{t, x}}^\alpha
	\end{align*}
	We estimate $\max\{1/\Delta_i, 2^i\} \le C_\beta 2^i$ and use (\ref{eq:building-block-meas}),
	\begin{align}
		\begin{split}
			\norm*{F(a)}_{C^\alpha([t_i^m, t_{i + 1}] \times \R^d)}
			 & \le C_{d,\alpha,\beta} \, 2^{-i} \, 2^{i\beta} \, 2^{i\alpha} \, 2^{(\nu(i + 1) + 1)\alpha} \\
			 & \le C_{d, \alpha, \nu, \beta} 2^{-i(1 - \alpha - \beta - \alpha \nu)} \,.
		\end{split}
		\label{eq:alpha-bb-estimate}
	\end{align}
	Since $\alpha < \frac{1 - \beta}{1 + \nu}$, this is largest for $i = 1$.

	We can now complete the proof of the Proposition. By (\ref{eq:bounded-nabla-Lr}), (\ref{eq:alpha-time-split}),
	(\ref{eq:alpha-rec-contraction}) and (\ref{eq:alpha-bb-estimate}), $F$ maps $A$ to $A$.
	By (\ref{eq:Lr-contraction}) and (\ref{eq:alpha-contraction}), $F$ is a contraction.
\end{proof}

\subsubsection{Construction of the solution $u$}

Analogously to the construction of the vector field, we obtain $u$ as the fixed point
of a contraction.
Let $1 < \gamma, \lambda < \infty$ and let $U$ be the space of all
\begin{align*}
	u \in L^\gamma(0, 1; L^1(\R^d)) \cap \Lip([0, 1]; W^{-1,\lambda}(\R^d))
\end{align*}
such that
\begin{align*}
	u|_{t = 0} = u^0, \quad u_{t = 1} = u^1, \quad u \ge 0 \ \text{a.e.}, \quad
	\supp u \subset [0, 1] \times Q_1(0) \,.
\end{align*}
Here $\Lip([0, 1]; W^{-1,\lambda}(\R^d))$ is the set of Lipschitz-continuous functions
$[0, 1] \to W^{-1,\lambda}(\R^d)$.
We define the following seminorm on $\Lip([0, 1]; W^{-1,\lambda}(\R^d))$.
\begin{align*}
	\norm*{u}_{\Lip([0, 1]; W^{-1, \lambda}(\R^d))} \colonequals
	\sup_{s, t \in [0, 1], t \ne s} \! \frac{\norm*{u(s,\cdot) - u(t,\cdot)}_{W^{-1,\lambda}(\R^d)}}{|s - t|}
\end{align*}
The norm of $U$ is given by
\begin{align*}
	\norm*{u}_U \colonequals \norm*{u}_{L^\gamma(0, 1; L^1(\R^d))}
	+ \norm*{u}_{\Lip([0, 1]; W^{-1,\lambda}(\R^d))} \,.
\end{align*}

We define a map $G$ on $U$ similarly to $F$. Let $u \in U$. If $t \in [t_i, t_i^m)$, we set
\begin{align}
	G(u)(t, x) = 2^{\nu di} \sum_{k \in \ind_i^d} u\paren*{
		\frac{t_i^m - t}{\Delta_i}, \,
		2^{(\nu + 1)i} (x - c_i(k))
	} \,.
	\label{eq:G-def-recursive}
\end{align}
If $t \in [t_i^m, t_{i + 1})$, we set
\begin{align}
	G(u)(t, x) = 2^{\nu d (i + 1)} \sum_{k \in \ind_i^d} \bbd\paren*{
		\frac{t - t_i^m}{\Delta_i}, \,
		2^i(x - c_i(k)); \,
		2^{-\nu(i + 1) - 1}, \,
		2^{-\nu i}
	} \,,
	\label{eq:G-def-bb}
\end{align}
where $\bbd$ is from Proposition \ref{prop:bb-ex}.
If $t = 1$, we define $G(u)(t, x) = u^1(x) = \chi_{Q_1(0)}$.

\begin{proposition}
	Let $d \ge 2$ and $1 < \gamma, \lambda < \infty$. If $\nu \ge 2$ and
	\begin{align}
		\frac{\nu d}{1 - \beta} < \frac \lambda {\lambda - 1} \,,
		\label{eq:solution-contraction-assumptions}
	\end{align}
	then $G$ maps $U$ to $U$ and is a contraction.
	\label{prop:G-contraction}
\end{proposition}

We omit the proof of Proposition \ref{prop:G-contraction}.
It can be found in \cite{BCK26}, Proposition 2.4.

\begin{proof}[Proof of Theorem \ref{theorem:nonuniqueness}]
	For every $0 < \beta < 1$, (\ref{eq:bb-contraction-assumptions})
	is satisfied for large enough $\nu$ since $r < d$. By choosing $\lambda > 1$ close enough to $1$, also
	(\ref{eq:solution-contraction-assumptions}) holds.
	By the Banach fixed-point theorem and Propositions \ref{prop:F-contraction} and \ref{prop:G-contraction},
	there exist unique fixed points $a \in A$ and $u \in U$ of $F$ and $G$ respectively.

	Let $\phi \in C^\infty_c(\R^d)$. To prove that $a$ and $u$ satisfy (\ref{eq:a-transport}),
	we need to show that $\int_{\R^d} u \phi \,dx$ is almost everywhere differentiable in $[0, 1]$
	and for a.e. $t \in [0, 1]$
	\begin{align}
		\frac d {dt} \int_{\R^d} u \phi \,dx = - \int_{\R^d} u a \cdot \nabla \phi \,dx \,.
		\label{eq:transport-c-0-weak}
	\end{align}
	Since
	\begin{align*}
		\abs*{\int_{\R^d} u(t, x) \phi \,dx - \int_{\R^d} u(s, x) \phi \,dx}
		\le \norm*{u(t) - u(s)}_{W^{-1,\lambda}} \norm*{\phi}_{W^{1,\lambda}} \,,
	\end{align*}
	and $u \in \Lip([0, 1]; W^{-1,\lambda})$, the derivative on the left-hand side of
	(\ref{eq:transport-c-0-weak}) exists for almost every $t \in [0, 1]$.
	Let $M \subset [0, 1]$ be the set where $\int_{\R^d} u \phi \,dx$ is differentiable
	and (\ref{eq:transport-c-0-weak}) holds. Since $a \in C_c([0, 1] \times \R^d)$ and
	$u \in L^\gamma(0, 1; L^1(\R^d))$, the right-hand side of (\ref{eq:transport-c-0-weak}) is a
	measurable function on $[0, 1]$, hence $M$ is measurable.

	By definitions (\ref{eq:F-def-bb}) and (\ref{eq:G-def-bb}), on $[t_i^m, t_{i + 1})$ for all $i \in \N_{\ge 1}$,
	$F(a)$ and $G(u)$ consist of scaled and translated copies of $\bb$ and $\bbd$.
	Thus $\bigcup_{i \in \N_{\ge 1}} (t_i^m, t_{i + 1}) \subset M$ by Proposition \ref{prop:bb-ex}
	and since $a$ and $u$ are fixed points.
	On the intervals $[t_i, t_i^m)$,
	$F(a)$ and $G(u)$ are defined as scaled and translated copies of $a$ and $u$.
	Because $a$ and $u$ are fixed points and the measure of $\bigcup_{i \in \N_{\ge 1}} [t_i, t_i^m)$
	is $1/2$, we conclude
	\begin{align*}
		\meas^1(M) \ge \frac 1 2 + \frac {\meas^1(M)} 2 \,.
	\end{align*}
	Thus $\meas^1(M) = 1$.

	It remains to show that $u \in L^\infty(0, 1; L^1(\R^d))$. By choosing $\phi$ in
	(\ref{eq:transport-c-0-weak}) such that $\phi|_{Q_1(0)} = 1$, we see that
	\begin{align*}
		\int_{\R^d} u(t, x) \,dx = \int_{\R^d} u^1 \,dx = 1
	\end{align*}
	for every $t \in [0, 1]$. Since $u \ge 0$, also $\norm*{u(t)}_{L^1_x} = 1$ for every $t \in [0, 1]$.
\end{proof}

\subsubsection{Translating cubes}
\label{sec:translating-cubes}

We now show how to construct $\bb$ and $\bbd$ with the properties listed in Proposition \ref{prop:bb-ex}.
The construction here differs from the one in section 7 of \cite{BCK26}.
The extension of the two-dimensional vector field $\boldsymbol{b}_2$ used there is not a
vector field on $\R^d$ when $d \ge 3$.

Let $\xi \in \R^d$ such that $|\xi| = 1$.
Let $\psi = \eta(2|x|) \in C^\infty_c((-1, 1))$, where $\eta$ is
a smooth cutoff function defined in (\ref{eq:eta-R}). With $\Psi(x) \colonequals \prod_{i = 1}^d \psi(x_i)$,
we define the tensor
\begin{align*}
	\omega_{ij}^\xi(x) \colonequals \big(\xi_i x_j - \xi_j x_i \big) \Psi(x) \,,
\end{align*}
and the vector field $\Gamma^\xi : \R^d \to \R^d$ by
\begin{align}
	\Gamma_i^\xi \colonequals \frac 1 {d - 1} \sum_{j = 1}^d \del_j \omega_{ij}^\xi
\end{align}
Since $\supp \Psi \subset Q_2(0)$, also $\Gamma^\xi$ is supported in $Q_2(0)$. Moreover,
\begin{align*}
	\Gamma^\xi_i|_{Q_1(0)} = \frac 1 {d - 1} \sum_{j = 1}^d \del_j \big(\xi_i x_j - \xi_j x_i \big)
	= \frac 1 {d - 1} \sum_{j = 1}^d \big(\xi_i - \xi_j \delta_{ij} \big)
	= \xi_i \,.
\end{align*}
Since $\omega$ is antisymmetric and smooth,
\begin{align*}
	\div \Gamma^\xi = \frac 1 {d - 1} \sum_{i = 1}^d \del_i \sum_{j = 1}^d \del_j \omega_{ij} = 0 \,.
\end{align*}

Next, we define a vector field $\tau(t, x; \cubelen, z^0, z^1)$ that moves a cube of side length $\cubelen$
from $z^0 \in \R^d$ to $z^1 \in \R^d$. Let $\xi \colonequals (z^1 - z^0)/|z^1 - z^0|$ and
$\zeta \in C^\infty([0, 1]; [0, 1])$ such that $\zeta|_{[0, 1/3]} = 0$ and $\zeta_{[2/3, 1]} = 1$.
We set
\begin{align*}
	z(t) \colonequals z^0 + \zeta(t)(z^1 - z^0)
\end{align*}
and define
\begin{align}
	\tau(t, x; \cubelen, z^0, z^1) \colonequals -|z'(t)| \, \Gamma^\xi \paren*{
		\frac{x - z(t)}{\cubelen}
	}
	\qquad \text{for $t \in [0, 1], \, x \in \R^d$} \,.
	\label{eq:tau-def}
\end{align}
Finally, we arrange $2^d$ copies of $\tau$:
\begin{align}
	\bb(t, x; \cubelen, \boxlen) \colonequals
	\sum_{k \in \ind_1^d} \tau\big(t, x; \cubelen, \boxlen c_1(k), c_1(k)\big) \,.
	\label{eq:bb-def}
\end{align}
The solution $\bbd$ is defined by
\begin{align*}
	\bbd(t, x) \colonequals \chi_{P(\cubelen, \boxlen (1 - \zeta(t)) + \zeta(t))}(x) \,,
	\qquad \text{where} \ \,
	P(\cubelen,\boxlen) \colonequals \bigcup_{k \in \{-1, 1\}^d} Q_s(\boxlen c_1(k)) \,.
\end{align*}

By (\ref{eq:tau-def}),
\begin{align*}
	|\del_t \tau|
	\le C_{\zeta, \Gamma^\xi} \paren*{|z_1 - z_0| + \frac{|z_1 - z_0|^2} {\cubelen}} \,, \qquad
	|\nabla_x \tau|
	\le C_{\zeta, \Gamma^\xi} \, \frac{|z_1 - z_0|} \cubelen \,.
\end{align*}
In the definition of $\bb$, we set $z^0 = \boxlen c_1(k)$ and $z_1 = c_1(k)$. Since $|c_1(k)| \le \sqrt {d/2}$
for all $k \in \ind_1^d$, and since $0 < \boxlen < 1$,
\begin{align*}
	|\nabla_{(t, x)} \bb| \le 2^d \, C_{\zeta, \Gamma^\xi} \paren*{
		|1 - \boxlen|\sqrt{d/2} + \frac {|1 - \boxlen|\sqrt{d/2} + |1 - \boxlen|^2 \, d/2}{\cubelen} }
	\le C_{\zeta, \Gamma^\xi, d}\, \frac 1 \cubelen \,.
\end{align*}
The last inequality follows because $0 < \cubelen < 1$.
Since $\meas^d(\supp \Gamma^\xi) \le 2^d$ and $|\ind_1^d| = 2^d$, we have
$\meas^d(\supp \bb(t, \cdot)) \le 4^d \cubelen^d$ for all $t \in [0, 1]$.

\newpage

\appendix

\section*{Appendix}
\addcontentsline{toc}{section}{Appendix}

\setcounter{equation}{0}
\renewcommand{\theequation}{A.\arabic{equation}}

\theoremstyle{theorem}
\newtheorem{appendixproposition}{Proposition}
\renewcommand{\theappendixproposition}{A.\arabic{appendixproposition}}

In the following Proposition, we denote by $L^1(0, T; C^1_\loc(\R^d; \R^d))$ the space of
functions $f : [0, T] \times \R^d \to \R^d$ such that $f \in L^1(0, T; C^1_b(K;\R^d))$
for every compact $K \subset \R^d$.

\begin{appendixproposition}
	Let $b \in L^1(0, T; C^1_{\loc}(\R^d; \R^d))$. Let $\Omega \subset \R^d$
	be open and $X : [0, T] \times \Omega \to \R^d$ be a flow for $b$, i.e. $X(\cdot, x)$
	is an absolutely continuous solution to (\ref{eq:ode}) with $x^0 = x$ for every $x \in \Omega$.
	Then $X(t) \in C^1(\Omega; \R^d)$ for every $t \in [0, T]$. Furthermore, $DX(\cdot, x)$ is
	absolutely continuous on $[0, T]$ for every $x \in \Omega$ and
	\begin{align}
		\del_t DX(t, x) = Db(t, X(t, x)) DX(t, x) \,.
	\end{align}
	\label{prop:smooth-dependence}
\end{appendixproposition}

\begin{proof}
	For every $x \in \Omega$, let $J(\cdot, x) \in W^{1,1}((0, T); \R^{d \times d})$ be a
	solution of the linear differential equation
	\begin{align}
		\del_t J(t, x) = Db(t, X(t, x)) J(t, x), \qquad J(0, x) = \id_{\R^d} \,.
		\label{eq:J-linear-ode}
	\end{align}
	For the existence of such a solution, note that for fixed $t$ and $x$,
	\begin{align*}
		Db(t, X(t, x)) : \R^{d \times d} \to \R^{d \times d}, \qquad J \mapsto Db(t, X(t, x)) J
	\end{align*}
	is Lipschitz continuous with Lipschitz constant $\norm*{Db(t, X(t, x))}_{\R^d \to \R^d}$.
	Here $\norm*{\cdot}_{\R^d \to \R^d}$ is the matrix norm induced by the euclidean norm on $\R^d$.
	Since $X([0, T], x) \subset \R^d$ is compact, $\norm*{Db(t, X(t, x))}_{\R^d \to \R^d} \in L^1([0, T])$.
	Thus the existence and uniqueness of a
	solution to (\ref{eq:J-linear-ode}) follow from the Cauchy-Lipschitz theorem. Since $Db(t, X(t, x))$
	depends continuously on $x$, also $J(t, x)$ is continuous in $x$.

	Let $x \in \Omega$ and let $\eps > 0$ such that $\overline{B_\eps(x)} \subset \Omega$.
	For $h \in B_\eps(0)$, we define
	\begin{align*}
		r_h(t, x) \colonequals X(t, x + h) - X(t, x) - J(t, x)h \,.
	\end{align*}
	Our goal is to show that $\lim_{h \to 0} r_h(t, x)/|h| = 0$ for all $t\in [0, T]$.
	By (\ref{eq:J-linear-ode}) and since $\del_t X(t, x) = b(t, X(t, x))$,
	\begin{align}
		\frac d {dt} |r_h(t, x)| \le \big|b(t, X(t, x + h)) - b(t, X(t, x)) - Db(t, X(t, x))J(t, x)h\big| \,.
		\label{eq:rh-derivative}
	\end{align}
	Writing $k = X(t, x + h) - X(t, x)$, we have
	\begin{align*}
		b(t, X(t, x) + k) - b(t, X(t, x)) = \int_0^1 Db(t, X(t, x) + sk) k \,ds \,.
	\end{align*}
	Thus
	\begin{align}
		\begin{split}
			|b(t, X(t, x) + k)
			 & - b(t, X(t, x)) -Db(t, X(t, x))k|                                                        \\
			\hspace{0.1em}
			 & \quad \le |k| \int_0^1 \norm*{Db(t, X(t, x) + sk) - Db(t, X(t, x))}_{\R^d \to \R^d} \,ds \\
			 & \quad \le |k| \, A_t(X(t, x), |k|) \,,
		\end{split}
		\label{eq:b+k}
	\end{align}
	where
	\begin{align*}
		A_t(z, \lambda) \colonequals \sup_{y \in B_\lambda(z)} \norm*{Db(t, y) - Db(t, z)}_{\R^d \to \R^d} \,.
	\end{align*}
	Define $K \colonequals X([0, T] \times \overline{B_\eps(x)}) \subset \R^d$.
	Since $X$ is continuous, $K$ is compact.
	By combining (\ref{eq:b+k}) with (\ref{eq:rh-derivative}),
	\begin{align*}
		\frac d{dt} |r_h(t, x)|
		 & \le \big|Db(t, X(t, x))(k - J(t, x)h)\big| + |k| \, A_t(X(t, x), |k|) \\
		 & \le \norm*{Db(t, X(t, x))}_{\R^d \to \R^d}|r_h(t, x)|
		+ C|h| \, A_t(X(t, x), C|h|)
	\end{align*}
	Here $C = \sup_{t \in [0, T]} \Lip(X(t)|_{B_\eps(x)})$, which is finite by Remark \ref{remark:local-flows}.
	Thus, by Grönwall's inequality and $r_h(0, x) = 0$
	\begin{align*}
		|r_h(t, x)| \le \exp\paren*{\norm*{Db}_{L^1(0, T; C_b(K))}}
		\, C |h| \int_0^t A_s(X(s, x), C|h|) \,ds \,.
	\end{align*}
	Since $A_s(X(s, x), C|h|) \le 2\norm*{Db(s)}_{C_b(K')}$, where $K' \colonequals K + B_{C|h|}(0)$,
	by the dominated convergence theorem
	\begin{align*}
		\lim_{h \to 0} \int_0^t A_s(X(s, x), C|h|) \,ds = 0 \,.
	\end{align*}
	Thus $\lim_{h \to 0} r_h/|h| = 0$ and $X$ is differentiable in $x$ with
	derivative $DX(t, x) = J(t, x)$.
\end{proof}

\begin{appendixproposition}
	Let $\Omega \subset \R^d$ be open and bounded, with Lipschitz boundary $\del \Omega$.
	If $u \in BV(\Omega)$, then $u \in W^{s, 1}(\Omega)$
	for every $0 < s < 1$ and
	$\norm*{u}_{W^{s, 1}(\Omega)} \le C_{\Omega, s} \big(\norm*{u}_{BV(\Omega)} + \norm*{u}_{L^1(\Omega)}\big)$.
	\label{prop:BV-Ws1}
\end{appendixproposition}

\begin{proof}
	We extend $u$ by 0 to $\R^d$ and denote the extension also by $u$.
	By \cite{EG15}, Theorem 5.6 (p.~204) and Theorem 5.8 (p.~210),
	\begin{align*}
		\|u\|_{BV(\R^d)} \le C_\Omega \|u\|_{BV(\Omega)} \,.
	\end{align*}
	Therefore, it suffices to show that
	\begin{align*}
		\|u\|_{W^{s, 1}(\R^d)} \le C_{d, s} \, \paren*{
			\|u\|_{BV(\R^d)} + \|u\|_{L^1(\R^d)}
		} \,.
	\end{align*}
	First, we change variables $w = z + h$ in the definition of
	$[u]_{W^{s, 1}(\R^d)}$.
	\begin{align*}
		[u]_{W^{s, 1}(\R^d)}
		 & = \int_{\R^d} \int_{\R^d} \frac{|u(z) - u(w)|}{|z - w|^{d + s}} \,dw \,dz
		= \int_{\R^d} \int_{\R^d} \frac{|u(z) - u(z + h)|}{|h|^{d + s}} \,dh \,dz
	\end{align*}
	Next, we split the integral over $h$ for $|h| < 1$ and $|h| \ge 1$.
	\begin{align*}
		[u]_{W^{s, 1}(\R^d)}
		\le \int_{B_1} \frac 1 {|h|^{d + s - 1}} \int_{\R^d} \frac{|u(z) - u(z + h)|}{|h|} \,dz \,dh
		+ 2 \|u\|_{L^1(\R^d)} \int_{\R^d \setminus B_1} \frac 1 {|h|^{d + s}} \,dh
	\end{align*}
	Since $s > 0$, the integral of $1/|h|^{d + s}$ over $\R^d \setminus B_1$ is a constant only
	depending on $d$ and $s$. By
	applying Hölder's inequality to the integral over $B_1$, we get
	\begin{align*}
		[u]_{W^{s, 1}(\R^d)}
		 & \le \int_{B_1} \frac 1 {|h|^{d + s - 1}} \,dh \
		\sup_{h \in B_1} \paren*{
			\int_{\R^d} \frac{|u(z) - u(z + h)|}{|h|} \,dz
		}
		+ C_{s, d}\|u\|_{L^1(\R^d)}                                              \\
		\addlinespace
		 & \le \tilde C_{s, d} \|u\|_{BV(\R^d)} + C_{s, d} \|u\|_{L^1(\R^d)} \,.
	\end{align*}
	Here, we used that
	\begin{align}
		\|u - u(\cdot + h)\|_{L^1(\R^d)} \le |h| \|u\|_{BV(\R^d)} \,.
		\label{eq:BV-finite-difference}
	\end{align}
	The estimate (\ref{eq:BV-finite-difference}) follows from
	an approximation by smooth functions, see \cite{EG15}, p.~199.
	Furthermore, $s < 1$ implies that the integral of $1/|h|^{d + s - 1}$ over $B_1$ is finite.
\end{proof}

\newpage
\printbibliography

@article{DL89,
  author = {Ronald J. DiPerna and Pierre-Louis Lions},
  journal = {Inventiones mathematicae},
  number = {3},
  pages = {511--547},
  title = {Ordinary differential equations, transport theory and Sobolev spaces},
  volume = {98},
  year = {1989},
  doi = {10.1007/BF01393835}
}

@article{CL08,
  title = {Estimates and regularity results for the DiPerna-Lions flow},
  author = {Gianluca Crippa and Camillo de Lellis},
  pages = {15--46},
  volume = {2008},
  number = {616},
  journal = {Journal für die reine und angewandte Mathematik},
  doi = {10.1515/CRELLE.2008.016},
  year = {2008},
}

@article{Amb04,
    title = {Transport equation and Cauchy problem for BV vector fields},
    author = {Luigi Ambrosio},
    year = {2004},
    volume = {158},
    pages = {227--260},
    doi = {10.1007/s00222-004-0367-2},
    journal = {Inventiones mathematicae},
}

@article{AT17,
     title = {Lecture notes on the {DiPerna{\textendash}Lions} theory in abstract measure spaces},
     author = {Luigi Ambrosio and Dario Trevisan},
     journal = {Annales de la Facult\'e des sciences de Toulouse : Math\'ematiques},
     pages = {729--766},
     year = {2017},
     publisher = {Universit\'e Paul Sabatier, Toulouse},
     volume = {Ser. 6, 26},
     number = {4},
     doi = {10.5802/afst.1551},
}

@book{Eva10,
    title = {Partial Differential Equations},
    author = {Lawrence C. Evans},
    year = {2010},
    publisher = {American Mathematical Society},
    edition = {2},
}

@book{EG15,
    title = {Measure Theory and Fine Properties of Functions},
    author = {Lawrence C. Evans and Ronald F. Gariepy},
    year = {2015},
    publisher = {CRC Press},
    edition = {Revised edition}
}

@article{DPV12,
    title = {Hitchhiker's guide to the fractional Sobolev spaces},
    journal = {Bulletin des Sciences Mathématiques},
    volume = {136},
    number = {5},
    pages = {521--573},
    year = {2012},
    issn = {0007-4497},
    doi = {10.1016/j.bulsci.2011.12.004},
    author = {Eleonora {Di Nezza} and Giampiero Palatucci and Enrico Valdinoci}
}

@article{BCK26,
    title = {Sharp nonuniqueness in the transport equation with Sobolev velocity field},
    journal = {Journal of the European Mathematical Society},
    doi = {10.4171/JEMS/1756},
    author = {Elia Bruè and Maria Colombo and Anuj Kumar},
    year = {2026}
}

@book{Gal11,
    title = {An Introduction to the Mathematical Theory of the Navier-Stokes Equations},
    author = {Giovanni P. Galdi},
    year = {2011},
    edition = {2},
    doi = {10.1007/978-0-387-09620-9},
    publisher = {Springer New York, NY}
}

@inbook{AC08,
    title = {Existence, Uniqueness, Stability and Differentiability Properties of the Flow Associated to Weakly Differentiable Vector Fields},
    author = {Luigi Ambrosio and Gianluca Crippa},
    doi = {10.1007/978-3-540-76781-7_1},
    pages = {3--57},
    year = {2008},
    publisher = {Springer, Berlin, Heidelberg},
    series = {Lecture Notes of the Unione Matematica Italiana},
    booktitle = {Transport Equations and Multi-D Hyperbolic Conservation Laws},
    volume = {5}
}

@book{Ste70,
    title = {Singular Integrals and Differentiability Properties of Functions},
    author = {Elias M. Stein},
    publisher = {Princeton University Press},
    year = {1970}
}

@article{FML05,
    title = {Weak Solutions, Renormalized Solutions and Enstrophy Defects in 2D Turbulence},
    author = {Milton C. Lopes Filho and Anna L. Mazzucato and Helena J. Nussenzveig Lopes},
    journal = {Archive for Rational Mechanics and Analysis},
    year = {2005},
    volume = {179},
    pages = {353--387},
    doi = {10.1007/s00205-005-0390-5}
}

@article{CS15,
    title = {Renormalized Solutions of the 2D Euler Equations},
    author = {Gianluca Crippa and Stefano Spirito},
    year = {2015},
    volume = {339},
    pages = {191--198},
    journal = {Communications in Mathematical Physics},
    doi = {10.1007/s00220-015-2411-z}
}
[heading=bibintoc]

\end{document}